\documentclass[12pt,twoside,a4paper]{article}  
\usepackage{amssymb,amsfonts,amsmath,amsbsy,theorem,amscd,dina42e}
\usepackage{xcolor}
\usepackage[T1]{fontenc}
\newcommand{\dd}{\text{\it \dj\hspace{1pt}}}

\newcommand{\ol}[1]{\overline{#1}}

\numberwithin{equation}{section}
\newcommand{\C}{\ensuremath{\mathbb{C}}}
\newcommand{\R}{\ensuremath{\mathbb{R}}}

\newcommand{\Rn}{\ensuremath{\mathbb{R}^n}}

\newcommand{\N}{\ensuremath{\mathbb{N}}}
\newcommand{\Z}{\ensuremath{\mathbb{Z}}}

\newcommand{\F}{\ensuremath{\mathcal{F}}}

\newcommand{\SD}{\ensuremath{\mathcal{S}}}

\newcommand{\sd}{\, d}
\newcommand{\supp}{\operatorname{supp}}
\newcommand{\eps}{\ensuremath{\varepsilon}}
\newcommand{\weight}[1]{\langle #1\rangle}

\newcommand{\Div}{\operatorname{div}}
\newcommand{\OP}{\operatorname{OP}}

\newcommand{\T}{\ensuremath{\mathbb{T}}}

\newcommand{\habil}[1]{}
\newtheorem{thm}{Theorem}[section]

\newtheorem{corollary}[thm]{Corollary}
\newtheorem{lem}[thm]{Lemma}
\newtheorem{lemma}[thm]{Lemma}
\newtheorem{defn}[thm]{Definition}
\newtheorem{definition}[thm]{Definition}
\newtheorem{theorem}[thm]{Theorem}
\newtheorem{prop}[thm]{Proposition}
\newtheorem{proposition}[thm]{Proposition}

\newtheorem{claim*}{Claim}
\newtheorem{assumption}[thm]{Assumption}
\theorembodyfont{\rmfamily}
\newtheorem{rem}[thm]{Remark}
\newtheorem{remark}[thm]{Remark}

\newenvironment{proof*}[1]{{\bf Proof
#1:}}{\hspace*{\fill}\rule{1.2ex}{1.2ex}\\ }
\newenvironment{proof}{\noindent{\bf
Proof:\,}}{\hspace*{\fill}\rule{1.2ex}{1.2ex}\\ }

\newcommand{\wA}{\widetilde A}

\newcommand{\pr}{\operatorname{pr}}
\newcommand{\inj}{\operatorname{i}}
\newcommand{\vn}{{\vec n}}

\newcommand{\crp}{\overline{\mathbb R}_+}

\newcommand{\rn}{{\mathbb R}^n}
\newcommand{\rnp}{{\mathbb R}^n_+}

\newcommand{\crnp}{\overline{\mathbb R}^n_+}

\newcommand{\comega}{\overline\Omega }

\newcommand{\srplus}{\cal S(\overline{\mathbb R}_+)}

\newcommand{\ang}[1]{\langle {#1} \rangle}

\newcommand{\simto}{\overset\sim\rightarrow}

\newcommand{\rp}{ \mathbb R_+}

\newcommand{\tr}{\operatorname{tr}}

\newcommand{\Ami}{A_{\min}}
\newcommand{\Ama}{A_{\max}}

\begin{document}
\begin{titlepage}
\title{Extension Theory and  Kre\u\i{}n-type Resolvent
  Formulas for Nonsmooth Boundary Value
  Problems} 
\author{Helmut~Abels\footnote{Fakult\"at f\"ur Mathematik,  
Universit\"at Regensburg,  
93040 Regensburg, Germany, E-mail {\tt helmut.abels@mathematik.uni-regensburg.de}}, Gerd~Grubb\footnote{Department of Mathematical Sciences,
Copenhagen University,
Universitetsparken 5,
 DK-2100 Copenhagen, Denmark.
E-mail {\tt grubb@math.ku.dk} 
}, Ian~Geoffrey~Wood\footnote{School of Mathematics, Statistics \& Actuarial Science,
Cornwallis Building,
University of Kent,
Canterbury,
Kent CT2 7NF, UK.
E-mail {\tt i.wood@kent.ac.uk}  
}}
\end{titlepage}
\maketitle

\begin{abstract}{
The theory of selfadjoint extensions of symmetric operators, and more
generally the theory of extensions of dual pairs, was
implemented some years ago
for boundary value problems for elliptic operators on 
smooth bounded domains. Recently, the questions have been taken up
again for nonsmooth
domains, with results first on $C^{1,1}$-domains 
for symmetric or smooth second-order operators, 
and next on quasi-convex Lipschitz domains 
for the selfadjoint realizations of the Laplacian. In
the present work we show that pseudodifferential methods can be used
to obtain a full characterization, including Kre\u\i{}n resolvent formulas,
of the realizations of nonselfadjoint second-order operators on
$C^{\frac32+\varepsilon }$ domains; more precisely, we treat domains 
with $B^\frac32_{p,2}$-smoothness and operators with
$H^1_q$-coefficients, for suitable $p>2(n-1)$ and $q>n$. 
The advantage of the
pseudodifferential boundary operator calculus is that the
operators are represented by a principal part and a lower-order
remainder, leading to regularity results; in particular we analyze
resolvents, Poisson solution operators and Dirichlet-to-Neumann operators in
this way, also in Sobolev spaces of negative order. Some unbounded domains are allowed.} 
\end{abstract}


\section{Introduction}\label{intro}

The systematic theory of selfadjoint extensions of a symmetric operator in a
Hilbert space $H$, or more generally, adjoint pairs of extensions of a
given dual pair of operators in $H$, has its origin in fundamental
works of Kre\u\i{}n \cite{Krein47}, Vishik \cite{Vishik52} and Birman
\cite{Birman56}. There have been several lines of development since
then. For one thing, there are the early works of Grubb
\cite{Grubb68}--\cite{Grubb74} completing and extending the theories
and giving an implementation for results for  boundary value problems
for elliptic 
PDEs. Another line has been the development by, among others,
Kochubei \cite{K75}, 
Gorbachuk--Gorbachuk \cite{GG84}, Derkach--Malamud \cite{DM91},
Malamud--Mogilevskii \cite{MM02}, where the tendency
has been to
incorporate the problems into studies of relations (generalizing
operators), with applications to (operator valued) ODEs; keywords in
this development are boundary triples, Weyl-Titchmarsh
$m$-functions. More recently this has been applied to PDEs 
(e.g., Amrein-Pearson
\cite{AP04}, Behrndt and coauthors \cite{BL07, BL12, BLL13a, BLL13b, BR12}, Brown-Marletta-Naboko-Wood \cite{BMNW08}, Kopachevski{\u{\i}}-Kre\u{\i}n \cite{KK04}, Malamud \cite{Mal10}, Ryzhov \cite{R07}). Further references are given in Brown-Grubb-Wood
\cite{BGW09}, where a connection between the two lines of development
is worked out.

One of the interesting aims  is to establish
Kre\u\i{}n resolvent formulas, linking the resolvent of a general
operator with the resolvent of a fixed reference operator by
expressing the difference in terms of operators connected to boundary
conditions, encoding spectral information.

In the applications to elliptic PDEs, Kre\u\i{}n-type resolvent formulas
are by now well-established in the case of operators with smooth
coefficients on smooth domains, but there remain challenging questions
about the validity in nonsmooth cases, and their applications.
 
One difficulty in implementing the extension theory in nonsmooth cases
lies in the fact that one needs mapping properties of direct and
inverse operators  not only in the most usual Sobolev spaces, but
also in spaces of low order, even of negative order over the boundary.
Another difficulty is to arrive at a theory where ellipticity
considerations are still applicable, in the way that the
operators are defined from
principal symbols plus lower-order error terms. This is important for
regularity questions, as well as for questions of spectral estimates. 


Gesztesy and Mitrea have addressed the 
extension problem 
for the Laplacian on
Lipschitz domains, showing Kre\u\i{}n-type  resolvent formulas in
\cite{GM08,GM09,GM09b}  involving  Robin problems under the
hypothesis that the boundary is of H\"o{}lder class $C^{\frac32+\varepsilon}$. More recently,
they have described the selfadjoint realizations of the Laplacian in  \cite{GM10} (based on the
abstract theory of \cite{Grubb68}), under a
more general hypothesis of quasi-convexity, which includes convex
domains and necessitates nonstandard boundary value spaces.
Posilicano and Raimondi gave in \cite{PR09} an analysis of selfadjoint
realizations of
second-order problems on $C^{1,1}$-domains. Grubb treated
nonselfadjoint realizations on $C^{1,1}$-domains in \cite{G08}, including
Neumann-type boundary conditions \begin{equation}
\chi u=C\gamma_0u, \label{eq:Neucond}
\end{equation}
with $C$ a differential operator of order 1, where the other mentioned
works mainly treat cases \eqref{eq:Neucond} with $C$ of order $<1$ or
nonlocal. (\cite{G08} can
be considered as a pilot project for the present paper.)
It should also be mentioned that Behrndt and Micheler \cite{BM13} recently
have shown how a
parametrization of the selfadjoint realizations of the Laplace
operator on a Lipschitz domain can be obtained by use of the theory of
quasi-boundary triples due to Behrndt et al.\  (cf.\ e.g.\ \cite{BL12}).
Compared with our results less regularity of the boundary is needed in the analysis. But the results are restricted to the Laplacian, while in the following we work with general second order elliptic operators. Moreover, in order to deal with Lipschitz boundary, where the usual results on elliptic regularity might fail, Behrndt and Micheler work with suitable more abstractly defined function spaces. In the case that the boundary is of class $C^{\frac32+\eps}$ for some $\eps>0$, their function spaces coincide with the classical ones, which we use in the following, cf. \cite[Theorem 4.10]{BM13} and the discussion below. 

Our aim in this paper is
to set up a construction of general extensions and resolvents
that works in $L_2$ Sobolev spaces 
when the regularity of $\Omega$ is in a scale of
function spaces larger than $\bigcup_{\varepsilon >
  0}C^{\frac32+\varepsilon }$, the coefficients of the elliptic
operator $A$ in another larger scale, 
yet 
allowing the use
of  pseudodifferential calculi that can take ellipticity
of boundary conditions into account and give precise information on
the principal parts of the operators.
We here choose to
work with operators having coefficients in scales of Sobolev spaces
and their generalizations to Besov and Bessel-potential spaces, since
this allows rather precise multiplication properties, and convenient
trace mapping results; then H\"older
space  properties can be read off using the well-known embedding
theorems.  
The resulting hypothesis on $\partial\Omega $ is that it can be parametrized by functions in the  Besov
space $B^{\frac32}_{p,2}$ for some $p>2(n-1)$. We note that this assumption is equivalent to $\tfrac32 - \tfrac{n-1}p>1$, where $\tfrac32 - \tfrac{n-1}p$ is the regularity number of the Besov space $B^{\frac32}_{p,2}$, which is the scaling exponent of the highest order parts of the norms under dilations of functions. It is the most relevant number for Sobolev embeddings, estimates of nonlinearities and applications to nonlinear partial differential equations. We note that (locally) $B^{\frac32}_{p,2}$ is inbetween $C^{1+\tau}$ and $C^{\frac32+\eps}$ for $\tau=\frac12 -\tfrac{n-1}p>0$ and any $\eps>0$. But the regularity number of  $B^{\frac32}_{p,2}$ is the same as the one of $C^{1+\tau}$ and can be much smaller than $\frac32$.

The theory of pseudodifferential boundary value problems (originating
in Boutet de Monvel \cite{BoutetDeMonvel} and further developed e.g.\
in the book of Grubb
\cite{FunctionalCalculus}; introductory material is given in \cite{G09}) is well-established for operators with $C^\infty $-coefficients
on $C^\infty $ domains. 
It has been extended to nonsmooth
cases by Abels \cite{NonsmoothGreen}, along the lines  of the extension of
pseudodifferential operators on open sets in 
Kumano-Go and Nagase \cite{KumanoGoNagase}, Marshall \cite{MarschallSobolevCoeff}, Taylor
\cite{TaylorNonlinearPDE}, \cite{ToolsForPDE}. These results have been applied to studies of the Stokes operator in Abels \cite{HInftyInLayer} and Abels and Terasawa~\cite{BIPVariableViscosity}, which in particular imply optimal regularity results for the instationary Stokes system, cf.  Abels~\cite{SolonnikovProceedings}. For applications to quasi-linear differential equations and free boundary value problems non-smooth coefficients are essential, cf. e.g. Abels~\cite{FreeSurface} and Abels and Terasawa~\cite{AbelsTerasawa2}.  The present
paper builds 
on \cite{NonsmoothGreen} and ideas of \cite{BIPVariableViscosity} and develops additional material.

Our final results will be formulated for operators acting between $L_2$ 
Sobolev spaces, but along the
way we also need $L_p$-based variants with $p\ne 2$ for the operator-
and domain-coefficients.  Here the
integral exponent will be called $p$ when we describe the domain $\Omega $
and its boundary $\Sigma =\partial\Omega $, and $q$ when we describe
the given partial differential operators $A$ and boundary operators
and their rules of calculus. There is then an optimal choice of how to
link $p$ and $q$, together with the dimension and the smoothness
parameters of the spaces where the operators act; this is expressed in Assumption
\ref{tau-assumption2}. 

The results in the paper have been applied in Grubb \cite{G12} to show
spectral asymptotic estimates for the boundary term in Kre\u\i{}n
formulas established here.

We originally intended to include $2m$-order operators $A$ with $m>1$,
but the coefficients in  Green's formula needed an extra, lengthy
development of symbol classes that made us
postpone this to a future publication.
\medskip

\noindent{\it Plan of the paper.} In Section 2, we recall the facts on
Besov and Bessel-potential function spaces that we shall need, define
the domains with boundary in these smoothness classes, and establish
a useful diffeomorphism property. Nonsmooth pseudodifferential
operators are recalled, with mapping- and composition-properties, and
Green's formula for second-order nonsmooth elliptic operators $A$ on
appropriate nonsmooth domains is established.  The Appendix gives
further information on pseudodifferential boundary operators ($\psi$dbo's) with
nonsmooth coefficients, extending some results of \cite{NonsmoothGreen} to
$S^m_{1,\delta}$ classes. Section 3 recalls the
abstract extension theory of \cite{Grubb68}, \cite{Grubb74},
\cite{BGW09}. In Section 4 we use the $\psi$dbo calculus to construct 
the resolvent $(A_\gamma-\lambda)^{-1}$ and Poisson solution
operator $K^\lambda_\gamma$ for the Dirichlet problem in the nonsmooth
situation, by
localization and parameter-dependent estimates. The construction shows that the principal part of the resolvent belongs  to the class of non-smooth pseudodifferential boundary operators, which is essential for the subsequent analysis. Section 5 gives an
extension of Green's formula to low-order spaces, and provides an
analysis of $K^\lambda_\gamma$ and the associated Dirichlet-to-Neumann
operator $P^\lambda_{\gamma,\chi }=\chi K^\lambda_\gamma$, needed for the
interpretation of the abstract theory. In particular it is shown that the operators coincide with operators of the pseudodifferential calculi up to lower order operators, which is one of the central results of the paper. Finally, the interpretation  is worked out in
Section 6, leading to a full validity of the characterization of the
closed realizations of $A$ in terms of boundary conditions, and including
Kre\u\i{}n-type resolvent formulas for all closed realizations $\wA$.   
 Section 7 gives a special analysis of the Neumann-type boundary
 conditions \eqref{eq:Neucond} entering in the theory, showing in
 particular that regularity of solutions holds 
 when $C-P^0_{\gamma,\chi }$ is elliptic.

\newcommand{\xx}{\tilde{x}}

\section[Basics on function spaces]{Basics on function spaces and operators on non\-smooth
    domains}\label{prelim}
\subsection{Function spaces on non\-smooth domains}
\label{sec:Funct.sp}

For convenience we here recall the definitions and properties of
function spaces that will be used throughout this paper. Proofs
can be found e.g.\ in Triebel~\cite{Triebel1} and Bergh and L\"ofstr\"om~\cite{Interpolation}.  All spaces are
Banach spaces, some $L_2$-based spaces are also Hilbert spaces.

The usual multi-index notation for differential operators with
$\partial=\partial_x=(\partial_1,\dots,\partial_n)$,
$\partial_j=\partial_{x_j}=\partial/\partial x_j$, and 
$D=D_x=(D_1,\dots,D_n)$, $D_j=D_{x_j}=-i\partial/\partial x_j$, will
be employed. 

For the spaces defined over ${\R}^n$, the Fourier transform ${\F}$ is
used to define operators such as $p(D_x)u={\F}^{-1}(p(\xi ){\F}u)$ (also
called $\operatorname{Op}(p)u$), for suitable functions $p(\xi )$. In
particular, with $\weight{\xi}=(1+|\xi|^2)^{1/2}$, $\weight{D_x}^s$
stands for 
$(1-\Delta )^{ s/2}$. $\SD(\Rn)$ denotes the Schwartz space of smooth,
rapidly decreasing functions and  $ \SD'(\Rn)$ its dual space, the
space of tempered distributions.
\medskip

\noindent {\it Function spaces}. 
The Bessel potential space in $\Rn$ of order $s\in \R$ is defined for $1<p<\infty$ by
\begin{equation*}
  H_p^s(\Rn) = \left\{f\in \SD'(\Rn): \weight{D_x}^s  f \in L_p(\Rn)\right\},
\end{equation*}
normed by $\|f\|_{H_p^s(\Rn)}=\|\weight{D_x}^s {f}\|_{L_p(\Rn)}$.
For $s=m$, a non-negative integer, $H^m_p({\R}^n)$ equals the space
of $L_p({\R}^n)$-functions with derivatives up to order $m$ in
$L_p({\R}^n)$, also denoted $W^m_p({\R}^n)$. 
In the case $p=2$, we omit the lower index and simply write $H^s(\Rn)$
instead of $H_2^s(\Rn)$. We denote the sesquilinear duality pairing of
$u\in H^s(\Rn)$ with $v\in H^{-s}(\Rn)$ by $(u,v)_{s,-s}$ (linear in
$u$, conjugate linear in $v$).

To describe the regularity, both of domains and of operator-coefficients, we shall also need Besov spaces $B^s_{p,q}(\Rn)$, where $s\in \R, 1\leq p,q\leq \infty$. These are defined by
$ B^s_{p,q}(\Rn)= \left\{f\in \SD'(\Rn): \|f\|_{B^s_{p,q}(\Rn)}<\infty\right\}$,
where
\begin{eqnarray*}
  \|f\|_{B^s_{p,q}(\Rn)} &=& \left(\sum_{j=0}^\infty 2^{sjq} \|\varphi_j(D_x) f\|_{L_p(\Rn)}^q\right)^\frac1q\quad \text{if}\ q<\infty,\\
  \|f\|_{B^s_{p,\infty}(\Rn)} &=& \sup_{j\in\N_0}2^{sj} \|\varphi_j(D_x) f\|_{L_p(\Rn)}.
\end{eqnarray*}
Here,  
$\varphi_j$, $j\in \N_0$, is a partition of unity on $\Rn$ such that $\supp \varphi_0 \subseteq \{\xi\in \Rn: |\xi|\leq 2\}$ and $\supp \varphi_j\subseteq \{\xi\in \Rn: 2^{j-1}\leq |\xi|\leq 2^{j+1}\}$ if $j\in \N$, chosen such that $\varphi_j(\xi)= \varphi_1(2^{1-j}\xi)$ for all $j\in \N$, $\xi\in\Rn$.

The parameter $s$ indicates the smoothness of the functions. 
The second parameter $p$ is called the integration exponent. The third
parameter $q$ 
is called the summation exponent; it measures smoothness on a \emph{finer scale} than $s$, which can be seen by the following simple relations:
\begin{alignat}{2}\label{eq:BesovEmb1}
  B^s_{p,1}(\Rn) \hookrightarrow B^s_{p,q_1}(\Rn)&\hookrightarrow B^s_{p,q_2}(\Rn)\hookrightarrow B^s_{p,\infty}(\Rn)&\quad & \text{if}\ 1\leq q_1\leq q_2\leq \infty,\\\label{eq:BesovEmb2}
  B^{s+\eps}_{p,\infty}(\Rn) &\hookrightarrow B^s_{p,1}(\Rn),
\end{alignat}
where $s\in \R$, $\eps>0$, and $1\leq p\leq \infty$ are arbitrary.
(The sign $\hookrightarrow$ indicates continuous embedding.)  
The embeddings follow directly from the definition and the fact that 
$\ell_{q_1}(\N_0)\hookrightarrow \ell_{q_2}(\N_0)$ if 
$1\leq q_1\leq q_2\leq \infty$. Here, $\ell_{q}(\N_0)$ is the
space of sequences $(a_k)_{k\in \N_0}$
such that $\left(\sum_{k=0}^\infty |a_k|^q\right)^\frac1q <\infty$ in
the case $q<\infty$ and $\sup_{k\in \N_0} |a_k|<\infty$ if $q=\infty$,
provided with the hereby defined norm. 

We recall that for $p=q$ and $s\in \R_+\setminus \N$, $B^s_{p,p}(\Rn)$ equals
the Sobolev-Slobodetski\u\i {} space $W^s_p(\Rn)$, whereas for $s\in \N_0$,
it is $H^s_p(\Rn)$ that equals $W^s_p(\Rn)$. (In the following, the $H$- and
$B$-notation will be used for clarity; these scales of spaces have
the best interpolation properties.)
In the case $p=2$, all three
spaces coincide, for general $s$:
\begin{equation}\label{eq:HsBesovID}
  H^s(\Rn)=H^s_2(\Rn) = B^s_{2,2}(\Rn)=W^s_2(\Rn).
\end{equation}

The spaces $B^s_{\infty, \infty }(\Rn)$, also denoted $\mathcal C^s(\Rn)$
when $s>0$ (H\"older-Zygmund spaces), play a special role. For $s\in
\R_+\setminus \N$, $B^s_{\infty, \infty }(\Rn)=\mathcal C^s(\Rn)$ can be identified with the
H\"older space $C^{k,\sigma }(\Rn)$, defined for $k+\sigma =s$, $k\in\N_0$ and $\sigma \in
(0,1]$, and also denoted $C^{s}(\Rn)$ when $\sigma \in (0,1)$.  For $s=k\in\N$, there are sharp inclusions
\[ C^k_b(\Rn)\hookrightarrow C^{k-1,1}(\Rn)\hookrightarrow \mathcal C^k(\Rn);
\]
here, $C^k_b(\Rn)$ is the usual space of bounded continuous
functions with bounded continuous derivatives up to order $k$.

At this point, let us recall some interpolation results:
Denoting the real and complex interpolation functors by  $(.,.)_{\theta,q}$ and $(.,.)_{[\theta]}$, respectively, we have that
 if $s_0,s_1\in \R$ with $s_0\neq s_1$, $1\leq p,q_0,q_1,r\leq \infty$, and $s= (1-\theta) s_0+ \theta s_1$, $\theta \in (0,1)$, then
\begin{equation}\label{eq:BesovInterpol1}
  (B^{s_0}_{p,q_0}(\Rn), B^{s_1}_{p,q_1}(\Rn))_{\theta, r} = B^s_{p,r}(\Rn).
\end{equation}
If additionally $\frac1p= \frac{1-\theta}{p_0} + \frac{\theta}{p_1}$ for some $1\leq p_0,p_1\leq \infty$ and $\frac1q= \frac{1-\theta}{q_0} + \frac{\theta}{q_1}$, then
\begin{equation}\label{eq:BesovInterpol2}
  (B^{s_0}_{p_0,q_0}(\Rn), B^{s_1}_{p_1,q_1}(\Rn))_{[\theta]} = B^s_{p,q}(\Rn),
\end{equation}
cf. \cite[Theorem 6.4.5]{Interpolation} or \cite[Section~2.4.1 Theorem]{Triebel1}. 
Using the same notation, we have in particular for the Bessel potential spaces \begin{align}\label{eq:BesselInterpolReal}  
  (H^{s_0}_p(\R^n), H^{s_1}_p(\R^n))_{\theta,r} &= B^s_{p,r}(\R^n),
 \\ (H^{s_0}_{p_0}(\R^n), H^{s_1}_{p_1}(\R^n))_{[\theta]}& =
 H^s_p(\R^n)
\nonumber
\end{align}
(cf.~\cite[Theorem 6.4.5]{Interpolation}).

\medskip

\noindent {\it General embedding properties.} 
For any $1<p<\infty$,  we have the following embeddings between Besov
spaces and Bessel potential spaces:
\begin{align}
  B^{s+\varepsilon }_{p,q_1}(\Rn) &\hookrightarrow H_p^s(\Rn) \hookrightarrow
  B^{s-\varepsilon }_{p,q_2}(\Rn) \qquad \text{for all } \eps>0, 1\leq q_1,q_2\leq
  \infty, s\in\R,\nonumber\\
B^{s}_{p,\min(2,p)}(\Rn) &\hookrightarrow H_p^s(\Rn) \hookrightarrow
B^{s}_{p,\max(2,p)}(\Rn) \qquad \text{for all } s\in\R,\label{eq:BesovEmb4}
\end{align}
cf. e.g. \cite[Theorem 6.4.4]{Interpolation}.

There are the following Sobolev embeddings for Bessel potential spaces:
  \begin{alignat}{2}\label{eq:BesselEmb}
    H^{s_1}_{p_1}(\Rn)&\hookrightarrow
    H^{s_0}_{p_0}(\Rn)
\qquad \text{if}\ s_1\geq s_0, s_1-\tfrac{n}{p_1}\geq s_0-\tfrac{n}{p_0},
\\
     H^{s}_p(\Rn)&\hookrightarrow \mathcal{C}^{\alpha}(\Rn)\qquad \text{if}\ \alpha=s-\tfrac np>0, 
  \end{alignat}
provided that $1<p_1\leq p_0 <\infty$, $1<p<\infty$.
In particular, $H^{s}_{p}(\Rn)\hookrightarrow L_{p}(\Rn)$ for $s\ge 0$.

 For the Besov spaces a Sobolev-type embedding is given by 
  \begin{alignat}{2}\label{eq:BesovEmb3}
    B^{s_1}_{p_1,q}(\Rn) &\hookrightarrow B^{s_0}_{p_0,q}(\Rn) &\qquad& \text{if}\ s_1\geq s_0 \ \text{and}\ s_1-\tfrac{n}{p_1}\geq s_0-\tfrac{n}{p_0},
  \end{alignat}
  for any $1\leq q\leq\infty$.
In particular, combining this with \eqref{eq:BesovEmb1}, we get
\begin{alignat}{2}
B^{s_1}_{p_1,q}(\Rn) \hookrightarrow   B^{s_0}_{\infty, q}(\Rn) \hookrightarrow B^{s_0}_{\infty, \infty}(\Rn) = \mathcal{C}^{s_0}(\Rn)\label{Hoelderreg}
 \end{alignat}
whenever $s_0=s_1-\frac{n}{p_1}>0$. In the opposite direction, we have from \eqref{eq:BesovEmb1} and \eqref{eq:BesovEmb2}
\begin{equation}\label{eq:Zygmundem} \mathcal
  C^{\alpha 
  }(\Rn)
=  B^{\alpha 
}_{\infty ,\infty }(\Rn)
\hookrightarrow  B^{\alpha -\varepsilon
}_{\infty ,2}(\Rn)
\end{equation}
when $\alpha >0$, $0<\varepsilon <\alpha $. We also note that
\begin{equation}\label{eq:SobolevEmb}
  H^{s_1}_{p_1}(\Rn)\cup B^{s_1}_{p_1,p_1}(\Rn)\hookrightarrow  H^{s_0}_{p_0}(\Rn)\cap B^{s_0}_{p_0,p_0}(\Rn)
\end{equation}
if $1<p_1<p_0 <\infty$ and $s_1-\frac{n}{p_1}\geq s_0-\frac{n}{p_0}$; this can be found in \cite[Section~2.8.1, equation (17)]{Triebel1}.

\medskip

\noindent {\it Function spaces over subsets of $\Rn$.}
The Bessel potential and  Besov spaces are defined on a domain
$\Omega\subset\Rn$ with $C^{0,1}$-boundary (see Definition \ref{bdryreg} below) simply by restriction:
\begin{eqnarray}\label{restr.spaces}
  H^s_p(\Omega) &=& \{f\in \mathcal{D}'(\Omega): f= f'|_{\Omega}, f'\in H^s_p(\Rn) \},\\
  B^s_{p,q}(\Omega) &=& \{f\in \mathcal{D}'(\Omega): f= f'|_{\Omega}, f'\in B^s_{p,q}(\Rn) \},\nonumber
\end{eqnarray}
for $s\in\R$  
and $1\leq p,q\leq \infty$.
Here $f'|_{\Omega}\in \mathcal{D}'(\Omega)$ is defined by
$\weight{f'|_{\Omega},\varphi}_{\mathcal{D}'(\Omega),
  \mathcal{D}(\Omega)}=
\weight{f,\varphi}_{\mathcal{D}'(\Rn),\mathcal{D}(\Rn)}$ for all
$\varphi \in C_0^\infty(\Omega)$, embedded in $  C_0^\infty(\Rn)$ by
extension by zero. 
The spaces are equipped with the quotient norms, e.g.,
\begin{eqnarray}\label{eq:QuotientNorm}
  \|f\|_{B^s_{p,q}(\Omega)} &=& \inf_{f'\in B^s_{p,q}(\Rn): f'|_{\Omega}= f}\|f'\|_{B^s_{p,q}(\Rn)}.
\end{eqnarray}

In particular, $H^m_p(\Omega )$ is for  $m\in {\N}_0$ and $1<p<\infty$ equal to the usual Sobolev space $W^m_p(\Omega)$ of  
$L_p(\Omega )$-functions with derivatives up to order $m$ in
$L_p(\Omega )$. We recall that
there is an extension operator $E_\Omega$ which is a bounded linear operator
$E_\Omega\colon W^m_p(\Omega)\to W^m_p(\R^n)$,
for all $m\in \N_0$,  $1\leq p\leq \infty$, and satisfies $E_\Omega f|_{\Omega} = f$ for all $f\in
W^m_p(\Omega)$.
This holds when $\Omega$ is merely a Lipschitz domain,
cf.\ e.g.\ Stein~\cite[Chapter VI, Section 3.2]{Stein:SingInt} and trivially carries over to $H^m_p(\Omega)$ for $1< p< \infty$. Moreover, in view of the fact that $H^s_p(\Omega)$ is a retract of $H^s_p(\Rn)$, one has that all interpolation and Sobolev embedding results for $H^s_p(\Rn)$ are inherited by the spaces on $\Omega$. 

We shall also need the spaces
\begin{equation*}
  H^s_0(\comega)=\{u\in H^s(\Rn): \operatorname{supp}u\subset \comega\}.
\end{equation*}
Here, $H^s_0(\comega)$ identifies in a natural
way with the dual space of $H^{-s}(\Omega )$, for all
$s\in\R$, cf.\ \cite[Theorem~3.30]{McLean}. 
For $s$ integer $\ge 0$, $H^s_0(\comega)$ equals the closure of
$C_0^\infty(\Omega)$ in $H^s(\Omega)$ and is usually denoted
$H^s_0(\Omega )$ (see also
\cite[Theorem~3.33]{McLean}). 

\medskip

\noindent {\it Traces.}
Next, let us recall the well-known trace theorems:
The trace map $\gamma_0 $  from $\Rn_+$ to $\R^{n-1}$, defined on
smooth functions with bounded support,  extends by
continuity to continuous maps
 for $s>\frac1p$, $1<p<\infty$, $1\le q\le \infty $, 
  \begin{eqnarray*}
    \gamma_0 \colon  H^s_p(\Rn_+)& \to & B_{p,p}^{s-\frac1p}(\R^{n-1}),\\
    \gamma_0 \colon  B^s_{p,q}(\Rn_+)& \to & B_{p,q}^{s-\frac1p}(\R^{n-1}).
  \end{eqnarray*}  
All of these maps are surjective and have continuous right inverses.

\medskip

\noindent{\it Vector-valued Besov and Bessel potential spaces.}
In the following let $X$ be a Banach space. Then
$L_p(\R^n;X)$, $1\leq p< \infty$, is defined as the space of strongly measurable functions
$f\colon \R^n\to X$ with 
\begin{equation*}
  \|f\|_{L_p(\R^n;X)}:= \left(\int_{\Rn} \|f(x)\|_X^p
  dx\right)^{\frac1p}<\infty
\end{equation*}
and $L_\infty(\R^n;X)$ is the space of all strongly measurable and
essentially bounded functions. 
Similarly, $\ell_p(\N_0;X)$, $1\leq p\leq \infty$, denotes the $X$-valued variant of $\ell_p(\N_0)$.
 
Furthermore, let $\SD(\Rn;X)$ be the space  of smooth rapidly decreasing
functions $f\colon\Rn\to X$ 
and let $\SD'(\Rn;X):= \mathcal{L}(\SD(\Rn),X)$ denote the space of
tempered $X$-valued distributions. 
Then the $X$-valued
variants of the Bessel potential and Besov spaces of order $s\in\R$ are defined as
\begin{alignat*}{1}
  H^s_p(\Rn;X)&:=\{f\in\SD'(\Rn;X): \weight{D_x}^s f \in L_p(\Rn;X)\} \quad
  \text{if} \ 1<p<\infty,\\ 
  B^s_{p,q}(\Rn;X)&:=\{f\in\SD'(\Rn;X): (2^{sj}\varphi_j(D_x) f)_{j\in\N_0} \in \ell_q(\N_0;L_p(\Rn;X))\},
\end{alignat*} 
where $1\leq p,q\leq \infty$. 
Here, $p(D_x)f\in \mathcal{S}'(\Rn;X)$ is defined by
\begin{equation*}
  \weight{p(D_x)f,\overline\varphi}= \weight{f,\overline{\overline{p}(D_x)\varphi}}\qquad \text{for all}\ \varphi \in \SD(\Rn).
\end{equation*}
We will also make use of the Banach space
\begin{alignat*}{1}
\operatorname{\it BUC}(\R^n;X)=\{ f\in L_\infty(\Rn;X): f \hbox{ is uniformly continuous}\}
\end{alignat*} 
with the supremum norm.

The properties of the function spaces discussed above for the scalar case, carry over to the vector-valued case. For details, we refer to e.g.~\cite{ABHN01,HP57} (for the Bochner integral and its properties) and to \cite{Amann} (for vector-valued function spaces).

In the following we will use some special anisotropic Sobolev spaces.

\begin{definition}\label{def:W2p} Let $I=(0,\infty)$ or $\R$,  and let
  $\Omega=\R^{n-1}\times I$, with coordinates $(x',x_n)$. For $k\in \N_0$, $1\leq p<\infty$, set 
$$
W^{k}_{(2,p)}(\Omega)=\left\{f\in L_2(I;L_p(\R^{n-1})):\partial_x^\alpha f \in L_2(I;L_p(\R^{n-1})), |\alpha|\leq k\right\}.
$$ 
\end{definition}

\begin{lem}\label{lem:W2p}
One has for $k\ge 1$ that
\begin{eqnarray}
 W^{k}_{(2,p)}(\Omega)
\label{eq:BUCEmbedd}
 &\hookrightarrow & \operatorname{\it BUC}(I;B^{k-\frac12}_{p,2}(\R^{n-1})).
\end{eqnarray}
Here, the trace mapping $u\mapsto u|_{x_n=0}$ is surjective from $
W^{k}_{(2,p)}(\Omega)$ to $B^{k-\frac12}_{p,2}(\R^{n-1})$. Namely, when
$g(x')\in B^{k-\frac12}_{p,2}(\R^{n-1})$, then
$G(x',x_n)=(e^{-Ax_n}g)(x')$ is in $
W^{k}_{(2,p)}(\Rn_+)$ with $G(x',0)=g(x')$ (where $e^{-Ax_n}$ is the
semigroup generated by $-A=-\ang{D_{x'}}$). $G(x',x_n)$ extends to a function 
$G\in W^{k}_{(2,p)}(\Rn)$.
\end{lem}
\begin{proof}
First of all,
\begin{eqnarray*}
 W^{k}_{(2,p)}(\Omega)&\hookrightarrow& L_2(I;H^{k}_p(\R^{n-1}))\cap H^{k}(I;L_p(\R^{n-1})).
\end{eqnarray*}
To obtain
\begin{equation*}
  L_2(I;H^{k}_p(\R^{n-1}))\cap H^{k}(I;L_p(\R^{n-1}))
 \hookrightarrow  \operatorname{\it BUC}(I;B^{k-\frac12}_{p,2}(\R^{n-1}))
\end{equation*}
 one can apply 
\cite[Corollary~3.12.3]{Interpolation}
with $\eta_j=\frac12, p_j=2$, $j=0,1$, a result from Lions' trace
method of real interpolation, to obtain
\begin{equation}
u|_{x_n=0}\in (L_p(\R^{n-1}),H^{k}_p(\R^{n-1}))_{1-\frac1{2k},2}=B^{k-\frac12}_{p,2}(\R^{n-1}),\label{traceW2p}
\end{equation}
 for every 
$u\in W^{k}_{(2,p)}(\Omega)$; the identity follows from \eqref{eq:BesselInterpolReal}. 
Next, this is combined with the strong continuity
of the translations $(\tau_h u)(x)=(x',x_n+h)$, $h\geq 0$, 
in  $L_2(I;H^{k}_p(\R^{n-1}))\cap H^{k}(I;L_p(\R^{n-1}))$ as in the proof of \cite[Chapter III, Theorem 4.10.2]{Amann}.

For the last assertion, let $A^s=\weight{D_{x'}}^{s}$, $s\in \R$; here $A=A^1$. Then 
$$
H^k_p(\R^{n-1})=\{f\in L_p(\R^{n-1}):\weight{D_{x'}}^kf\in L_p(\R^{n-1}) \}=: {D}(A^k)
$$ 
for all $k\in \N_0$, $1<p<\infty$, as explained in
\cite[Theorem~6.2.3]{Interpolation}. Now when $g$ is given,  let $G(x',x_n)=
(e^{-Ax_n}g)(x')$ for $x_n\ge 0$. Since 
\begin{equation*}
  W^{k}_{(2,p)}(\Rn_+)=\bigcap_{0\leq j\leq k}
  H^j(0,\infty;H^{k-j}_p(\R^{n-1})),
\end{equation*}
we have by \cite[Corollary~3.5.6, Theorem~3.4.2]{ButzerBerens} 
that $A^kG\in L^2(0,\infty;L_p(\R^{n-1}))$, so 
$
G\in W^{k}_{(2,p)}(\Rn_+)$. Here, we use that
$
B^{k-\frac12}_{p,2}(\R^{n-1})=(L_p(\R^{n-1}),H^{k}_p(\R^{n-1}))_{1-\frac1{2k},2}
$, as noted above in \eqref{traceW2p}. $G(x',x_n)$ is extended
to a function in $W^{k}_{(2,p)}(\Rn) $ by a standard 
``reflection'' in $x_n=0$, as explained e.g.\ in \cite[Th.\ I 2.2]{LM68}.
\end{proof}

By use of this lemma we derive the following product estimate,
which is essential for the low boundary regularity that we shall allow:
\begin{lem}\label{lem:Product}
  Let $I,\Omega$ and $W^k_{(2,p)}(\Omega)$ be as in Definition {\rm \ref{def:W2p}}. Then for every $k\in \N$ and $2\leq p< \infty$ such that $k-\frac12 - \frac{n-1}p >0$ there is some $C_{k,p}>0$ such that
  \begin{equation*}
    \|fg\|_{H^k(\Omega)}\leq C_{k,p} \|f\|_{H^k(\Omega)}  \|g\|_{W^k_{(2,p)}(\Omega)},
  \end{equation*}
  for all $f\in H^k(\Omega)$, $g\in W^k_{(2,p)}(\Omega)$. Moreover, if $k=1$ and $\tau=\frac12 - \frac{n-1}p$, then
  \begin{equation}\label{eq:CommEstim}
\|f\partial_{x_j}g\|_{L_2(\Omega)}\leq C_{p} \|f\|_{H^{1-\tau}(\Omega)}  \|g\|_{W^1_{(2,p)}(\Omega)}
  \end{equation}
uniformly with respect to $f\in H^1(\Omega), g\in W^1_{(2,p)}(\Omega)$.
\end{lem}

\begin{proof}
  First of all
  \begin{equation*}
W^k_{(2,p)}(\Omega) \hookrightarrow \operatorname{\it BUC}(I;B^{k-\frac12}_{p,2}(\R^{n-1}))\hookrightarrow  L_\infty(\Omega)    
  \end{equation*}
by (\ref{eq:BUCEmbedd}); the second embedding follows from
(\ref{Hoelderreg})  since $k-\frac12 - \frac{n-1}p>0$.
Furthermore,
\begin{equation*}
H^k(\Omega) \hookrightarrow
\operatorname{\it BUC}(I;H^{k-\frac12}(\R^{n-1}))\hookrightarrow
\operatorname{\it BUC}(I;L_r(\R^{n-1})), \hbox{ for } \tfrac1r = \tfrac12-\tfrac1p.  
\end{equation*}
Here we apply (\ref{eq:BUCEmbedd}) with $p=2$ for the first embedding,
and for
the second embedding we use (\ref{eq:BesselEmb}),
noting that $k-\frac12-\frac{n-1}p>0$ is equivalent to $k-\frac12 -\frac{n-1}2> - \frac{n-1}r$.
Next we observe that for all $|\alpha|\leq k$
\begin{equation}\label{eq:ProductRule}
  \partial_x^\alpha (fg)= \sum_{0\leq \beta \leq \alpha,\beta \neq\alpha}\binom{\alpha}{\beta}\partial_x^{\alpha-\beta}f \partial_x^\beta g + f\partial^{\alpha}_x g. 
\end{equation}
Since $f\in L_\infty(I;L_r)$ and $\partial_x^\alpha g \in L_2(I;L_p)$,
where $\frac1r+ \frac1p=\frac12$, we have $f\partial_x^\alpha g\in
L_2(\Omega)$. For the other terms where $ |\beta|<|\alpha |\le k$,  we note that  
\begin{equation}\label{eq:Terms}
  \partial_x^{\alpha -\beta }f \in L_2(I;H^{k-|\alpha |+|\beta|}(\R^{n-1})), \quad \partial_x^{\beta}g \in L_\infty(I; B^{k-\frac12-|\beta|}_{p,2}(\R^{n-1})).
\end{equation}
One has in general
\begin{equation}\label{eq:ProdEst}
  \|u v\|_{L_2(\R^{n-1})} \leq C_{M,M'}
  \|u\|_{H^M(\R^{n-1})}\|v\|_{B^{M'}_{p,2}(\R^{n-1})}, 
\end{equation}
provided that $M,M'\in \N_0$ and $M+M'- \frac{n-1}p  >0$. This
estimate easily follows from the Sobolev-type embedding theorems: If
$M'-\frac{n-1}p>\eps>0$, then from \eqref{Hoelderreg}  we have $B^{M'}_{p,2}(\R^{n-1})\hookrightarrow \mathcal{C}^{\eps}(\R^{n-1})\hookrightarrow
L_\infty(\R^{n-1})$ 
and the statement is trivial. If $M'-\frac{n-1}p<0$, then \eqref{eq:SobolevEmb} implies $B^{M'}_{p,2}(\R^{n-1})\hookrightarrow L_r(\R^{n-1})$ with $\frac1r=\frac1p - \frac{M'}{n-1}$, and $M+M'- \frac{n-1}p  >0$ implies by \eqref{eq:BesselEmb} that $H^M(\R^{n-1})\hookrightarrow L_{\tilde{r}}(\R^{n-1})$ with $\frac12= \frac1r+\frac1{\tilde{r}}$. If $M'-\frac{n-1}p=0$, one can choose some $\tilde{p}<p$ such that $M+M'- \frac{n-1}{\tilde{p}}  >0$ and apply the preceding case. The estimate is also a consequence of Hanouzet~\cite[Th\'eor\`eme 3]{HanouzetBesovProd}. 

Using \eqref{eq:ProdEst} for products of functions as in
\eqref{eq:Terms}, we obtain
altogether that $\partial_x^{\alpha-\beta}f \partial_x^\beta g \in L_2(\R^{n-1})$ for all $0\leq \beta \leq \alpha,|\alpha|\leq k$.

Finally, if $k=1$, then we have that
\begin{equation*}
H^{1-\tau}(\Omega) \hookrightarrow
\operatorname{\it BUC}(I;H^{\frac12-\tau}(\R^{n-1}))\hookrightarrow
\operatorname{\it BUC}(I;L_r(\R^{n-1})), \;\tfrac1r = \tfrac12-\tfrac1p.  
\end{equation*}
Therefore
\begin{equation*}
  \|f\partial_{x_j} g\|_{L_2(\Omega)}\leq
  \|f\|_{L_\infty(I;L_r)}\|\partial_{x_j} g\|_{L_2(I,L_p)}\leq
  C\|f\|_{H^{1-\tau}(\Omega)}\|g\|_{W^1_{(2,p)}(\Omega )},
\end{equation*}
which proves the last statement.
\end{proof}

\medskip

\noindent{\it Domains with nonsmooth boundary.}
For the following, let $n\ge 2$, let  $M$ be a positive integer, and let $1\leq
p,q\leq\infty$ be such
that  $M-\frac32- \frac{n-1}p>0$. 

\begin{definition}\label{bdryreg}
Let $\Omega $ be an open subset of $\Rn$. We say that $\Omega$ has
a boundary of class $B^{M-\frac12}_{p,q}$ in the following three cases:

$1^\circ$ $\Omega =\Rn_\gamma $, where 
\[ \Rn_\gamma= \{x\in \Rn: x_n >\gamma(x')\}
\]
for a function $\gamma\in B^{M-\frac12}_{p,q}(\R^{n-1})$.

$2^\circ$ $\partial\Omega $ is compact, and each $x\in\partial\Omega $
has an open neighborhood $U$ satisfying: For a suitable choice of
coordinates on $\Rn$, there is a function $\gamma (x')\in
B^{M-\frac12}_{p,q}(\R^{n-1})$ such that $U\cap\Omega = U\cap \Rn_\gamma $ and
$U\cap\partial\Omega = U\cap \partial\Rn_\gamma $.

$3^\circ$ For a large ball $B_R=\{x\in\rn: |x|<R\}$, $\Omega \setminus
B_R$ equals $\Rn_+\setminus B_R$. The points $x\in B_{R+1}\cap
\partial\Omega $ have the property described in $2^\circ$.

There are similar definitions with other
function spaces. 
\end{definition}

In the second case, one can cover $\partial\Omega $ by a finite set of
such coordinate neighborhoods $U$. Note that exterior domains are
allowed. The third case is included in order to
show a simple case with noncompact
boundary where a finite system of coordinate neighborhoods suffices (namely finitely
many $U$'s covering $\partial\Omega \cap B_{R+1}$ and a trivial one covering $\partial\Omega \setminus
B_R$), to
describe the smoothness structure. More general such cases can be defined as
the ``admissible manifolds'' in \cite{FunctionalCalculus}.

We shall work under the following general hypothesis:

\begin{assumption}\label{tau-assumption} $n\ge 2$, $M\in \N$, $2\le p <
  \infty $,
with 
\begin{equation}
\tau:=M-\tfrac32-\tfrac{n-1}{p}>0.\label{asspt}
\end{equation}
Moreover, $\Omega $ is an open subset of $\Rn$ with
boundary $\partial\Omega$ of regularity $B^{M-\frac12}_{p,2}$, as in
 Definition {\rm \ref{bdryreg}}. 
\end{assumption}

\begin{rem}\label{rem:holder} Under Assumption \ref{tau-assumption},
 it follows from
\eqref{Hoelderreg} that $\partial\Omega $ is H\"{o}lder continuous with
exponent $1+\tau$ (if $\tau\not\in\N$). In the converse direction, 
if $\partial\Omega$ is  H\"{o}lder continuous with exponent
$M-\frac12+\eps$ for some $\eps>0$, then in view of \eqref{eq:Zygmundem},
 $\partial\Omega\in
B^{M-\frac12}_{\infty ,2}$. This in
turn implies $\partial\Omega\in B^{M-\frac12}_{p, 2}$ for every $1\leq
p\leq\infty$ if $\Omega$ is of type $2^\circ$ or $3^\circ$ in
Definition~\ref{bdryreg}, since $L^\infty(U)\hookrightarrow L^p(U)$ for
every $1\leq p\leq \infty$ when $U$ is bounded. In other words,
\[ \partial\Omega \in C^{M-\frac12+\varepsilon }\implies 
\partial\Omega \in B^{M-\frac12}_{p, 2}\implies \partial\Omega \in
C^{M-\frac12-\frac{n-1}p }=C^{1+\tau }. 
\]
 if $\Omega$ is of type $2^\circ$ or $3^\circ$ and $\tau\not\in\N$.
\end{rem}

When $U$ and $V$ are subsets of $\rn$, and $F\colon V\to U$ is a
bijection, we denote the pull-back mapping by $F^*$:
\[ (F^*u)(x)=u(F(x))\text{ for }x\in V,\quad (F^{-1,*}v)(y)=v(F^{-1}(y))\text{ for }y\in U,
\]
when $u$ is a function on $U$, $v$ is a function on $V$. The gradient
$\nabla u=(\partial _ju)_{j=1}^n$ is viewed as a column vector.

\begin{proposition}\label{prop:CoordinateTrafo}
Under Assumption {\rm \ref{tau-assumption}},  let $\gamma\in B^{M-\frac12}_{p,2}(\R^{n-1})$,  and let $\Rn_\gamma= \{x\in \Rn: x_n >\gamma(x')\}$. Then
  there is a $C^1$-diffeomorphism $F_\gamma\colon \Rn \to \Rn$ with
  $\nabla F_\gamma \in C^{\tau'}(\Rn)^{n^2}$ for $\tau'\le \tau$,
  $\tau '\in (0,\frac12)$ (cf.\ {\rm \eqref{asspt}}),
  such that $F_\gamma (\Rn_+)= \Rn_\gamma$ and  $F_\gamma^\ast \colon
  H^{s}(\Rn)\to H^{s}(\Rn)$ as well as $F_\gamma^\ast \colon
  H^{s}(\Rn_\gamma)\to H^{s}(\Rn_+)$ for all $0\leq s\leq M$.
\end{proposition}
\begin{proof}
We begin by defining $\Gamma(x',x_n)$ as the lifting of $\gamma (x')$
by the construction described in the last
statement in Lemma \ref{lem:W2p}; then $\Gamma\in W^M_{(2,p)}(\Rn)$.
In particular, this implies that
\begin{eqnarray*}
 \lefteqn{\nabla \Gamma \in L_2(\R;H^{M-1}_p(\R^{n-1}))^n\cap H^{M-1}(\R;L_p(\R^{n-1}))^n}\\
 && \hookrightarrow \operatorname{\it
   BUC}(\R;B^{M-\frac32}_{p,2}(\R^{n-1}))^n \hookrightarrow
  BUC(\R;C^0_b(\R^{n-1}))^n
\end{eqnarray*}
in view of (\ref{eq:BUCEmbedd}); we here use that
$B^{M-\frac32}_{p,2}(\R^{n-1})\hookrightarrow
\mathcal{C}^\tau(\R^{n-1})\hookrightarrow C^0_b(\R^{n-1})$ since $\tau=M-\frac32 -\frac{n-1}p>0$. Hence $\Gamma \in
C^1_b(\Rn)$. Moreover, we have from \eqref{eq:BesovEmb3}
\begin{equation*}
  \nabla \Gamma \in H^{1}(\R;H^{M-2}_p(\R^{n-1}))^n\hookrightarrow C^{\frac12}(\R;H^{M-2}_p(\R^{n-1}))^n.
\end{equation*}
Due to (\ref{eq:BesselEmb}), we have $H^{M-2}_p(\R^{n-1})\hookrightarrow B^{M-2}_{p,p}(\R^{n-1})$ and using 
 $\nabla \Gamma\in \operatorname{\it BUC}(\R;B^{M-\frac32}_{p,2}(\R^{n-1}))^n$, (\ref{eq:BesovInterpol1}) yields
\begin{equation*}
  \nabla \Gamma \in C^{\tau'}(\R;B^{(n-1)/p}_{p,1}(\R^{n-1}))^n\hookrightarrow C^{\tau'}(\R;C^0_b(\R^{n-1}))^n, 
\end{equation*}
where $\tau'=\tau$ if $0<\tau <\frac12$ and $\tau'\in (0,\frac12)$ is
arbitrary otherwise. Here one uses the general estimate
\begin{eqnarray*}
  \|f(t)-f(s)\|_{X}&\leq& C\|f(t)-f(s)\|_{X_0}^{1-\theta}\|f(t)-f(s)\|_{X_1}^\theta \\
  &\leq & C\|f\|_{BUC(\R;X_0)}^{1-\theta}\|f\|_{C^{\frac12}(\R;X_1)}^\theta|t-s|^{\theta/2}
\end{eqnarray*}
for all $t,s\in \R$, where $X=(X_0,X_1)_{\theta,1}$.
Thus $\nabla \Gamma\in C^{\tau'}(\Rn)^n$. Note that similarly
\begin{equation}\label{eq:W2pEmbedding}
  W^{M-1}_{(2,p)}(\Rn)\hookrightarrow C^{\tau'}(\Rn).
\end{equation}

 Now we define, for some $\lambda >0$, 
$$
F_\gamma(x)= x+ 
\begin{pmatrix}
  0\\ \Gamma(x',\lambda x_n)
\end{pmatrix}=\big(F_{\gamma ,k}(x)\big)_{k=1}^n.
$$
Then $\partial_{x_n} F_{\gamma,n}(x)= 1+\lambda (\partial_{x_n}
\Gamma)(x',\lambda x_n)$. Hence, if $\lambda$ is sufficiently small,
\linebreak$F_{\gamma,n}(x',\cdot)\colon \R\to \R$ is strictly increasing and
surjective for every $x'\in \R^{n-1}$. Therefore $F_\gamma\colon
\Rn\to \Rn$ is a $C^1$-diffeomorphism with $F_\gamma(\Rn_+)=
\Rn_\gamma$, and   $\nabla
F_\gamma=(\partial_{x_j}F_{\gamma ,k})_{j,k=1}^n$ (the transposed
functional matrix) is in  $ C^{\tau'}(\Rn)^{n^2}$. 

Next let $u \in H^{k}(\Omega)$, $0\leq k\leq M$, $\Omega=\Rn$ or $\Omega = \Rn_\gamma$. We prove $F^\ast_\gamma (u) \in H^{k}(\Omega)$ by mathematical induction. If $k=0$, then the statement is true since $F_\gamma $ is a $C^1$-diffeomorphism. 
For $u\in H^{1}(\Omega)$, 
\begin{equation}\label{eq:chain}
  \partial_{x_j} (u(F_\gamma(x))) = (\nabla u)(F_\gamma(x))\cdot \partial_{x_j} F_\gamma(x)  
\end{equation}
where $(\nabla u)(F_\gamma(x))\in L_2(\Omega)^n$ by the argument for $k=0$ and 
and $\partial_{x_j}F_\gamma(x)\in C^0_b(\Omega)^n$.
Next we assume that the statement is true for some $1\leq k <M$. 
Then for $u\in H^{k+1}(\Omega)$, 
we have $(\nabla u)(F_\gamma(x))\in H^{k}(\Omega)^n$ by the assumption,
and $\partial_{x_j}F_\gamma(x)\in W^{M-1}_{(2,p)}(\Omega)^n
\hookrightarrow W^{k}_{(2,p_k)}(\Omega)^n$ for some $2\leq p_k <\infty$ with $k-\frac{n-1}{p_k}-\frac12>0$. Therefore Lemma~\ref{lem:Product} implies that $\partial_{x_j} (u(F_\gamma(x))\in H^{k}(\Rn_+)$ for all $j=1,\ldots, n$, which implies the statement for $k+1\leq M$.
This proves the statement for integer $0\leq k\leq M$. For real $0\leq s\leq M$ the statement follows by interpolation.
\end{proof}

For the case $M=2$, we specify the result in the following
corollary. We use the notation $C : D=\sum_{j,k=1}^nc_{jk}d_{jk}=
\tr (C^T D)$ 
for the ``scalar product'' of
square matrices $C=(c_{jk})_{j,k=1}^n$ and $D=(d_{jk})_{j,k=1}^n$.  $\partial^2 u$ stands for the Hessian
$(\partial_{x_j}\partial_{x_k} u)_{j,k=1}^n$, and $e_j$ is the $j$-th
coordinate vector written as a column. 

\begin{corollary}\label{cor:CoordinateTrafo}
  Let $\gamma\in B^{\frac32}_{p,2}(\R^{n-1})$, 
with $2\le p<\infty$, $\frac12 -\frac{n-1}p> 0$, and let $\Rn_\gamma$ 
and $F_\gamma$ be as in Proposition~{\rm \ref{prop:CoordinateTrafo}}.
Then for every $j,k=1,\ldots,n$, 
\begin{eqnarray}
   F^\ast_\gamma \nabla u &=& \Phi (x)\nabla F^\ast_\gamma
   u,\label{eq:chain2}\\ 
F^\ast_\gamma \partial_{x_j}\partial_{x_k}u
   &=& \Phi _{j,k}(x):\partial^2 F^\ast_\gamma u + Ru,  \nonumber 
\end{eqnarray}
where $\Phi (x)= (\nabla F_\gamma(x))^{-1}\in W^1_{(2,p)}(\Rn_+)^{n^2}$
and $\Phi _{j,k}(x)=\Phi (x)^Te_j e^T_k \Phi (x)$, and 
\begin{equation*}
  R\colon H^{2-\tau}(\Rn_\gamma)\to L_2(\Rn_+)
\end{equation*}
is a bounded operator for $\tau =\frac12-\frac{n-1}p$. 
\end{corollary}
\begin{proof}
The chain rule \eqref{eq:chain} gives that $\nabla (u(F_\gamma
(x))=(\nabla F_\gamma )(x)(\nabla u)(F_\gamma (x))$, which implies the first
line in \eqref{eq:chain2}. In particular, 
\[ F_\gamma ^*\partial_{x_j}u=e^T_j F^*_\gamma (\nabla u)=e^T_j\Phi F^*_\gamma u,y
\]
where $e^T_j\Phi $ is the $j$-th row $\begin{pmatrix}\varphi _{j1}&\hdots& \varphi _{jn}
\end{pmatrix}$ in $\Phi $. Repeated use gives
\begin{eqnarray*}
  F^\ast_\gamma \partial_{j}\partial_{k}u  &=&e_j^T \Phi (x) \nabla
  F^{\ast}_\gamma\partial_{k}u =e_j^T \Phi (x) \nabla\left(e_k^T
    \Phi (x)\nabla  F^{\ast}_\gamma u\right) \\
&=& \begin{pmatrix}\varphi _{j1}&\hdots& \varphi _{jn}
\end{pmatrix}  \begin{pmatrix}\partial _1\\ \vdots\\ \partial_n
\end{pmatrix} \Bigl(\begin{pmatrix}\varphi _{k1}&\hdots& \varphi _{kn}
\end{pmatrix} \begin{pmatrix}\partial_1\\\vdots\\ \partial_n
\end{pmatrix}F^*_\gamma u\Bigr)\\
&=&\sum_{l,m=1}^n\varphi _{jl}\partial _l (\varphi _{km} \partial_m F^*_\gamma u)\\
&=&\sum_{l,m=1}^n\varphi _{jl} \varphi _{km} \partial _l\partial_m F^*_\gamma u
+\sum_{l,m=1}^n \varphi _{jl}(x)(\partial_{l}\varphi _{km}(x)) \partial_{m} F^\ast_\gamma u
\end{eqnarray*}
for all $u\in H^2(\Rn_\gamma)$. Here $(\varphi _{jl}\varphi _{km})_{l,m=1}^ n$
equals the matrix $\Phi ^Te_je^T_k\Phi =\Phi _{j,k}$ for each $j,k$. This shows the second
line in \eqref{eq:chain2}, where $R$ is estimated by use of \eqref{eq:CommEstim}:
\begin{equation*}
  \|\varphi _{jl}(x)(\partial_{l}\varphi _{km}(x)) \partial_{m} F^\ast_\gamma
  u\|_{L_2(\Rn_+)}\leq C\|\partial_{m}F^\ast_\gamma
  u\|_{H^{1-\tau}(\Rn_+)}\leq C'\|u\|_{H^{2-\tau}(\Rn_\gamma)}. 
\end{equation*}
\end{proof}

\begin{rem}\label{rem:Charts}
Now we can choose a covering of  $\comega$ by a system of open sets 
$U_0,\dots, U_J$ with coordinate mappings such
that the $U_j$ for $j=1,\ldots, J$ form a covering of
$\partial{\Omega}$, and $\overline U_0\subset\Omega $. Here in case
$2^\circ$ of Definition \ref{bdryreg}, we can assume that for each
$1\le j\le J$, $U_j$ is bounded, $U_j\cap \Omega=U_j\cap \Rn_{\gamma_j}$ for some $\gamma_j\in
B^{M-\frac12}_{p,2}(\R^{n-1})$ (modulo a rotation), and
the
diffeomorphism $F_j\equiv F_{\gamma _j}$ in $\rn$ mapping
$\Rn_+$ to $ \Rn_{\gamma_j}$ defined by
Proposition~\ref{prop:CoordinateTrafo} is such that $F_j^{-1}$ carries
$\Omega \cap U_j $ over to $V_j=\{(y',y_n):
\max_{k<n}|y_k|<a_j,\,0<y_n<a_j\}$ and $\Sigma \cap U_j$ over to
$\{(y',y_n): \max_{k<n}|y_k|<a_j,\,y_n=0\}$. 
In case $3^\circ$ of Definition \ref{bdryreg}, the open sets
$U_1,\dots,U_{J-1}$ are of this kind, and the set $U_J$ equals
\linebreak$\{(x',x_n) : |x|> R,\, x_n>-1\}$, with $F_J$ being the identity ($\gamma _J=0$).
We shall denote $F_j|_{\partial\Rn_+}=F_{j,0}$, for $j=1,\ldots,J$.

We also introduce $\eta_j, \psi_j,\varphi_j\in C^\infty(\Rn)$,
$j=0,\ldots,J$, as non-negative functions supported in $U_j$ such that $\varphi_0,\ldots,\varphi_J$ is a partition of unity on $\comega$, and
\begin{equation*}
  \psi_j=1 \quad \text{on}\ \supp\varphi_j,\quad   \eta_j=1 \quad \text{on}\ \supp\psi_j, \quad \text{for all}\ j=0,\ldots,J.
\end{equation*}
\end{rem}

It is accounted for in \cite{McLean} in the case of a Lipschitz
boundary (and it also follows in our case by use of the fact that $F_\gamma$ defined in Proposition \ref{prop:CoordinateTrafo} is
a $C^1$-diffeomorphism), that the  surface measure $d\sigma $ on
$\partial\Rn_\gamma$ 
satisfies
\[
d\sigma =\kappa (x')\,dx',\text{ where }
\kappa (x')=\sqrt{1+|\nabla_{x'}\gamma(x')|^2};
\]
here $\kappa\in 
 B^{M-\frac32}_{p,2}(\R^{n-1})$.
We define
\begin{equation*}
  H^s(\partial\Rn_\gamma)= \left\{u\in L_2(\partial\Omega):
    F^\ast_{\gamma,0}u\in H^s(\R^{n-1})\right\},\text{ when }s\geq 0 ,
\end{equation*} 
 provided with the inherited Hilbert space norm
$\|u\|_{H^s(\partial\Rn_\gamma )}=\|F^\ast_{\gamma,0}u\|_{H^s(\R^{n-1})}$.
Furthermore, we put as in \cite{McLean}, for $s>0$,
\begin{equation*}
  \|u\|_{H^{-s}(\partial\Rn_\gamma)}= \left\|\kappa F^\ast_{\gamma,0}u 
\right\|_{H^{-s}(\R^{n-1})}, 
\end{equation*}
 for $u\in L_2(\partial\Rn_\gamma)$, and define $H^{-s}(\partial\Rn_\gamma)$ as
the completion of $L_2(\partial\Rn_\gamma)$ with respect to this norm. 
Then 
\begin{equation*}
  \|u\|_{H^{-s}(\partial\Rn_\gamma)} = \sup_{0\neq v\in H^s(\partial\Rn_\gamma)} \frac{|(u,v)_{L_2(\partial\Rn_\gamma)}|}{\|v\|_{H^s(\partial\Rn_\gamma)}}
\end{equation*}
for all $u\in L_2(\partial\Rn_\gamma)$. Here
$H^{-s}(\partial\Rn_\gamma)$ is naturally identified with the dual of 
$H^s(\partial\Rn_\gamma)$ (more precisely, we hereby mean the space of conjugate linear continuous
functionals as in \cite{LM68}, also called the antidual space) in such
a way that the sesquilinear duality, denoted
$(u,v)_{H^{-s}(\partial\Rn_\gamma),H^s(\partial\Rn_\gamma)}$ or
$(u,v)_{-s,s}$ for short, coincides with the $L_2$-scalar product when
$u\in L_2(\partial\Rn_\gamma )$. We also write
$\overline{(u,v)}_{-s,s}$ as $(v,u)_{s,-s}$.

Moreover, for $-M+\frac32\leq s< 0$ we define $F^{-1,\ast}_{\gamma,0}\colon H^s(\R^{n-1})\to H^s(\partial\Rn_\gamma)$ by
\begin{equation*}
  (F^{-1,\ast}_{\gamma,0}v,\varphi)_{H^{s}(\partial\Rn_\gamma),H^{-s}(\partial\Rn_\gamma)}
=( v,\kappa  F^\ast_{\gamma,0}\varphi )_{H^{s}(\R^{n-1}),H^{-s}(\R^{n-1})} \quad \text{for all}\ \varphi \in H^{-s}(\partial\Rn_\gamma),
\end{equation*}
consistently with the definition of $F^{-1,\ast}_{\gamma,0}v$ for
$v\in L_2(\R^{n-1})$.
 Here $F^{-1,\ast}_{\gamma,0}v\in H^s(\partial\Rn_\gamma)$, 
the dual
space of $ H^{-s}(\partial\Rn_\gamma)$,  
since 
\begin{equation*}
\left\|\kappa F^{\ast}_{\gamma,0}\varphi \right\|_{H^{-s}(\R^{n-1})}
\leq C\left\|\kappa
\right\|_{B^{M-\frac32}_{p,2}(\R^{n-1})}\left\|F^\ast_{\gamma,0}\varphi 
\right\|_{H^{-s}(\R^{n-1})}
\leq C'\left\|\varphi\right\|_{H^{-s}(\partial\Rn_\gamma)} 
\end{equation*}
for all $\varphi \in H^{-s}(\partial\Rn_\gamma)$, because of $0< -s\leq
M-\frac32$ and \eqref{eq:BesselProd2} below. To incorporate the factor 
$\kappa $, we also introduce the modified pull-back mappings
\begin{alignat}{2}\label{eq:TildeF1}
  \widetilde{F}_{\gamma,0}^{\ast}(u)&=\kappa {F}_{\gamma,0}^{\ast}(u)&\quad &\text{for all}\ u\in H^{s}(\partial\Rn_\gamma),\\\label{eq:TildeF2}
  \widetilde{F}_{\gamma,0}^{-1,\ast}(v)&={F}_{\gamma,0}^{-1,\ast}\left(\kappa
    v
\right)&\quad &\text{for all}\ v\in H^{s}(\R^{n-1}),
\end{alignat}
for all $0\leq  s\leq M-\frac32$, whereby
\begin{alignat*}{2}
(\widetilde{F}_{\gamma,0}^{\ast}(u), \varphi)_{-s,s}&=(u, F^{-1,\ast}_{\gamma,0}(\varphi))_{-s,s}&\qquad& \text{for all}\ u\in H^{-s}(\partial\Rn_\gamma), \varphi \in H^{s}(\R^{n-1}),\\
(\widetilde{F}_{\gamma,0}^{-1,\ast}(v), \varphi)_{-s,s}&=(v, F^{\ast}_{\gamma,0}(\varphi))_{-s,s}&\qquad& \text{for all}\ v\in H^{-s}(\R^{n-1}), \varphi \in H^{s}(\partial\Rn_\gamma),
\end{alignat*}
$0\leq  s\leq M-\frac32$. 

Then we have altogether:
\begin{lem}\label{lem:Pullbacks}
Under the assumptions above, we have the mapping properties 
\begin{alignat*}{2}
  F^{\ast}_{\gamma,0}&\colon H^{s}(\partial\Rn_\gamma)\to H^s(\R^{n-1}), F^{-1,\ast}_{\gamma,0}\colon H^{s}(\R^{n-1})\to H^s(\partial\Rn_\gamma)&\ &
\text{if}\ -M+\tfrac32 \leq s\leq M-\tfrac12 ,\\
  \widetilde{F}^{\ast}_{\gamma,0}&\colon H^{s}(\partial\Rn_\gamma)\to H^s(\R^{n-1}), \widetilde{F}^{-1,\ast}_{\gamma,0}\colon H^{s}(\R^{n-1})\to H^s(\partial\Rn_\gamma)&\ &
\text{if}\ -M+\tfrac12 \leq s\leq M-\tfrac32 ,
\end{alignat*}
continuously, and $\widetilde{F}^{-1,\ast}_{\gamma,0}$ is the adjoint
of $F^\ast_{\gamma,0} $,   $F^{-1,\ast}_{\gamma,0}$ is the adjoint of $\widetilde{F}^\ast_{\gamma,0}$.
\end{lem}

Now one can define $H^s(\partial\Omega)$ using suitable partitions of unity, 
when $\partial\Omega$ is of class $B^{M-\frac12}_{p,2}$ as in
$2^\circ$ and $3^\circ$ of
Definition \ref{bdryreg}; cf.\ Remark \ref{rem:Charts}.

\begin{theorem}\label{thm:Traces}
  Let $\Omega\subseteq \Rn$ be as in Definition {\rm
    \ref{bdryreg}}. For every $s\in (\frac12, M]$, there is a continuous linear mapping
  $\gamma_0$ such that $\gamma_0 u= u|_{\partial\Omega}$ for all $u\in
  H^s(\Omega)\cap C^0(\overline{\Omega})$  and $\gamma_0\colon
  H^s(\Omega) \to H^{s-\frac12}(\partial\Omega)$ is bounded. 
 Moreover, there is a continuous right inverse  of $\gamma_0\colon H^s(\Omega) \to H^{s-\frac12}(\partial\Omega)$.
\end{theorem}
\begin{proof}
  In the case $\Omega=\Rn_\gamma$,  the statement can be  reduced to
  the well-known corresponding fact in the case of $\Omega =\Rn_+$ using
  $F_\gamma$ established in Proposition~\ref{prop:CoordinateTrafo}. With the
  aid of suitable partitions of unity, the general cases can be
  reduced to the case $\Omega=\Rn_\gamma$. 
\end{proof}

We note that Gau\ss 's formula
\begin{equation}\label{eq:Gauss}
  \int_\Omega \Div f(x) \sd x = -\int_{\partial\Omega} \vn\cdot f(x) \sd \sigma(x) 
\end{equation}
is valid for any $f\in C^1(\overline{\Omega})^n$ with compact support
if $\Omega$ is a Lipschitz domain. Here $\vn=(n_1,\ldots,n_n)$ denotes
the interior normal of $\partial\Omega$. A proof can be found e.g. in
\cite[Theorem~3.34]{McLean}. Because of \cite[Theorems 3.29 and
3.38]{McLean}, \eqref{eq:Gauss} also holds true for any $f\in H^1(\Omega)^n$.

\subsection{Pointwise multiplication and inversion}
\label{sec:Inversion} 

First of all, we recall the following product estimates: For every $r>0$, $|s|\leq r$, and $1\leq p\leq q<\infty$ such that $\frac1p+\frac1q\leq 1$ and $r-\frac{n}q>0$ there is some constant $C_{r,s,p,q}>0$ such that
\begin{equation}\label{eq:BesselProd}
  \|fg\|_{H^s_p(\Rn)}\leq C_{r,s,p,q}\|f\|_{H^r_q(\Rn)}\|g\|_{H^s_p(\Rn)},\quad f\in H^r_q(\Rn), g\in H^s_p(\Rn),
\end{equation}
cf. e.g. Johnsen~\cite[Theorems~6.1 and 6.4]{JohnsenPointwiseMultipliers}.

Moreover, due to Hanouzet~\cite[Th\'eor\`eme 3]{HanouzetBesovProd} we have
\begin{equation}\label{eq:BesovProd}
  \|fg\|_{B^s_{p, \max(q_1,q_2)}}\leq C_{r,s,p,p_1,q_1,q_2}\|f\|_{B^r_{p_1,q_1}}\|g\|_{B^s_{p,q_2}}
\end{equation}
for all $ f\in B^r_{p_1,q_1}(\Rn),g\in B^s_{p,q_2}(\Rn)$ provided that  $1\leq p\leq p_1\leq \infty$, $1\leq q_1,q_2\leq \infty$, $r>\frac{n}{p_1}$, and 
$$-r+n\left(\tfrac{1}{p_1}+\tfrac{1}{p}-1\right)_+ < s\leq r,$$ 
see also \cite[Theorem~6.6]{JohnsenPointwiseMultipliers}. In particular, this implies
\begin{equation}\label{eq:BesselProd2}
  \|fg\|_{H^s(\R^{n})}\leq C_{s,r,p}\|f\|_{B^r_{p,2}(\R^{n})}\|g\|_{H^{s}(\Rn)}
\end{equation}
for all $f\in B^r_{p,2}(\Rn),  g\in H^s(\Rn)$, provided that $2\leq p\leq \infty$, $r-\frac{n}p>0$ and $-r<s\leq r$.

Concerning pointwise inversion,
let $X=B^s_{p,q}(\Rn)$ with $s>\frac{n}p$, $1\leq p,q\leq \infty$ or $X= H^s_p(\Rn)$ with $s>\frac{n}p$, $1 < p< \infty$. Then
$G(f)\in X$ for all $G\in C^\infty(\R)$ and $f\in X$. 
This implies that $f^{-1}\in X$ for all $f\in  X$ such that $|f|\geq c_0>0$. We refer to Runst \cite{RunstCompOp} for an overview, further results, and references.

\subsection{Pseudodifferential operators with nonsmooth coefficients}

Let $X$ be a Banach space such that $X\hookrightarrow L_\infty(\Rn)$.
\begin{defn}
  For every $m\in \R$ the symbol space $XS^m_{1,0}(\Rn\times\Rn)$, $n\in\N$, is the set of all $p\colon \Rn\times \Rn\to \C$ such that for every $\alpha \in \N_0^m$ there is some $C_\alpha>0$ satisfying
  \begin{equation*}
    \|\partial_{\xi}^\alpha p(.,\xi)\|_X \leq C_{\alpha} \weight{\xi}^{m-|\alpha|} \qquad \text{for all}\ \xi \in \Rn.
  \end{equation*}
  The space $XS^m_{\operatorname{cl}}(\Rn\times\Rn)$ is the set of all $p\in XS^m_{1,0}(\Rn\times\Rn)$, which are classical symbols in the sense that there are $p_j\in XS^{m-j}_{1,0}(\Rn\times\Rn)$, $j\in \N_0$ that are homogeneous with respect to $|\xi|\geq 1$ and satisfy
  \begin{equation*}
    p(x,\xi)-\sum_{j=0}^{N-1} p_j(x,\xi) \in X S^{m-N}_{1,0}(\Rn\times\Rn)\qquad \text{for all}\ N\in\N.
  \end{equation*}
\end{defn}

In order to define pseudodifferential operators on $\partial\Omega$ with $B^{M-\frac12}_{p,2}$-regularity, we recall:
\begin{thm}\label{lem:HsBddnes}
  Let $p\in H^r_q S_{1,0}^{m}(\Rn\times \Rn)$ for some $2\leq q <\infty$ and  $r>\frac{n}q$. Then 
  \begin{equation*}
    p(x,D_x)\colon H^{s+m}(\Rn)\to H^s(\Rn) \qquad \text{for all}\ -r< s\leq r
  \end{equation*}
  is a bounded linear operator.
\end{thm}
The theorem follows from \cite[Theorem~2.2]{MarschallSobolevCoeff}. We
note that, if $p(x,\xi)= \sum_{|\alpha|\leq m} a_\alpha (x)
\xi^\alpha$ for some $a_\alpha\in H^r_q(\Rn)$, then $p(x,D_x)$ is a
differential operator with coefficients in $H^r_q(\Rn)$ and the
statement in the theorem easily follows from the product estimate \eqref{eq:BesselProd}
provided that $|s|\leq r$, $2\leq q<\infty$ and $r>\frac{n}q$.

Let us recall the so-called symbol-smoothing: For every $p\in C^rS^m_{1,0}(\Rn\times \Rn)$ and $0<\delta < 1$ there is a decomposition  
\begin{equation}\label{eq:SymbolSmoothing1}
p= p^\sharp+ p^b,\quad \text{where}\ p^\sharp\in S^m_{1,\delta}(\Rn\times \Rn), p^b\in  C^rS^{m-r\delta}_{1,\delta}(\Rn\times \Rn).
\end{equation}
We refer to \cite[end of Section 1.3]{TaylorNonlinearPDE} for a proof. The definition of $C^r S^{m-r\delta}_{1,\delta}(\Rn\times \Rn)$ is given in the appendix.

In order to estimate the remainder term $p^b(x,D_x)$ we will use:
\begin{prop}\label{prop:HsBddness}
  Let $p\in C^rS^m_{1,\delta}(\Rn\times \Rn)$,
  $m\in\R$, $\delta \in [0,1]$, $r>0$. Then    
\[
p(x,D_x)\colon H^{s+m}(\Rn) \to H^{s}(\Rn)
\]
for all $s\in\R$ with $-r(1-\delta) <s < r$.
\end{prop}
We refer to \cite[Theorem~2.1]{Marschall2} for a proof.

Now let $\partial\Omega$ be again  of class $B^{M-\frac12}_{p,2}$, where $M-\frac32-\frac{n-1}p>0$. Then
we can define a pseudodifferential operator $P$ on $\partial\Omega$ of order $m'\in \R$ and with coefficients in $H^r_q(\R^{n-1})$ for $2\leq q <\infty$ and  $r>\frac{n}q$, as
\begin{equation}\label{eq:DefnP}
  P u = \sum_{j=1}^J \psi_j F_{j,0}^{-1,\ast} p_j(x',D_{x'}) F_{j,0}^{\ast} \varphi_j u,
\end{equation}
where the $p_j\in H^r_qS^{m'}_{1,0}(\R^{n-1}\times \R^{n-1})$, and
$F_{j,0}\colon  V_j\to U_j\subset \partial\Omega$,
$j=1,\ldots,J$, where $V_j \subset \R^{n-1}$, are local charts forming an atlas of
$\partial\Omega$. Moreover, $\varphi_j$, $j=1,\ldots,J$ is a partition of unity on $\partial\Omega$ 
such that $\supp \varphi_j\subset U_j$, and the functions $\psi_j$ satisfy $\psi_j \equiv 1$ on
$\supp \varphi_j$ and have $\supp\psi_j \subset U_j$.
Here we assume that at least $\varphi_j,\psi_j\in B^{M-\frac12}_{p,2}(\partial\Omega)$. 

For later purposes we also define the modified pseudodifferential operator
\begin{equation}\label{eq:DefnPt}
  \widetilde{P} u = \sum_{j=1}^J \psi_j \widetilde{F}_{j,0}^{-1,\ast} p_j(x',D_{x'}) F_{j,0}^{\ast} \varphi_j u,
\end{equation}
where $p_j$ are as above and $\widetilde{F}_{j,0}^{-1,\ast}=\widetilde{F}_{\gamma_j,0}^{-1,\ast}$. For these operators we have the following slightly different continuity results:

\begin{corollary}
  Let $\partial\Omega$ be of class $B^{M-\frac12}_{p,2}$, $M-\frac32 -
  \frac{n-1}p>0$, $2\leq p<\infty$, and let 
 $p_j\in H^r_q S_{1,0}^{m'}(\R^{n-1}\times \R^{n-1})$ for some
 $2\leq q <\infty$ and  $r>\frac{n-1}q$, $j=1,\ldots,J$. Then with $P$ as in {\rm (\ref{eq:DefnP})} we have for
 every $s\in \R$ such that  $-r<s\leq r$, $s,s+m'\in [-M+\frac32,M-\frac12]$,  
  \begin{equation*}
    P\colon H^{s+m'}(\partial\Omega)\to H^s(\partial\Omega) 
  \end{equation*}
  is a well-defined linear and bounded operator. Moreover, for
 every $s\in \R$ such that  $-r<s\leq r$, 
$s \in [-M+\frac12,M-\frac32]$,  $s+m'\in [-M+\frac32,M-\frac12]$ and $\widetilde{P}$ as in {\rm (\ref{eq:DefnPt})}  we have: 
  \begin{equation*}
    \widetilde{P}\colon H^{s+m'}(\partial\Omega)\to H^s(\partial\Omega) 
  \end{equation*}
  is a well-defined linear and bounded operator.
\end{corollary}
\begin{proof} 
  The proof follows immediately from Theorem~\ref{lem:HsBddnes} and Lemma~\ref{lem:Pullbacks} and local charts.
\end{proof}

In particular, 
we can define differential operators with $H^r_q$-coefficients in the manner above.

\begin{rem}
  We will not address the question of invariance under coordinate
  transformation of nonsmooth pseudodifferential operators. Therefore
  we will also not show that the definition (\ref{eq:DefnP}) does not
  depend in an essential way on the choice of the charts and the cut-off
  functions $\varphi_j,\psi_j$.  
\end{rem}

We recall from \cite[Corollary~3.4]{MarschallSobolevCoeff}:
\begin{thm}\label{thm:PsDOCompo} 
  Let $p_j\in H^r_q S^{m_j}_{1,0}(\Rn\times\Rn)$, $m_j\in \R$, $j=1,2$, $2\leq q<\infty$ and $r>\frac{n}q$. Then for every $0<\tau\leq 1$ with $\tau < r-\frac{n}q$
  \begin{equation*}
    p_1(x,D_x)p_2(x,D_x)- (p_1p_2)(x,D_x)\colon H^{s+m_1+m_2-\tau}(\Rn)\to H^s(\Rn) 
  \end{equation*}
  is a bounded operator
  provided that
  \begin{equation}\label{eq:Cond1}
 s,s+m_1\in (-r+\tau, r].
  \end{equation}
\end{thm}
Note here that if the $p_j(x,D_x)$ are differential operators, then
Theorem~\ref{thm:PsDOCompo} can be proved by elementary but lengthy
estimates using Sobolev embeddings. As one consequence of the 
theorem one has that 
\begin{equation}\label{eq:Reduction1}
  \sum_{|\alpha|,|\beta|\leq m } D_x^\alpha (a_{\alpha,\beta}(x) D_x^\beta u)=  \sum_{|\alpha|,|\beta|\leq m } a_{\alpha,\beta}(x) D_x^{\alpha+\beta} u+ R u,
\end{equation}
where 
\begin{equation}\label{eq:Reduction2}
  R\colon H^{s+2m-\tau}(\Omega)\to H^s(\Omega) \qquad \text{if}\ -r+\tau <s \leq  r-m 
\end{equation}
provided that $a_{\alpha,\beta}\in H^r_q(\Omega)$ and $\Omega$ is a
Lipschitz domain. In fact, the statement in the case $\Omega=\Rn$
follows from Theorem~\ref{thm:PsDOCompo}, and then for a general Lipschitz
domain $\Omega$ one can obtain the statement by extension to $\Rn$.

\subsection{Green's formula for second order boundary value problems}\label{subsec:Green}
Since the smoothness properties of the coefficients in Green's formula
for general $2m$-order operators are quite complicated to analyse and
would take up much space, we
shall in the present paper  restrict the attention to the second-order
case from here on (expecting to take up higher-order problems elsewhere).

Consider a second order strongly elliptic operator $A$, 
\begin{equation}\label{eq:DefnA}
  A u = -\sum_{j,k=1}^n \partial_{x_j} (a_{jk}
  \partial_{x_k} u) +\sum_{j=1}^n a_j\partial_{x_j}u+ a_0u,
\end{equation}
with
\begin{equation}\operatorname{Re}\sum_{j,k=1}^n a_{jk}(x)\xi _j\xi
  _k\ge c_0|\xi |^2, \text{ all }x\in\Omega,\,\xi \in \Rn;\label{str.ell}
\end{equation}
$c_0>0$. We assume that the  $a_{jk}$ and $a_j$ are in
$H^1_q(\Omega)$ and $a_0\in L_q(\Omega)$, where $q\ge 2$ and $1-\frac{n}q>0$, 
 and we apply the expression to $u\in H^2(\Omega)$.  
In view of \eqref{eq:BesselProd},
\begin{equation*}
  \|fg\|_{H^s(\Omega)} \leq C_{q}\|f\|_{H^1_q(\Omega)}\|g\|_{H^s(\Omega)}\qquad \text{for all}\ |s|\leq 1;
\end{equation*} 
hence
\begin{equation}\label{eq:MappingA}
  A \colon H^{2+s}(\Omega)\to H^s(\Omega)\qquad \text{for all}\ s\in [-2,0].
\end{equation}

Concerning the domain $\Omega $, we assume that $\Omega $ is as in
Definition \ref{bdryreg} $2^\circ$ or $3^\circ$ with $M=2$ (so 
$\partial\Omega \in 
B^{\frac32}_{p,2}$ and $\frac12-\frac{n-1}p>0$, in particular $p > 2$).

Denote $\partial_\vn=\vn\cdot \partial$, the normal derivative, where $\vn=(n_1,\ldots,n_n)$ is the interior unit normal  on
$\partial\Omega$. We shall 
denote $\partial_\tau =\pr_\tau \partial$, where
$\pr_{\tau}=I-\vn\otimes \vn$; the ``tangential gradient''.
(Here $\vn\otimes\vn$ is the matrix $(n_jn_k)_{j,k=1,\dots,n}=\vn\,
\vn^T$, $\vn$ used as a column vector.) 
Setting $\partial_{\tau ,j}=e_j\cdot \partial_\tau $, we have (at
points of $\partial\Omega $), since 
\[
\partial_{\tau ,j}u=e_j\cdot \partial_\tau u=e_j\cdot \partial 
 u-e_j\, \vn \,\vn^T \partial u=\partial_{x_j}u-n_j \vn\cdot \partial u,
\]
that
\begin{equation}\label{eq:dtauj}
\partial_{x_j}u=n_j\partial_\vn u+\partial_{\tau ,j}u.
\end{equation}
 When $\xi\in \C^n$, 
we set $\xi_\tau=\pr_\tau \xi=(I-\vec
 n\otimes\vec n)\xi $. For $j\in\N_0$, we define
\begin{equation*}
  \gamma_j u := \big((\vn\cdot \partial_x)^j
  u\big)|_{\partial\Omega}=\partial_\vn^j u|_{\partial\Omega}.
\end{equation*}

By our assumptions, $\vec n\in
B^{\frac12}_{p,2}(\partial\Omega)^n\hookrightarrow
H^{\frac12}_p(\partial\Omega)^n$ (cf.\ 
\eqref{eq:BesovEmb4}, recall that $p\geq 2$). The product rule
\eqref{eq:BesselProd} applies with $r=s=\frac12$, $p=q$, $n$ replaced
by $n-1$, to show that $H^{\frac12}_p(\partial\Omega)$ is an algebra 
with respect to pointwise multiplication. The rule
\eqref{eq:BesselProd} also applies with $p=2$, $r=\frac12$, $n$ replaced by $n-1$
and $q$ replaced by our general $p$, to show that
multiplication by elements of
$H^{\frac12}_p(\partial\Omega)$ preserves $H^{s}(\partial\Omega)$
for $|s|\le \frac12$.

Then since $\gamma _0\partial_{x_j}\colon H^s(\Omega )\to
H^{s-\frac32}(\partial\Omega )$ continuously for $s\in (\frac32,2]$,
 $\gamma _1\colon H^s(\Omega )\to
H^{s-\frac32}(\partial\Omega )$ continuously for $s\in (\frac32,2]$.

Let 
\begin{equation*}
a'_{jk}(x)= \ol{a_{kj}(x)},\quad a'_j(x)= \ol{a_j(x)},\quad a_0'(x)=\ol{a_0(x)}-\sum_{j=1}^n \partial_{x_j}\ol{a_j(x)}  
\end{equation*}
 for all $x\in
\Omega$, $j,k=1,\ldots,n$, and let $A'$ denote the operator
defined as in (\ref{eq:DefnA}) with $a_{jk},a_j,a_0$ replaced by
$a'_{jk},a'_j,a_0'$; it is the formal adjoint of $A$.

It will be convenient for the following to have
$\frac12-\frac{n-1}p\le 1-\frac nq$; this can be achieved by replacing
$p$ by a smaller $p> 2$. Then we can set $\tau =\frac12-\frac{n-1}p$
as in Assumption \ref{tau-assumption}. To sum up, we make the
following assumption:

\begin{assumption}\label{tau-assumption2} $n\ge 2$, $2< q\le \infty
  $ and $2< p<\infty $, with 
\begin{equation}
1-\tfrac nq\ge\tfrac12-\tfrac{n-1}p> 0 ;\quad \tau:=\tfrac12-\tfrac{n-1}{p}.\label{asspt2}
\end{equation}
The domain $\Omega $ is as in Definition  {\rm \ref{bdryreg}}
$2^\circ$ or $3^\circ$, with boundary $\partial\Omega$ of regularity 
$B^{\frac32}_{p,2}$.
In \eqref{eq:DefnA}, the coefficients $a_{jk}$ and $a_j$ are in
$H^1_q(\Omega )$ and $a_0\in L_q(\Omega )$.
\end{assumption}

The inequalities for $q$ and $p$ mean that $(\frac n q, \frac {n-1}
p)$ belongs to the polygon $\{ (x,y) : 0\le x<1, 0<y<\frac12, y\ge
x-\frac12\}$. All $q>n$ and $p>2(n-1)$ can occur (but not all at the
same time).
Under Assumption \ref{tau-assumption2},
\begin{equation*} 
  B^{1-\tfrac1q}_{q,q}(\partial\Omega)+ B^{\frac12}_{p,2}(\partial\Omega)\hookrightarrow H^{\frac12}_p(\partial\Omega)
\end{equation*}
by (\ref{eq:SobolevEmb}), where we also use that
$B^{\frac12}_{p,2}\hookrightarrow H^{\frac12}_p$.
Recall from Remark \ref{rem:holder} that the boundary regularity $B^{\frac32}_{p,2}$ includes
$C^{\frac32+\varepsilon }$ and is included in  $ C^{1+\tau }$.

\begin{thm}\label{thm:Green} Under Assumption {\rm \ref{tau-assumption2}},
the following Green's formula holds:
  \begin{equation}\label{Green}
    (Au,v)_\Omega - (u,A'v)_{\Omega} 
=(\chi u, \gamma_0
v)_{\partial\Omega}-(\gamma_0 u, \chi' v)_{\partial\Omega}, 
  \end{equation}
  for all $u,v\in H^2(\Omega)$. 
 Here, setting   $B= (a_{jk})_{j,k=1}^n$, $B'= (a_{jk}')_{j,k=1}^n$,
 and $b'_0= \sum_{j=1}^n n_ja_j' $, and defining 
  \begin{align*} 
    s_0&=\sum_{j,k=1}^n a_{jk}(x) n_j n_k= \vn^TB\vn ,\\
 {\cal A}_1&= b_1\cdot \partial_\tau, \text{ with } b_1=(\vn^T B)_\tau,
 \quad{\cal A}_1'=b_1'\cdot \partial_\tau+  b_0',\text{ with } b_1'=(\vn^T B')_\tau, \\
  \end{align*}
we have that
  \begin{equation}
      \chi= s_0 \gamma_1 +  {\cal A}_1\gamma_0,\quad
      \chi'= \overline s_0 \gamma_1 +  {\cal A}_1'\gamma_0.\label{eq:conormal}
  \end{equation}
 Here $s_0, b'_0\in
H^{\frac12}_p(\partial\Omega)$,  
$b_1,b_1'\in H^{\frac12}_p(\partial\Omega)^n$.

Furthermore, 
$s_0$ is invertible with $s_0^{-1}$ likewise in
$H^{\frac12}_p(\partial\Omega)$.

\end{thm}
\begin{proof}
  It is well-known that when coefficients and boundary are smooth, then 
the Gauss formula \eqref{eq:Gauss} implies  \begin{equation*}
      (Au,v)_\Omega - (u,A'v)_{\Omega} =(\chi u, \gamma
v)_{\partial\Omega}-(\gamma u, \chi' v)_{\partial\Omega},
  \end{equation*}
  for all $u,v\in H^2(\Omega)$, where
  \begin{align*}
    \chi u &= \sum_{j,k=1}^n n_j \gamma_0 a_{jk}\partial_{x_k} u,\\
    \chi' u &= \sum_{j,k=1}^n n_j \gamma_0 a_{jk}'\partial_{x_k} u+
    \sum_{j=1}^n n_j  \gamma _0a_j' u.
  \end{align*}
Here we can write, using \eqref{eq:dtauj},
\begin{align*} \chi u &=\sum_{j,k=1}^n n_j  a_{jk}\gamma _0 \partial _{x_k} u=
\sum_k (\vn^T B)_k \gamma _0(\partial_{\tau ,k}u+n_k\partial_\vn u)\\
&
=\vn^T B \vn \gamma _1 u+\vn^T B \partial_\tau \gamma _0u=\vn^T B \vn \gamma _1 u+(\vn^T B )_\tau \cdot\partial_\tau \gamma _0u;
\end{align*}
similarly, 
\[
\chi' u = \sum_{j,k=1}^n n_j \gamma_0 a_{jk}'\partial_{x_k} u+
    \sum_{j=1}^n n_j  a_j'\gamma_0 u=\vn^T B '\vn \gamma _1 u+(\vn^T B')_\tau \cdot \partial_\tau \gamma _0u+\vn^T  b'_0\gamma_0 u.
\]
This shows the asserted formulas in the smooth case.

The validity is extended in
\cite[Theorem~4.4]{McLean} to the case where $a_{jk},a_j\in
  C^{0,1}(\ol{\Omega})$ and 
$a_0\in L_\infty(\Omega)$, and  $\Omega $ is a bounded Lipschitz
domain. The unbounded cases in Definition \ref{bdryreg} are 
included by adding the appropriate (trivial) coordinate charts.

  The case $a_{jk},a_j\in H^1_q(\Omega)$, $a_0\in L_q(\Omega)$ can 
then   easily be proved by first replacing $a_{jk},a_j$, and $a_0$ by some 
smoothed $a_{jk}^\eps,a_j^\eps\in C^{0,1}(\ol{\Omega})$ and 
$a_0^\eps\in L_\infty(\Omega)\cap L_q(\Omega )$ such that $a_{jk}^\eps
\to_{\eps\to 0} a_{jk}$ and $a_j^\eps\to_{\eps\to 0} a_j$ in 
$H^1_q(\Omega)$ and $a_0^\eps\to a_0$ in $L_q(\Omega)$ and then
passing 
to the limit $\eps\to 0$. For this argument one uses
(\ref{eq:BesselProd}) 
with $s=r=1$ and $p=2$ to pass to the limit in all terms involving 
$a_{jk}, a_{jk}', a_j$, and $a_j'$. To pass to the limit in the term 
involving $a_0$ one uses $H^{\frac{n}q}(\Omega)\hookrightarrow L_r(\Omega)$ with $\frac1r= \frac12 - \frac1q$ since $\frac1q<\frac1n$, which implies
  \begin{equation}\label{eq:a0Estim}
    \|a_0 u\|_{L_2(\Omega)}\leq C\|a_0\|_{L_q(\Omega)}\|u\|_{H^{\frac{n}q}(\Omega)}.
  \end{equation}

Finally, since $A$ is strongly elliptic,
 $|s_0(x)|\geq C>0$ for all $x\in \ol{\Omega}$. Then $s_0^{-1}\in
 H^{\frac12}_p(\partial\Omega)$ because of the results at the end of
 Section~\ref{sec:Inversion}.  
\end{proof}

Note that since the coefficients in the trace operators $\chi $ and $\chi '$ are in $
H^{\frac12}_p(\partial\Omega)$, $\chi $ and $\chi '$ map $H^s(\Omega
)$ continuously into $H^{s-\frac32}(\partial\Omega )$ for $s\in (\frac32,2]$.

We shall also need a result on localization of $\chi$ and $\chi '$, and
their surjectivity. 

\begin{corollary}\label{cor:TraceChi}
  Let $\chi$ and $\chi'$ be as in the preceding theorem and let
  $U\subset\Rn$ be such that $\Omega\cap U$ coincides with
  $\Rn_\gamma\cap U$ (after a suitable rotation), where $\gamma\in
  B^{\frac32}_{p,2}(\R^{n-1})$. 
Then there is a trace operator 
  \begin{equation*}
    t(x',D_x)=s_1(x') \gamma_0\partial_{x_n} +  \sum_{|\alpha|\leq 1} c_\alpha(x') D_{x'}^\alpha\gamma_0, 
  \end{equation*}
  where $s_1, c_\alpha\in H^{\frac12}_p(\R^{n-1})$ for all $|\alpha|\leq 1$, such that
  \begin{equation}\label{eq:chiRepr}
    \chi (\psi u )= \eta F^{-1,\ast}_{\gamma,0} t(x',D_x)
    F^\ast_\gamma(\psi u )
  \end{equation}
  for any $u\in H^2(\Omega)$ and $\psi,\eta\in C_0^\infty(U)$ with
  $\eta \equiv 1$ on $\supp \psi$. Here $F_\gamma$ is the
  diffeomorphism from Proposition~{\rm \ref{prop:CoordinateTrafo}} and
  $F_{\gamma,0}=F_\gamma|_{\partial\Rn_+}$. Moreover, $s_1$ is
  invertible.

For every $s\in (\frac32,2]$ there is a continuous right-inverse of
  \begin{equation*}
    \begin{pmatrix}
      \chi\\ \quad\\ \gamma_0 
    \end{pmatrix}
    \colon H^s(\Omega)\to 
    \begin{matrix}
      H^{s-\frac32}(\partial\Omega)\\\times \\H^{s-\frac12}(\partial\Omega)
    \end{matrix};
  \end{equation*}
this holds in particular with $\chi $ replaced by $\gamma _1$.  The analogous statements hold for $\chi'$.
\end{corollary}
\begin{proof}
To prove the first statement, let $\Omega\cap U$ coincide with
$\Rn_\gamma\cap U$ after a suitable rotation of $\Omega$ and let
$\psi,\eta \in C_0^\infty(U)$ with $\eta\equiv 1$ on $\supp
\psi$. Then with $B$ as in Corollary~\ref{cor:CoordinateTrafo}, 
  \begin{align*}    
   s_0\gamma_1 (\psi u)&=\eta s_0\gamma _0(\vec n\cdot \nabla(\psi u))=\eta
   s_0 \gamma_0( \vn\cdot F^{-1,\ast}_{\gamma}B\nabla
   F^{\ast}_\gamma(\psi u))\\ 
&= \eta s_0  F^{-1,\ast}_{\gamma,0}a_0 \gamma_0 (\partial_{x_n}
F^{\ast}_\gamma(\psi u)) + \eta s_0  F^{-1,\ast}_{\gamma,0}B'\nabla_{x'}
\gamma_0 F^{\ast}_\gamma(\psi u) 
  \end{align*}
  where $a_0= (F^{\ast}_{\gamma,0}\vn) \cdot  B(x',0)e_n\in
  H^{\frac12}_p(\R^{n-1})$ and $B'= (F^{\ast}_{\gamma,0}\vn) \cdot
   B(x',0)(I-e_n\otimes e_n)\in
  H^{\frac12}_p(\R^{n-1})$.
Moreover,
\begin{eqnarray*}
  (F^{\ast}_{\gamma,0}\vn)\cdot  (B(x',0)e_n)&=& \frac1{\sqrt{1+|\nabla_{x'}\gamma|^2}}
  \begin{pmatrix}
    -\nabla_{x'}\gamma(x')\\ 1 
  \end{pmatrix}
\cdot \frac1{b(x')}  \begin{pmatrix}
    -\nabla_{x'}\gamma(x')\\ 1+b(x') 
  \end{pmatrix}\\
  &=&\frac{\sqrt{1+|\nabla_{x'}\gamma|^2}}{b(x')} + \frac1{\sqrt{1+|\nabla_{x'}\gamma|^2}}\geq c>0 
\end{eqnarray*}
where $b(x')= 1+\lambda\partial_{x_n}\Gamma(x',0)\in
[\frac12,\frac32]$ as in the proof of
Proposition~\ref{prop:CoordinateTrafo}. Since $A$ is elliptic,
$|s_0(x)|\geq c_0>0$ for all $x\in\partial\Omega$, too. Therefore
$s_1=s_0F^{-1,\ast}_{\gamma ,0}a_0\in H^{\frac12}_p(\partial\Omega)$ is invertible. It is
easy to observe that $\mathcal{A}_1\gamma_0 u=\eta
F^{-1,\ast}_{\gamma,0} \sum_{|\alpha'|\leq 1} c_\alpha' D_{x'}\gamma_0
\psi u$ for some $c_\alpha'\in H^{\frac12}_p(\R^{n-1})$. This proves
the first statement.  

  To prove the last statement, we first note that there is a linear
  extension operator $K$ such that $K\colon H^{s-\frac32}(\R^{n-1})\to
  H^s(\Rn_+)$ for all $s\in (\frac32,2]$ and $\gamma_1 K v= v$ and
  $\gamma_0 Kv=0$ for all $v\in H^{s-\frac32}(\R^{n-1})$. 
  Let $F_j,F_{j,0}, \varphi_j,\psi_j,\eta_j$ be as in
  Remark~\ref{rem:Charts} for $M=2$. Using $K$ and the coefficients
  $s_{1,j}$
 in \eqref{eq:chiRepr} with
  respect to $\Rn_{\gamma_j}$, we define
  \begin{equation*}
    K_1v =\sum_{j=1}^N \psi_j F^{-1,\ast}_j K s_{1,j}^{-1}F^{\ast}_{j,0}\varphi_j v.
  \end{equation*}
Then 
  $\gamma_0 K_1v=0$ and therefore
  \begin{eqnarray*}
    \chi K_1v&=& \sum_{j=1}^N \psi_j F^{-1,\ast}_j s_{1,j}\gamma_1 K s_{1,j}^{-1}F^{\ast}_{j,0}\varphi_j v
=\sum_{j=1}^N \psi_j F^{-1,\ast}_j s_{1,j}s_{1,j}^{-1} F^{\ast}_{j,0}\varphi_j v=v,
  \end{eqnarray*}
  where we have applied \eqref{eq:chiRepr} with respect to
  $\Rn_{\gamma_j}$.

 Now we define 
\begin{equation*}
  \mathcal{K}
  \begin{pmatrix}
    v_1\\v_2
  \end{pmatrix}
  = K_0 v_2 +K_1(v_1 - \chi K_0 v_2)
\end{equation*}
for all $v_1\in H^{s-\frac32}(\partial\Omega)$, $v_2\in
H^{s-\frac12}(\partial\Omega)$, where $K_0\colon
H^{s-\frac12}(\partial\Omega)\to H^s(\Omega)$ is a right inverse of
$\gamma_0$, which exists in view of Theorem~\ref{thm:Traces}. Then $\mathcal{K}$ is a right-inverse of $\begin{pmatrix}
      \chi\\ \gamma_0 
    \end{pmatrix}$. In the special case $A=-\Delta $, $\chi =\gamma
    _1$.

\end{proof}

\setcounter{section}{2}
\section{Extension theory}\label{extend}

In this section we briefly recall some elements of the theory of extensions 
of dual pairs established in Grubb \cite{Grubb68} (building on works of
Krein \cite{Krein47}, Vishik \cite{Vishik52} and Birman \cite{Birman56}) and its
relation to $M$-functions shown in Brown-Grubb-Wood \cite{BGW09}. 

We start with a pair of closed, densely defined linear
operators $A_{\min}$, $A'_{\min}$ in a Hilbert
space $H$ satisfying:
\[
A_{\min}\subset (A'_{\min})^*=A_{\max},\quad A'_{\min}\subset (A_{\min})^*=A'_{\max};
\]
a so-called dual pair.
By $\mathcal M$ we denote the set of linear operators lying between the
minimal and maximal operator:
\[\mathcal M=\{\wA\mid \Ami\subset
\wA\subset \Ama\},\quad
\mathcal M'=\{\wA'\mid \Ami'\subset
\wA'\subset \Ama'\}.\]
Here we write $\wA u$ as $Au$ for any $\wA$, and $\wA' u$ as $A'u$ for any $\wA'$.
We assume that there exists an $A_\gamma \in\mathcal M$ with $0\in
\varrho (A_\gamma )$; then $A_\gamma ^*\in \mathcal M'$ with $0\in \varrho
(A_\gamma ^*)$. 

Denote
\[
Z=\operatorname{ker}\Ama, \quad Z'=\operatorname{ker}\Ama',
\]
and define the basic non-orthogonal decompositions 
\begin{align}
D(\Ama)&=D(A_\gamma )\dot+ Z,\text{ denoted }u=u_\gamma +u_\zeta
=\pr_\gamma u+\pr_\zeta u,\nonumber\\
D(\Ama')&=D(A_\gamma ^*)\dot+ Z',\text{ denoted } v=v_{\gamma '}+v_{\zeta '}
=\pr_{\gamma '}v+\pr_{\zeta '}v;\nonumber
\end{align}
here $\pr_\gamma =A_\gamma ^{-1}\Ama$, $\pr_\zeta =I-\pr_\gamma $, and
$\pr_{\gamma '}=(A^*_\gamma )^{-1}\Ama'$, $\pr_{\zeta '}=I-\pr_{\gamma '}$.
By $\pr_Vu=u_V$ we denote the {\it orthogonal projection} of $u$ onto
a subspace $V$.

The following ``abstract Green's formula'' holds for $u\in D(\Ama)$,
$v\in D(\Ama')$:
\begin{equation}(A u,v) - (u, A' v)=((A u)_{Z'}, v_{\zeta '})-(u_\zeta ,
(A 'v)_Z).\label{tag1.1}
\end{equation}
It can be used to show that when $\wA\in \mathcal M$ and we set
$W=\overline{\pr_{\zeta '} D(\widetilde A^*)}$, then
\[
\{\{u_\zeta , (A u)_W\}\mid u\in D(\wA)\}\text{ is a graph.} 
\]
Denoting
the operator with this graph by $T$, we have:

\begin{theorem}\label{Theorem1.1} {\rm {\cite{Grubb68}}} For the closed $\wA\in\mathcal M$,
there is a {\rm 1--1} correspondence
\[
\wA \text{ closed } \longleftrightarrow \begin{cases} T:V\to W,\text{ closed, densely
defined}\\
\text{with }V\subset Z,\; W\subset Z',\text{ closed subspaces.} \end{cases}
\]
\noindent Here $D(T)=\pr_\zeta  D(\widetilde A)$, $V=\overline{D(T)}$, $
W=\overline{\pr_{\zeta '} D(\widetilde A^*)}$, and
\begin{equation}
Tu_\zeta=(Au)_W \text{ for all }u\in D(\wA),\text{ (the {\bf defining
equation})}.\label{tag1.2}
\end{equation}
In this correspondence,

{\rm (i)} $\wA^*$ corresponds similarly to $T^*:W\to V$.

{\rm (ii)} $\operatorname{ker}\wA=\operatorname{ker}T$;
\quad $\operatorname{ran}\wA=\operatorname{ran}T+(H\ominus W)$.

{\rm (iii)} When $\wA$ is invertible, 
\[\wA^{-1}=A_\gamma
^{-1}+\inj_{V }T^{-1}\pr_W.\]

\end{theorem}

Here $\inj_{V }$ indicates the injection of $V$ into $H$ (it is
often left out). 

Now provide the operators with a spectral parameter $\lambda $, then this
implies, with
\begin{align}
Z_\lambda& =\operatorname{ker}(\Ama-\lambda ), \quad Z'_{\bar\lambda }=\operatorname{ker}(\Ama'-\bar\lambda ),\nonumber\\
D(\Ama)&=D(A_\gamma )\dot+ Z_\lambda ,  \quad u=u^\lambda _\gamma +u^\lambda _\zeta
=\pr^\lambda _\gamma u+\pr^\lambda _\zeta u,\text{ etc.:}\nonumber
\end{align} 

\begin{corollary}\label{Corollary1.2} Let $\lambda \in \varrho (A_\gamma )$. For the
closed $\wA\in\mathcal M$,  there is a {\rm 1--1} correspondence
\[
\wA -\lambda  \longleftrightarrow \begin{cases} T^\lambda :V_\lambda \to W_{\bar\lambda },\text{ closed, densely
defined}\\
\text{with }V_\lambda \subset Z_\lambda ,\; W_{\bar\lambda }\subset Z'_{\bar\lambda },\text{ closed subspaces.} \end{cases}
\]
\noindent Here $D(T^\lambda )=\pr^\lambda _\zeta  D(\widetilde A)$, $V_\lambda =\overline{D(T^\lambda )}$, $
W_{\bar\lambda }=\overline{\pr^{\bar\lambda }_{\zeta '} D(\widetilde A^*)}$, and
\[
T^\lambda u^\lambda _\zeta=((A-\lambda )u)_{W_{\bar\lambda }} \text{ for all }u\in D(\wA).\]

 Moreover,

{\rm (i)} $\operatorname{ker}(\wA-\lambda )=\operatorname{ker}T^\lambda $;
\quad $\operatorname{ran}(\wA-\lambda )=\operatorname{ran}T^\lambda +(H\ominus W_{\bar\lambda })$.

{\rm (ii)} When $\lambda \in \varrho (\wA)\cap \varrho (A_\gamma )$, 
\begin{equation}(\wA-\lambda
)^{-1}=(A_\gamma -\lambda )
^{-1}+\inj_{V_\lambda  }(T^\lambda )^{-1}\pr_{W_{\bar\lambda
  }}.\label{tag1.3}
\end{equation}

\end{corollary}

This gives a Kre\u\i{}n-type resolvent formula for any closed
$\wA\in\mathcal M$.

The operators $T$ and $T^\lambda $ are related in the following way: Define
\begin{align}
E^\lambda &=I+\lambda (A_\gamma -\lambda )^{-1},\quad
F^\lambda =I-\lambda A_\gamma ^{-1},\nonumber\\
E^{\prime\bar\lambda }&=I+\bar\lambda (A^*_\gamma -\bar\lambda )^{-1},\quad
F^{\prime\bar\lambda }=I-\bar\lambda (A^*_\gamma )^{-1},\nonumber
\end{align}
then $E^\lambda F^\lambda =F^\lambda E^\lambda =I$,
$E^{\prime\bar\lambda }F^{\prime\bar\lambda }=F^{\prime\bar\lambda
}E^{\prime\bar\lambda }=I$ on $H$. Moreover, $E^\lambda $ and
$E^{\prime\bar\lambda }$ restrict to 
homeomorphisms 
\[
E^\lambda _V: V\simto V_\lambda ,\quad E^{\prime\bar\lambda
}_W:W\simto W_{\bar\lambda },\]
with inverses denoted $F^\lambda _V$ resp.\ $F^{\prime\bar\lambda }_W$.
In particular, $D(T^\lambda )=E^\lambda _V D(T)$.

\begin{theorem}\label{Theorem1.3}
 Let $
G^\lambda _{V,W}=-\pr_W\lambda
E^\lambda \inj_{V }$; then 
\begin{equation}(E^{\prime\bar\lambda }_W)^* T^\lambda E^\lambda _V=T+G^\lambda
_{V,W}.\label{eq:G}
\end{equation}
In other words, $T$ and $T^\lambda $ are related by the
commutative diagram 
\begin{equation}
\CD
V_\lambda     @<\sim<{E^\lambda _{V}} <  V    \\
@V  T^\lambda VV           @VV  T+G^\lambda _{V,W} V\\
W_{\bar\lambda }   @>\sim >(E^{\prime\bar\lambda }_W)^* >   W
\endCD \hskip2cm D(T^\lambda )=E^\lambda _V D(T).\label{1.2a}
 \end{equation}
\end{theorem}
 
This is a straightforward elaboration of \cite[Prop.\ 2.6]{Grubb74}. 

It was shown in  \cite{BGW09} how this relates to formulations in
terms of $M$-functions. First there is the following result in the
case where $V=Z$, $W=Z'$, i.e., $\pr_\zeta
D(\widetilde A)$ is dense in $Z$ and $\pr_{\zeta '}  D(\widetilde A^*)$
is dense in $Z'$:

\begin{theorem}\label{Theorem1.4}  
Let $\wA$ correspond to $T:Z\to Z'$ by
Theorem {\rm \ref{Theorem1.1}}. There is a holomorphic operator family
$M_{\wA}(\lambda )\in\mathcal L(Z',Z)$  defined for  $\lambda \in \varrho (\wA)$ by
\[
M_{\wA}(\lambda )=\pr_\zeta (I-(\wA-\lambda )^{-1}(\Ama-\lambda ))A_\gamma
^{-1}\inj_{Z' },\]
\noindent Here $M_{\wA}(\lambda )$ relates to $T$ and $T^\lambda $
by:
\begin{equation}M_{\wA}(\lambda )=-(T+G^\lambda _{Z,Z'})^{-1}=
- F_Z^\lambda(T^\lambda )^{-1}(F_{Z'}^{\prime\bar\lambda })^*\text{, for }\lambda \in \varrho (\wA)\cap
\varrho (A_\gamma ).\label{tag1.4}
\end{equation}
\end{theorem}

This is directly related to $M$-functions (Weyl-Titchmarsh functions) 
introduced by other authors,
see details and references in  \cite{BGW09}.
Moreover, the construction extends in a natural way to all the closed
$\wA\in\mathcal M$, giving the following result:

\begin{theorem}\label{Theorem1.5}   Let $\wA$ correspond to $T:V\to W$ by Theorem
{\rm \ref{Theorem1.1}}. For any
$\lambda \in \varrho (\wA)$, there is a well-defined
$M_{\wA}(\lambda )\in \mathcal L(W,V)$, holomorphic in $\lambda $
and satisfying

{\rm (i)} $M_{\wA}(\lambda )=\pr_\zeta (I-(\wA-\lambda )^{-1}(\Ama-\lambda ))A_\gamma
^{-1}\inj_{W }.$

{\rm (ii)} When $\lambda \in \varrho (\wA)\cap
\varrho (A_\gamma )$, 
\[
M_{\wA}(\lambda )=-(T+G^\lambda _{V,W})^{-1}.
\]

{\rm (iii)} For $\lambda \in \varrho (\wA)\cap
\varrho (A_\gamma )$, it enters in a Kre\u\i{}n-type resolvent formula  
\begin{equation}
(\wA-\lambda
)^{-1}=(A_\gamma -\lambda )
^{-1}-\inj_{V_\lambda  }E^\lambda _VM_{\wA}(\lambda
)(E^{\prime\bar\lambda }_W)^*\pr_{W_{\bar\lambda }}.\label{tag1.5a}
\end{equation}

\end{theorem}

Other Kre\u{\i}n-type resolvent formulas in a general framework of
{\it relations} can be found e.g.\  in Malamud and Mogilevski\u\i{} 
\cite[Section 5.2]{MM02}.

\setcounter{section}{3}
\section{The resolvent construction}\label{resolv}

\subsection{Realizations}\label{sec:Realizations}

The abstract extension theory in the preceding section was implemented
for boundary value problems for elliptic operators $A$ with smooth
coefficients on  smooth domains $\Omega $ in \cite{Grubb68}--\cite{Grubb74},
with further results worked out in \cite{BGW09} on Kre\u\i{}n resolvent
formulas and $M$-functions. Our aim in the present paper is to extend
the validity to the nonsmooth situation introduced in Section
\ref{subsec:Green}. An important ingredient in this is to show that the Dirichlet problem for
$A$ has a resolvent and a Poisson solution operator with appropriate
mapping properties. 

As $\Ami$, $\Ami'$, $\Ama$ and $\Ama'$ we take the operators in
$L_2(\Omega )$ defined by
\begin{align}
&\Ami\text{ resp.\ }\Ami' =\text{ the closure of }A|_{C_0^\infty
  (\Omega )}\text{ resp.\ }A'|_{C_0^\infty (\Omega )}.\label{tag2.2}\\
&\Ama=(\Ami')^*,\quad \Ama'=(\Ami)^*.\nonumber
\end{align} 
Then $\Ama$ acts like $A$ with domain consisting of the functions
$u\in L_2(\Omega)$ such that $ Au$, defined weakly, is in $L_2(\Omega )$.
$\Ama'$ is defined similarly from $A'$.

By extension of the coefficients $a_{jk}$, $a_j$, to all of $\R^n$
(preserving the degree of smoothness) we can extend $A$ to a uniformly
strongly elliptic operator $A_e$ on
$\R^n$; by addition of a constant, if necessary, we can assume that it
has a positive lower bound. By a variant of the resolvent construction
described below 
(easier here, since there is no boundary) we get unique solvability of the
equation
$A_eu=f$ on $\R^n$ with $f\in L_2(\R^n)$, with a solution $u\in
H^{2}(\R^n)$. Then 
the graph-norm $(\|Au\|^2_{L_2(\Omega )}+\|u\|^2_{L_2(\Omega )})^{\frac12}$ and
the $H^{2}$-norm are equivalent on $H^{2}_0(\Omega )$, so 
\begin{equation}
D(\Ami)= H^{2}_0(\Omega ).
\label{tag2.2a}
\end{equation}

As $A_\gamma $ we take the 
Dirichlet realization of $A$;
it is the restriction of $\Ama$ with
domain
\[
D(A_\gamma )= D(\Ama)\cap H^1_0(\Omega ).
\]
and equals the operator defined by
variational theory (Lions' version of the Lax-Milgram lemma, the notation used here is
as in \cite{G09}, Ch.\ 12), applied to the sesquilinear
form 
\begin{equation}\label{sesqform}
a(u,v)=\Sigma _{j,k=1}^n (a_{jk}D_{x_k}u, D_{x_j}v)+(\Sigma _{j=1}^n a_{j}iD_{x_j}u+a_0u, v),
\end{equation}
 with
domain $H^1_0(\Omega )\subset L_2(\Omega )$.
$A_\gamma $ also has positive lower bound.
The analogous operator for $A'$ is its Dirichlet realization
$A'_\gamma $; it equals the adjoint of $A_\gamma $. The inequality
\eqref{str.ell} implies that the principal symbol takes its values in
a sector \linebreak
$\{\lambda \in \C: |\arg \lambda |\le \pi/2 -\delta \}$ with $\delta
>0$. The resolvent $(A_\gamma -\lambda )^{-1}$ is well-defined and $O(\ang\lambda ^{-1})$ for large
$|\lambda |$ on the rays $\{re^{i\eta }\}$ with $\eta \in (\pi /2-\delta ,3\pi /2+\delta $).

The linear
operators $\wA$ with $\Ami\subset \wA\subset \Ama$ are the
realizations of $A$. 

In the detailed study of the Dirichlet problem that now follows, we first treat a half-space case by pseudodifferential methods, and
then use this to treat the general case by localization.

\subsection{The halfspace case}
\label{subsec:halfspace}

In this subsection, we consider the case of a uniformly strongly elliptic
second order operator $a(x,D_x)$ on $\rnp$ in $x$-form (i.e., defined
from $\sum_{j,k=1}^n a_{jk}(x)\xi _j\xi _k$ by the formula
\eqref{eq:x-form}, not in 
divergence form). More precisely, we assume that
 $a(x,D_x)u:= \sum_{j,k=1}^n a_{jk}(x)D_{x_j}D_{x_k} u$, where
 $a_{jk}\in C^\tau(\overline{\R}^n_+)$ for some $0<\tau\leq 1$. The
 case of a general domain will be treated by the help of this
 situation, using that
 $H^{\frac12}_p(\R^{n-1})\hookrightarrow C^\tau(\R^{n-1})$ and
 $W^1_{(2,p)}(\Rn_+)\hookrightarrow C^\tau(\overline{\R}^n_+)$, where
 $\tau= \frac12-\frac{n-1}p$, cf. (\ref{eq:W2pEmbedding}). For the
 construction of a parametrix on $\Rn_+$ it will be enough to use the
 $C^\tau$-regularity of the coefficients.

We define
\begin{equation}
\mathcal A^{\times}: =\begin{pmatrix} a(x,D_x)\\\gamma _0\end{pmatrix} \colon
H^{s+2}(\rnp) \to
\begin{matrix}
H^{s}(\rnp)\\ \times \\ { H}^{s+\frac32}({\mathbb R}^{n-1})\end{matrix},\label{tag4.2'}
\end{equation}
which maps continuously for all $|s|< \tau$,
since $a(x,D_x)\colon H^{s+2}(\Rn_+)\to H^s(\Rn_+)$ for all $|s|<
\tau$ by Proposition~\ref{prop:HsBddness}. 

To prepare for an application of
 Theorem \ref{Theorem3.1}, 
we apply order-reducing operators (cf.\
 Remark \ref{Remark3.2}) to reduce to $H^s$-preserving operators, introducing
 \begin{equation}
\mathcal A_1=\begin{pmatrix}I&0\\0&\Lambda _0^{\frac32} \end{pmatrix}\mathcal A^{\times}
\Lambda _{-,+}^{-2}=
\begin{pmatrix} a(x,D_x)\Lambda _{-,+}^{-2}\\ \quad\\\Lambda
  _0^{\frac32} \gamma _0 \Lambda _{-,+}^{-2}\end{pmatrix} \colon 
H^{s}(\rnp) \to
\begin{matrix}
H^{s}(\rnp)\\ \times \\ H^{s}({\mathbb R}^{n-1})\end{matrix}, \label{tag4.4}
\end{equation}
continuous for $|s|<\tau$; it is again in $x$-form with $C^\tau $-smoothness in
$x$. Since $\gamma _0$ is of class 1, $\Lambda ^{\frac32}_0\gamma _0\Lambda ^{-2}_{-,+}$
is of class $-1$ (and order $-\frac12$); $a(x,D_x)\Lambda ^{-2}_{-,+}$
is in fact of class
$-2$. (The notion of {\it class} is recalled at the end of Section \ref{sec:symbolsmoothing} and
extended to negative values in Remark \ref{Remark3.2}.)

 By Theorem \ref{Theorem3.1}
$2^\circ$, 
${\mathcal A}_1$ has a parametrix
$\mathcal B^0_1$ in $x$-form, of class $-1$, defined from the
inverse symbol;
\begin{equation}
\mathcal B^0_1=\begin{pmatrix} R^0_1& K^0_1\end{pmatrix} \colon \begin{matrix}
H^{s}(\rnp)\\ \times \\ H^{s}({\mathbb R}^{n-1})\end{matrix} \to
H^{s}(\rnp),\label{tag 4.5}\end{equation}
continuous for $|s|< \tau$. (We omit the class related condition
$s>-\frac32$, since $\tau \le 1$.) In particular, $R^0_1$ is of order 0, and
$K^0_1$ is a Poisson operator of order
$\frac12$, having symbol-kernel in $C^\tau S^{-\frac12}_{1,0}(\R^N\times\R^{n-1},\SD(\ol{\R}_+))$.
The remainder $\mathcal R_1=\mathcal A_1\mathcal B^0_1-I$
satisfies
\begin{equation}
\mathcal R_1\colon  \begin{matrix}
H^{s-\theta }(\rnp)\\ \times \\ H^{s-\theta }({\mathbb
R}^{n-1})\end{matrix} \to
\begin{matrix}
H^{s}(\rnp)\\ \times \\ H^{s}({\mathbb R}^{n-1})\end{matrix}
,\label{tag4.6}
\end{equation}
when $0<\theta <\tau$,
\begin{equation}
-\tau  +\theta <s<\tau ,\quad s > -\tfrac32+\theta .\label{tag4.7}
\end{equation}

Then the equation $\mathcal A_1\mathcal B_1^0=I+\mathcal R_1$, also written
\[
\begin{pmatrix} I&0\\0&\Lambda_0^{\frac32} \end{pmatrix}\mathcal A^{\times}\Lambda _{-,+}^{-2}\mathcal B_1^0=I+\mathcal R_1,
\]
implies by composition to the left with $\begin{pmatrix} I&0\\0&\Lambda_0^{-\frac32}\end{pmatrix}$ and to the right with $\begin{pmatrix} I&0\\0&\Lambda_0^{\frac32} \end{pmatrix}$:
\begin{equation}
\mathcal A^{\times}\Lambda _{-,+}^{-2}\mathcal B_1^0
\begin{pmatrix} I&0\\0&\Lambda_0^{\frac32} \end{pmatrix}
=I+\mathcal R',\text{ with } \mathcal R'=\begin{pmatrix} I&0\\0&\Lambda_0^{-\frac32}\end{pmatrix}\mathcal R_1\begin{pmatrix} I&0\\0&\Lambda_0^{\frac32} \end{pmatrix}.\nonumber
\end{equation}
Hence
\begin{equation} \mathcal B^0
=\Lambda _{-,+}^{-2}\mathcal B_1^0\begin{pmatrix} I&0\\0&\Lambda_0^{\frac32}
\end{pmatrix}
=\begin{pmatrix}\Lambda _{-,+}^{-2} R_1^0 &\Lambda _{-,+}^{-2} K_1^0\Lambda_0^{\frac32}
\end{pmatrix}
=\begin{pmatrix} R^0& K^0\end{pmatrix} \label{tag4.7a}
\end{equation}
is a parametrix of the $x$-form operator $\mathcal A^{\times}$, with
\begin{gather}
\mathcal A^{\times} \mathcal B^0=I+\mathcal R',\label{tag4.12a}\\
\mathcal B^0 \colon \begin{matrix}
H^{s}(\rnp)\\ \times \\ {H}^{s+\frac32}({\mathbb R}^{n-1})\end{matrix} \to
H^{s+2}(\rnp),\quad\mathcal R'\colon  \begin{matrix}
H^{s-\theta }(\rnp)\\ \times \\ { H}^{s+\frac32-\theta }({\mathbb
R}^{n-1})\end{matrix} \to
\begin{matrix}
H^{s}(\rnp)\\ \times \\  { H}^{s+\frac32}({\mathbb R}^{n-1})\end{matrix}
,\label{tag4.6a}
\end{gather}
for $s$ and $\theta $ as in \eqref{tag4.7}. 

\subsection{General domains}\label{subsec:gendom}

Now we consider the situation where the
domain $\Omega$ and the differential operator $A$ are as in Assumption
\ref{tau-assumption2}. From now on, we use the notation
$\partial\Omega =\Sigma $. We recall that the assumption implies that
$\Sigma $ is $ 
B^{\frac32}_{p,2}$, $\frac12-\frac{n-1}p>0$, the principal part of
$A$ is 
in divergence form with $H^1_q(\Omega)$-coefficients, and $\tau= \frac12 -\frac{n-1}p\leq 1-\frac{n}q$.
We have the
direct operator with $A$ as in \eqref{eq:DefnA} 
\begin{equation}
\mathcal A=\begin{pmatrix} A\\  \gamma_0 \end{pmatrix} \colon 
H^{s+2}(\Omega ) \to
\begin{matrix}
H^{s}(\Omega )\\ \times \\ { H}^{s+\frac32}(\Sigma )\end{matrix};\label{tag4.2}
\end{equation}
it is continuous for $-\tfrac32<s\le 0$.

First we replace  the differential operator $A$ 
by its principal part in $x$-form, namely
$$
a(x,D_x)u:= \sum_{j,k=1}^n a_{jk}(x)D_{x_j}D_{x_k} u.
$$
 Then $a(x,D_x)$ has $C^\tau$-coefficients since $H^1_q(\Omega)\hookrightarrow C^\tau(\overline{\Omega})$, and we have:
\begin{equation}\label{eq:xyForm}
  A-a(x,D_x)\colon H^{s+2-\theta}(\Omega)\to H^s(\Omega)
\end{equation}
for all $-1< s\leq 0$ 
and $0<\theta<\min (\tau, s+1)$. 
This statement follows from Theorem~\ref{thm:PsDOCompo} applied to the principal part (see also \eqref{eq:Reduction1}, \eqref{eq:Reduction2}) and from (\ref{eq:a0Estim}) since $\frac{n}q\leq 1-\tau$.

 Let $\mathcal{A}^{\times}$ be the operator obtained from $\mathcal{A}$ by replacing $A$ by $a(x,D_x)$.  Then $\mathcal{A}^{\times}$ has the mapping properties 
\begin{equation}
\mathcal A^{\times}=\begin{pmatrix} a(x,D_x)\\  \gamma_0 \end{pmatrix} \colon 
H^{s+2}(\Omega) \to
\begin{matrix}
H^{s}(\Omega)\\ \times \\ { H}^{s+\frac32}(\Sigma)\end{matrix}\label{tag4.2''}
\end{equation}
continuously for
$|s| \le  1$, 
since $a(x,D_x)\colon H^{s+2}(\Omega)\to H^s(\Omega)$ for all $|s|\leq
1$ in view of (\ref{eq:BesselProd}) and the fact that $a_{jk}\in H^1_q(\Omega)$. 

We shall use a system of local coordinates and cutoff functions as introduced in Remark
\ref{rem:Charts}, with $M=2$.

When the  differential operator $A$ is transformed to local
coordinates, the principal part of the
resulting operator $\underline A$ is an $x$-form operator with 
$C^\tau$-coefficients since $H^1_q(\R^n_+)\hookrightarrow 
C^\tau(\crnp)$. More precisely,  because of Corollary~\ref{cor:CoordinateTrafo},
\begin{equation*}
  F_j^{\ast}a(x,D_x)F_j^{-1,\ast}=\underline{a}_j(x,D_x) + \mathcal{R},
\end{equation*}
where
$
 \mathcal{R} \colon H^{2-\tau}(\R^n_+)\to L^2(\R^n_+)
$
and
\begin{equation*}
  \underline{a}_j(x,\xi)= \sum_{k,l=1}^n a_{kl}(F_j(x)) (\Phi 
_j(x)\xi)^{\alpha+\beta},
\end{equation*}
with $\Phi _j(x)= (\nabla F_j(x))^{-1}\in
W^{1}_{(2,p)}(\R^n_+)^{n^2}\hookrightarrow
C^\tau(\crnp)^{n^2}$. Hence $\underline a_j(x,D_x)$ has coefficients
in $C^\tau(\crnp)$. 

In each of these charts
one constructs a parametrix $\underline{\mathcal
  B}^0_j=\begin{pmatrix}R^0_j& K^0_j\end{pmatrix}$ for $\begin{pmatrix}
\underline{a}_j(x,D_x)\\\gamma _0\end{pmatrix}$ as in Section \ref{subsec:halfspace} (the
coefficients of $\underline A$ can
be assumed to be extended to $\crnp$ preserving ellipticity); for
$U_0$ which is disjoint from the boundary one takes a parametrix
$R^0_0$ of $a(x,D_x)$. Then one
defines $\mathcal{B}^0=\begin{pmatrix}R^0 & K^0\end{pmatrix}$
by
\begin{eqnarray}\label{eq:R0Parametrix}{}
R^0 f &=& \sum_{j=1}^J \psi_j F^{-1,\ast}_j R^0_j F^{\ast}_j \varphi_j
f+\psi _0 R^0_0\varphi _0f\\\label{eq:K0Parametrix}   
K^0 g &=& \sum_{j=1}^J \psi_j F^{-1,\ast}_j K^{0}_j F^{\ast}_{j,0} \varphi_j g
\end{eqnarray}
for all $f\in H^s(\Omega)$, $ g\in H^{s+\frac32}(\Sigma)$, where
$-\tau +\theta<s\leq 0$ and $\varphi_j,\psi_j$ as in Remark \ref{rem:Charts}. 
Here we recall from \eqref{tag4.7a} that
$K^0_j=\Lambda ^{-2}_{-,+}\tilde{k}^{0}_j(x',D_x)\Lambda
^{\frac32}_0$ with $\tilde k^0_j\in C^\tau
S^{-\frac12}_{1,0}(\R^{n-1}\times\R^{n-1}, \mathcal{S}(\overline{\R}_+))$.
 Then it follows
directly from the results so far that 
\begin{eqnarray*}
  \mathcal{A}\mathcal{B}^0
  \begin{pmatrix}
    f\\ g
  \end{pmatrix}
  = \begin{pmatrix}
    f\\ g
  \end{pmatrix}+ \mathcal{R}_1\begin{pmatrix}f\\ g \end{pmatrix},
\end{eqnarray*}
where 
\begin{equation*}
  \mathcal{R}_1\colon
  \begin{array}{c}
H^{s-\theta}(\Omega)\\ \times \\ H^{s+\frac32-\theta}(\Sigma)    
  \end{array}
\to 
\begin{array}{c}
H^{s}(\Omega)\\ \times\\ H^{s+\frac32}(\Sigma)   
\end{array}
\end{equation*}
is a bounded operator 
for all  $\theta$ and $s\in\R$ such that 
\begin{equation}
0<\theta<\tau,\quad -\tau +\theta <s\le 0. 
\label{tag4.15a}
\end{equation}

In the present construction, we shall actually
carry a spectral parameter along, which
will be useful for discussions of invertibility and resolvents. So we now
replace the originally given $A$ by $A-\lambda $, to be studied for
large $\lambda $ in a sector around $\R_-$. The parameter is taken
into the order-reducing operators as well, by replacing $\ang\xi
=(1+|\xi |^2)^{\frac12}$ by $(1+|\lambda |+|\xi |^2)^{\frac12}$.

The parametrix
will be of the form
\begin{equation}
\mathcal B^0(\lambda )=\begin{pmatrix} R^0(\lambda )& K^0(\lambda )\end{pmatrix} \colon \begin{matrix}
H^{s}(\Omega )\\ \times \\{ H}^{s+\frac32}(\Sigma )\end{matrix} \to
H^{s+2}(\Omega );\label{tag4.14}\end{equation}
with $H^1_q$-smoothness in $x$, where $-\tau <s\leq 0$. 
The remainder maps as follows:
 \begin{equation}
\mathcal R(\lambda )=\mathcal A(\lambda )\mathcal B^0(\lambda )-I\colon  \begin{matrix}
H^{s-\theta }(\Omega )\\ \times \\ H^{s+\frac32-\theta }(\Sigma )
\end{matrix} \to
\begin{matrix}
H^{s}(\Omega )\\ \times \\ { H}^{s+\frac32}(\Sigma )\end{matrix} \label{tag4.15}
\end{equation}
for $s$ and $\theta $ as in \eqref{tag4.15a}.

In order to get hold of exact inverses, we shall use a variant of an old trick
of Agmon \cite{A62}, which implies a useful
$\lambda $-dependent estimate of the remainder. (The technique was
developed further and applied to $\psi $dbo's in
\cite{FunctionalCalculus}, which could also be invoked here; but in
the present simple case of
differential operators the trick can be used more directly.) 

Consider $\lambda $ on a ray outside the sector where the principal
symbol $\sum_{j,k=1}^n a_{jk}(x)\xi _j\xi
  _k$ takes its values, i.e., we set $\lambda =e^{i\eta }\mu
^{2}$ ($\mu \ge 0$) with 
 $\eta \in (\pi /2-\delta , 3\pi /2+\delta)$. For the study of $A-\lambda $,
introduce an extra variable $t\in S^1$, and replace
$\mu $ by $D_t=-i\partial_t$, letting
\begin{equation}
\widehat A=A-e^{i\eta }D_t^{2} \text{ on }\Omega \times S^1. \label{4.15b}
\end{equation}
Then $\widehat A$ is elliptic on $\Omega \times S^1$ and its
Dirichlet problem is elliptic, and by
the
preceding construction (carried out with local coordinates respecting
the product structure),
\[
\widehat {\mathcal A} =\begin{pmatrix} \widehat A\\ \gamma _0
\end{pmatrix}\text{ has a parametrix }\widehat{\mathcal B}^0 ,
\]
with mapping properties of $\widehat{\mathcal B}^0 $ and the remainder
$\widehat{\mathcal R}=\widehat{\mathcal A} \widehat{\mathcal B}^0 -I$
as in
\eqref{tag4.14} and \eqref{tag4.15}
with $\Omega ,\Sigma $ replaced by $\widehat\Omega =\Omega \times
S^1$, $\widehat\Sigma =\Sigma \times
S^1$.

For functions $w$ of the form $w(x,t)=u(x)e^{i\mu t}$,
\[
\widehat {\mathcal A} w=\begin{pmatrix} ( A-e^{i\eta }\mu ^{2})w\\ \gamma _0
w\end{pmatrix}=\begin{pmatrix} ( A-\lambda )w\\ \gamma _0
w\end{pmatrix},
\]
and similarly, the parametrix $\widehat{\mathcal B}^0 $ and the remainder
$\widehat{\mathcal R}$ act on such functions like 
$\mathcal B^0 (\lambda )$ and
$\mathcal R(\lambda )$ applied in the $x$-coordinate.

Moreover, for  $w(x,t)=u(x)e^{i\mu t}$, $u\in \mathcal S({\mathbb R}^n)$, $\mu\in 2\pi \Z$, 
\[
\|w\|_{H^s({\mathbb R}^n\times S^1)}\simeq \|(1-\Delta +\mu ^2
)^{s/2}u(x)\|_{L_2({\mathbb R}^n)}\simeq \|(1+|\xi |^2+\mu ^2)^{s/2}\hat u(\xi )\|_{L_2},
\]
with similar relations for Sobolev spaces over other sets. Norms as in
the right-hand side are called $H^{s,\mu }$-norms; they were
extensively used in \cite{FunctionalCalculus}, see the Appendix there for
the definition on subsets. For the parametrix 
$\mathcal B^0 (\lambda)$  this implies
\begin{equation}
\|\mathcal B^0(\lambda )\{f,g\}\|_{H^{s+2,\mu }(\Omega )}
\le c_s\|\{f,g\}\|_{H^{s ,\mu }(\Omega )\times
{ H}^{s+\frac32 ,\mu }(\Sigma )}.\label{tag4.18a}
\end{equation}
The important observation is now that when $s'<s$ and
$w(x,t)=u(x)e^{i\mu t}$, then
\begin{align}
\|w\|_{H^{s'}({\mathbb R}^n\times S^1)}&\simeq \|(1+|\xi |^2+\mu
^2)^{s'/2}\hat u(\xi )\|_{L_2}\nonumber \\
&\le  \ang\mu ^{s'-s}\|(1+|\xi |^2+\mu ^2)^{s/2}\hat u(\xi )\|_{L_2}\simeq\ang\mu ^{s-s'}  \|w\|_{H^s({\mathbb R}^n\times S^1)},\nonumber
\end{align}
with constants independent of $u$ and $\mu $. Analogous estimates hold with
${\mathbb R}^n$ replaced by $\Omega $ or $\Sigma $.

Applying this principle to the estimates of the remainder $\widehat{\mathcal R}$, we find
that
\begin{align}
\|\mathcal R(\lambda )\{f,g\}\|_{H^{s,\mu }(\Omega )\times { H}^{s+\frac32,\mu
}(\Sigma )}
&\le c_s\|\{f,g\}\|_{H^{s-\theta ,\mu }(\Omega )\times
{ H}^{s+\frac32-\theta ,\mu }(\Sigma )}\nonumber
\\ &\le c'_s\ang\mu ^{-\theta }\|\{f,g\}\|_{H^{s,\mu }(\Omega )\times
{ H}^{s+\frac32,\mu }(\Sigma )}
\label{tag4.18}
\end{align}
for $s$ as in \eqref{tag4.15a}, $\lambda =e^{i\eta }\mu ^2 $ with $\mu
\in 2\pi {\mathbb N}_0$. One way to extend the
observation to 
arbitrary
$\lambda $ on the ray, is to write $\lambda =e^{i\eta }\mu ^2 =
e^{i\eta }(\mu _0+\mu ')^2$ with
$\mu _0\in 2\pi {\mathbb N}_0 $, $\mu '\in [0,2\pi )$, and set  $\lambda _0=
e^{i\eta }\mu _0^2$. Using \eqref{tag4.18a} and observing that
$(1+|\xi |^2+\mu _0^2)^{t/2}\simeq (1+|\xi |^2+(\mu _0+\mu ')^2)^{t/2} $
uniformly in $\xi ,\mu _0,\mu '$, by elementary inequalities, we find for
\begin{align}\mathcal R(\lambda )=\mathcal A(\lambda ) \mathcal B^0(\lambda _0)-I=\mathcal A (\lambda _0)\mathcal
  B^0(\lambda _0)-I+\begin{pmatrix}\lambda _0-\lambda \\0\end{pmatrix}\mathcal
  B^0(\lambda _0)
\end{align}
that
\begin{align*}
&\|\mathcal R(\lambda )\{f,g\}\|_{H^{s,\mu }(\Omega )\times { H}^{s+\frac32,\mu
}(\Sigma )}\le c''_s\|\mathcal R(\lambda )\{f,g\}\|_{H^{s,\mu _0}(\Omega
)\times { H}^{s+\frac32,\mu _0
}(\Sigma )}\\
&\le c'''_s\|\{f,g\}\|_{H^{s-\theta ,\mu _0}(\Omega )\times
{ H}^{s+\frac32-\theta ,\mu _0}(\Sigma )}
 \le c''''_s\|\{f,g\}\|_{H^{s-\theta ,\mu }(\Omega )\times
{ H}^{s+\frac32-\theta ,\mu }(\Sigma )}.
\end{align*}
So \eqref{tag4.18} also holds for general $\lambda $, when we define
${\mathcal B}^0(\lambda )={\mathcal B}^0(\lambda _0)$.

For each $s$, consider $\lambda= e^{i\eta }\mu^2 $ with $\mu\ge \mu _1$,
where $\mu _1$ is taken so large that
$c'_s\ang\mu ^{-\theta }\le \frac12$ for $\mu \ge \mu _1$. Then $I+\mathcal R(\lambda )$ has the
inverse $I+\mathcal R '(\lambda )=I+\sum_{k\ge 1}(-\mathcal R(\lambda ))^k$
(converging in the operator norm for operators on $H^{s,\mu }(\Omega
)\times { H}^{s+\frac32,\mu }(\Sigma )$), and, by definition of ${\mathcal B}^0(\lambda )$, 
\[
\mathcal A (\lambda )\mathcal B ^0(\lambda )(I+\mathcal R' (\lambda ))=I.
\]
This gives a right inverse
\[
\mathcal B (\lambda )=\mathcal B ^0(\lambda )+\mathcal B^0(\lambda )\mathcal R'(\lambda )=\begin{pmatrix}
R (\lambda )& K (\lambda )\end{pmatrix},
\]
with the same Sobolev space continuity \eqref{tag4.18a} as $\mathcal B ^0(\lambda )$, and
$\mathcal B ^0(\lambda )\mathcal R'(\lambda )$ of lower order: 
\begin{align}
\|\mathcal B^0(\lambda )\mathcal R'(\lambda )\{f,g\}\|_{H^{s+2,\mu }(\Omega )}
&\le c_s\|\{f,g\}\|_{H^{s-\theta  ,\mu }(\Omega )\times
{ H}^{s-\theta +\frac32 ,\mu }(\Sigma )}\nonumber
\\ &\le c'_s\ang\mu ^{-\theta }\|\{f,g\}\|_{H^{s  ,\mu }(\Omega )\times
{ H}^{s +\frac32 ,\mu }(\Sigma )}.\label{tag4.18b}
\end{align}

 Since
\begin{equation}
\mathcal A (\lambda )\mathcal B (\lambda )=\begin{pmatrix} (A-\lambda )R (\lambda
)& (A-\lambda )K (\lambda )\\ \gamma _0R (\lambda
)& \gamma _0 K (\lambda )
\end{pmatrix}=\begin{pmatrix} I&0\\0&I\end{pmatrix},\label{tag4.19}
\end{equation}
$R (\lambda )$ solves
\begin{equation}
(A-\lambda )u=f,\quad \gamma _0u=0, \label{tag4.20}
\end{equation}
and $K (\lambda )$ solves
\begin{equation}
(A-\lambda )u=0,\quad \gamma _0u=\varphi .\label{tag4.21}
\end{equation}

Since $R(\lambda )$ maps $L_2(\Omega )$ into $H^2(\Omega )\cap
H^1_0(\Omega )\subset D(A_\gamma )$, it must coincide with 
the resolvent $(A_\gamma -\lambda )^{-1}$
of $A_\gamma $ defined in Section \ref{sec:Realizations} by variational theory. The operator $K(\lambda )$ is the
Poisson-type solution operator of the Dirichlet problem with zero
interior data; it is often denoted by $K^\lambda _\gamma $ and
we shall also use this notation here. 
The operators have the mapping properties, for each $\lambda =e^{i\eta
}\mu ^2$, $\mu \ge \mu _1$, 
\begin{equation}
(A_\gamma -\lambda )^{-1}\colon H^s(\Omega )\to H^{s+2}(\Omega
),\quad K^\lambda _\gamma \colon { H}^{s+\frac32}(\Sigma )\to H^{s+2}(\Omega ),
\label{tag4.26}\end{equation}
for $s$ satisfying \eqref{tag4.15a}.

Moreover, the mapping properties extend to all the $\lambda $ for which the
resolvents and Poisson operators exist as solution operators to
\eqref{tag4.20}, \eqref{tag4.21}, in particular to $\lambda =0$.
For $A_\gamma ^{-1}$, this goes as follows: When $u\in H^1(\Omega )$ and $f\in
H^s(\Omega )$ with $s<1$, $f+\lambda u$ is likewise in $H^s(\Omega
)$. Then $A_\gamma u=f+\lambda u$ allows the conclusion $u\in
H^{s+2}(\Omega )$.  The argument works for all $|s|<\tau$. Moreover, since
$A_\gamma ^{-1}-(A_\gamma -\lambda)^{-1
}=-\lambda A_\gamma ^{-1}(A_\gamma -\lambda )^{-1}$ is of lower order
than $A_\gamma ^{-1}$, $A_\gamma ^{-1}$ coincides with $R^0(0)$ plus a
lower order remainder. 
The
Poisson operator solving \eqref{tag4.21} can be further described as
follows: There is a right
inverse $\mathcal K\colon { H}^{s+\frac32}(\Sigma )\to
H^{s+2}(\Omega )$ of $\gamma _0$ for $-\frac32<s\le
0$ (cf.\ Theorem \ref{thm:Traces}).
When we set $v=u-\mathcal K\varphi $,
we find that $v$ should solve
\[
(A-\lambda )v=-(A-\lambda )\mathcal K\varphi ,\quad \gamma _0v=0,
\]
to which we apply the preceding results; then when $\lambda \in \varrho (A_\gamma )$,
\begin{equation}
K^\lambda _\gamma =\mathcal K-(A_\gamma -\lambda )^{-1}(A-\lambda )\mathcal K;\label{tag4.21b}
\end{equation}
solves \eqref{tag4.21} uniquely. Thus $K^\lambda _\gamma $ exists for
all $\lambda \in \varrho (A_\gamma )$.

Since the formal adjoint $A'$ of $A$ is similar to $A$ (with regards
to strong ellipticity and smoothness
properties of the coefficients in its divergence form), the same construction works for the adjoint Dirichlet
problem, so also here we get the mapping properties
\begin{equation}
(A_\gamma ' -\bar\lambda )^{-1}\colon H^s(\Omega )\to H^{s+2}(\Omega ), \quad K^{\prime\bar\lambda }_\gamma
\colon { H}^{s+\frac32}(\Sigma )\to H^{s+2}(\Omega ),\label{tag4.21a}
\end{equation}
for $-\tau <s\le 0$.

The above analysis shows moreover that
\begin{eqnarray}\label{eq:Resolvs}
  R(\lambda)=R^0(\lambda) + S(\lambda), \qquad K(\lambda)=K^0(\lambda) + S'(\lambda),
\end{eqnarray}
where 
\begin{align}
  &\|R^0(\lambda)\|_{\mathcal{L}(H^{s,\mu}(\Omega ),H^{s+2,\mu
  }(\Omega ))},\; \|K^0(\lambda)\|_{\mathcal{L}(H^{s+\frac32,\mu}(\Sigma),H^{s+2,\mu}(\Omega
  ))}\text{ are }O(1),\; \nonumber
\\
&\|S(\lambda)\|_{\mathcal{L}(H^{s-\theta,\mu}(\Omega ),H^{s+2,\mu
  }(\Omega ))},\; \|S'(\lambda)\|_{\mathcal{L}(H^{s+\frac32-\theta,\mu}(\Sigma),H^{s+2,\mu}(\Omega
  ))}\text{ are }O(1),\label{eq:lambdaest1}
\\
&\|S(\lambda)\|_{\mathcal{L}(H^{s,\mu}(\Omega ),H^{s+2,\mu
  }(\Omega ))},\; \|S'(\lambda)\|_{\mathcal{L}(H^{s+\frac32,\mu}(\Sigma),H^{s+2,\mu}(\Omega
  ))}\text{ are }O(\ang\lambda^{-\theta/2}),\nonumber 
\end{align}
for $\lambda $ going to infinity on rays $\lambda =e^{i\eta }\mu ^2$,
$\eta\in (\pi/2-\delta ,3\pi/2+\delta )$, when $s,\theta$ are as in
\eqref{tag4.15a}.  Here $R^0(\lambda)$, $ K^0(\lambda)$ are explicit
parametrices as in \eqref{eq:R0Parametrix}--\eqref{eq:K0Parametrix}
(modified to depend on $\lambda$). 
For ``stationary'' norms, one has in particular 
\begin{align}
&\|R(\lambda )f\|_{s+2}+\ang\lambda ^{1+s/2}\|R(\lambda )f\|_{s}\le
C_s\min
\{\|f\|_{s}, \ang \lambda ^{s/2}\|f\|_0\},\label{eq:lambdaest3}\\
&\|K(\lambda )g\|_{s+2}+\ang\lambda ^{1+s/2}\|K(\lambda )g\|_{s}\le
C'_s\big(\|g\|_{s+\frac32}+ \ang \lambda ^{3/4+s/2}\|g\|_0\big).\label{eq:lambdaest4}
\end{align}

Note that $\|R(\lambda
)\|_{\mathcal{L}(L_2(\Omega ))}$ is $O(\ang\lambda ^{-1})$ on the ray.

Summing up, we have proved:

\begin{thm}
  Let $\Omega$, $\tau$, and $A$ be as in Assumption
{\rm \ref{tau-assumption2}} and let $-\tau <s\leq 0$.
Then for $\lambda \in \varrho (A_\gamma )$, the operator 
\begin{equation}
\begin{pmatrix} A-\lambda \\  \gamma_0 \end{pmatrix} \colon 
H^{s+2}(\Omega ) \to
\begin{matrix}
H^{s}(\Omega )\\ \times \\ { H}^{s+\frac32}(\Sigma )\end{matrix};\label{tag4.2a}
\end{equation}
has an inverse
\begin{equation}
\begin{pmatrix} R(\lambda )& K(\lambda )\end{pmatrix} =
\begin{pmatrix} (A_\gamma -\lambda )^{-1}& K^\lambda _\gamma \end{pmatrix} 
\colon \begin{matrix}
H^{s}(\Omega )\\ \times \\{ H}^{s+\frac32}(\Sigma )\end{matrix} \to
H^{s+2}(\Omega ).
\label{tag4.14a}\end{equation}

On the rays $\lambda =e^{i\eta }\mu ^2$ with $\eta \in (\pi /2-\delta
, 3\pi /2+\delta)$ (outside the range of the
principal symbol), the inverse exists for $|\lambda |$ sufficiently
large. $R(\lambda )$ and $K(\lambda )$ have the
structure in {\rm \eqref{eq:Resolvs}} and satisfy estimates {\rm
  \eqref{eq:lambdaest1}--\eqref{eq:lambdaest4}}.

Similar statements hold for $A'$.
\end{thm}

There is a class related condition $s>-\frac32$, cf.\ Theorem
\ref{Theorem3.1} and the beginning of  Section \ref{resolv}, that  prevents the
above construction (even if $\tau $ were $>2$) from defining the
Poisson operator to start in the space ${H}^{-\frac12}(\Sigma )$, 
but that will be needed for an
analysis as in Section \ref{extend}. Fortunately, it is possible to get
supplementing information in
other ways, as we shall see below.

\setcounter{section}{4}
\section{Dirichlet-to-Neumann operators}\label{dir2neu1}

\subsection{An extension of Green's formula  }

For a general treatment of realizations of $A$, we need to extend the
trace and Poisson operators to low-order Sobolev spaces. We begin by
establishing an extension of Green's formula.

For $\lambda \in \varrho (A_\gamma )$, $s\in [0,2]$, 
let 
\begin{equation}
Z^s_\lambda (A)=\{u\in H^s(\Omega )\mid (A-\lambda )u=0\};\label{tag2.3}
\end{equation}
it is a closed subspace of $H^s(\Omega )$. It follows from Theorem
\ref{thm:Traces} that the trace operator $\gamma _0$ is continuous:
\begin{equation}
\gamma _0\colon  Z^s_\lambda (A)\to H^{s-\frac12}(\Sigma ),\label{tag2.3a}
\end{equation}
for $s\in (\frac12,2]$. Moreover, in view of the solvability
properties shown in Section \ref{resolv}, it defines a homeomorphism   
\begin{equation}
\gamma _0\colon  Z^s_\lambda (A)\simto H^{s-\frac12}(\Sigma ),\label{tag2.3b}
\end{equation}
for $s\in (2-\tau ,2]$, with inverse $K^\lambda _\gamma =K(\lambda )$.

As shown in Section \ref{subsec:Green}, 
the trace operators
$\gamma _1 $, $\chi $ and $\chi '$ define continuous maps
\begin{equation}
\gamma _1 ,\chi ,\chi '\colon  Z^s_\lambda (A)\to  H^{s-\frac32}(\Sigma )
\label{tag2.4}
\end{equation}
for all $s\in (\frac32,2]$.

We need an extension of these mapping properties to all $s\in [0,2]$, 
along with an
extension of Green's formula to $u\in D(\Ama)$, $v\in H^{2}(\Omega
)$. 
 This is shown by the method of Lions and Magenes \cite{LM68}.  
Here we use the restriction operator $r_\Omega $
(restricting distributions from $\R^n$ to $\Omega $) and the
extension-by-zero operator $e_\Omega $ (extending functions on $\Omega
$ by zero on $\Rn\setminus\Omega $).

An important ingredient is the following
denseness result:

\begin{proposition}\label{Theorem9.8}
  The space $C_{(0)}^\infty (\comega)=r_\Omega C_0^\infty (\R^n)$ is dense in
  $D(A_{\operatorname{max}})$ (provided with the graph-norm).
\end{proposition}

\begin{proof} 
  This follows if we show that when $\ell $ is a continuous antilinear
  (conjugate linear) functional on $D(A_{\operatorname{max}})$ which
  vanishes on $C_{(0)}^\infty (\comega)$, then $\ell =0$. So let $\ell $
  be such a functional; it can be written as
\begin{equation}\label{tag9.23}
    \ell (u)=(f,u)_{L_2(\Omega )}+(g,Au)_{L_2(\Omega )}
\end{equation}  
for some $f,g\in L_2(\Omega)$. We know that $\ell (\varphi )=0$ for
$\varphi \in C_{(0)}^\infty (\comega)$. Any such $\varphi $ is the
restriction to $\Omega$ of a function $\Phi \in C_0^\infty (\mathbb
R^n)$, and in terms of such functions we have
\begin{equation}\label{tag9.24}
\ell (r_{\Omega }\Phi )=(e_{\Omega }f,\Phi)_{L_2(\mathbb R^n)}
+(e_{\Omega }g,A_e\Phi)_{L_2(\mathbb 
R^n)}=0,\text{ all } \Phi \in C_0^\infty (\mathbb R^n).
\end{equation}

The equations to the right in \eqref{tag9.24}
imply, in terms of the formal adjoint $A_e'$ on $\R^n$, 
\[ 
\ang{e_{\Omega }f+ A_e'e_{\Omega }g,\overline\Phi }=0,\text{ all }
\Phi \in C_0^\infty (\mathbb R^n),
\]
i.e., 
\begin{equation}\label{tag9.25} 
 e_{\Omega }f+ A_e'e_{\Omega }g=0, \text{ or }A_e'e_{\Omega }g=-e_{\Omega }f,
\end{equation}
as distributions on $\mathbb R^n$. Here we know that $e_{\Omega }g$ and $e_{\Omega }f$
are in $L_2(\mathbb R^n)$, and the solvability properties of $A_e'$ then imply  
that $e_{\Omega }g\in H^{2}(\mathbb R^n)$. Since it has support in $\comega$, it
identifies with a function in $H^{2}_0(\Omega)$, i.e., $g\in
H^{2}_0(\Omega)$. Then by \eqref{tag2.2a}, $g$ is in $ D(A'_{\min})$.  And \eqref{tag9.25}
implies that $A'g=-f$. But then, for any $u\in D(A_{\max})$,
\[
\ell (u)=(f,u)_{L_2(\Omega)} + (g,Au)_{L_2(\Omega)} = -(A'g,u)_{L_2(\Omega)}
+ (g,Au)_{L_2(\Omega)} = 0,
\]
since $A_{\max}$ and $A'_{\min}$ are adjoints. 
\end{proof}

We shall show:

\begin{theorem}\label{Theorem9.10} 
The collected trace operator $\{\gamma _0,\chi\}$, defined on $C_{(0)}^\infty (\comega)$, extends by
  continuity to a continuous mapping from $D(A_{\max})$ to
  $H^{-\frac12}(\Sigma )\times H^{-\frac32}(\Sigma )$.
  Here Green's formula {\rm \eqref{Green}} extends to the
  formula
\begin{equation}\label{tag9.29}
    (Au,v)_{L_2(\Omega)}-(u,A'v)_{L_2(\Omega)}=(\chi u,{\gamma _0
        v})_{-\frac32, \frac32}-(\gamma _0
      u,{\chi 'v})_{-\frac12,\frac12},
\end{equation}
for $u\in D(A_{\max})$, $v\in H^{2}(\Omega)$.

\end{theorem}

\begin{proof}
  Let $u\in D(A_{\max})$. We want to define $\{\gamma _0u,\chi
  u\}$ as a continuous antilinear functional on $H^{\frac12}(\Sigma
  )\times H^{\frac32}(\Sigma )$, depending continuously (and of course linearly) on $u\in
  D(A_{\max})$. For this we use that 
\[ \begin{pmatrix}\gamma _0\\ \quad \\ \chi '
  \end{pmatrix}\colon H^{2}(\Omega )\to \begin{matrix}H^{\frac32}(\Sigma )\\
    \times \\
H^{\frac12}(\Sigma )\end{matrix}\text{ has a
  continuous right inverse }{\cal K}'=\begin{pmatrix}{\cal K}'_0&{\cal
    K}'_1\end{pmatrix},\] 
a lifting operator, cf.\ Corollary \ref{cor:TraceChi}. For a given $\varphi =\{\varphi _0,\varphi _1\}\in
  H^{\frac12}(\Sigma )\times
  H^{\frac32}(\Sigma )$, we set 
\[ 
w_\varphi ={\cal K}'_0 \varphi _1-{\cal K}'_1\varphi
_0;
\text{ then
}\gamma _0w_\varphi =\varphi _1,\; \chi 'w_\varphi =-\varphi _0.
\]
Now we define
\begin{equation}\label{tag9.30}
\begin{split}
  \ell _u(\varphi )&=(Au,w_\varphi )-(u,A'w_\varphi ),\text{ noting that}\\
  |\ell _u(\varphi )|&\le C\|u\|_{D(A_{\max})}\|w_\varphi
  \|_{H^{2}(\Omega)} \le C'\|u\|_{D(A_{\max})}\|\varphi
  \|_{H^{\frac12}(\Sigma )\times
  H^{\frac32}(\Sigma )}.
\end{split}
\end{equation} 
So, $\ell_u$ is a continuous antilinear functional on $\varphi \in
H^{\frac12}(\Sigma )\times
  H^{\frac32}(\Sigma )$, hence defines an element $\psi =\{\psi
_0,\psi _1\}\in  H^{-\frac12}(\Sigma )\times
  H^{-\frac32}(\Sigma )$ such that
\begin{equation}
\ell_u(\varphi )=(\psi _0,\varphi _0)_{-\frac12,
\frac12}+(\psi _1,\varphi _1)_{-\frac32,
\frac32} .\label{tag9.30a}
\end{equation}
Moreover, it depends continuously on $u\in D(A_{\max})$, in view of
the estimates in \eqref{tag9.30}. If $u$ is in $C_{(0)}^\infty
(\comega)$, the defining formula in \eqref{tag9.30} can be rewritten
using Green's formula \eqref{Green}, which leads to 
\[ 
\ell_u(\varphi )=(Au,w_\varphi )-(u,A'w_\varphi )= (\chi u,\gamma _0
w_\varphi )-(\gamma _0u,\chi 'w_\varphi ) =(\chi u,\varphi _1)+(\gamma _0u,\varphi
_0)
\]
for such $u$. Since $\varphi _0$ and $\varphi _1$ run through full
Sobolev spaces, it follows by comparison with \eqref{tag9.30a} 
that $\psi _0=\gamma _0u$, $\psi _1=\chi u$,
when $u\in C_{(0)}^\infty (\comega)$, so the functional $\ell_u$ is
consistent with $\{\gamma _0u,\chi u\}$ then.  Since $C_{(0)}^\infty
(\comega)$ is dense in $D(A_{\max})$, we have found the unique
continuous extension.

Identity \eqref{tag9.29} is now obtained in general by extending
\eqref{Green} by continuity from $u\in C_{(0)}^\infty (\comega)$,
$v\in H^{2}(\Omega)$. \end{proof}

In particular, the validity of the mapping properties of $\gamma _0$ and
$\chi $ in \eqref{tag2.3a} and \eqref{tag2.4} extend to $s=0$.

\subsection{Poisson operators}

The next step is to
 extend the action of the Poisson operators to
low-order spaces.

\begin{lemma}\label{LemmaPoisson} 
The composed operator $\chi '(A'_\gamma
-\bar\lambda   )^{-1}\colon L_2(\Omega )\to H^{\frac12}(\Sigma )$ has as adjoint an operator  $(\chi '(A'_\gamma
  -\bar\lambda )^{-1})^*\colon  H^{-\frac12}(\Sigma )\to L_2(\Omega )$ extending $K^\lambda _\gamma $
(originally known to map $H^{s-\frac12}(\Sigma )$ to $Z^s_\lambda (A) $ for
$s\in (2-\tau ,2]$). Moreover, $(\chi '(A'_\gamma
  -\bar\lambda )^{-1})^*$ ranges in $Z^0_\lambda (A)$.
  \end{lemma}  

\begin{proof}
Let $\varphi \in H^{s-\frac12}(\Sigma )$ for some $s\in (2-\tau ,2]$,
let $u=K^\lambda _\gamma \varphi $. For any
$f\in L_2(\Omega )$, let $v=(A'_\gamma -\bar\lambda )^{-1}f$. Note that $(A-\lambda )u=0$ and $\gamma _0v=0$. Then by \eqref{tag9.29},
\begin{align*}
-(K^\lambda _\gamma \varphi ,f)&=    ((A-\lambda
)u,v)-(u,(A'-\bar\lambda) v)\\
&=-(\gamma _0
      u,{\chi 'v})_{-\frac12,\frac12}=-(\varphi ,{\chi '(A'_\gamma -\bar\lambda )^{-1}f})_{-\frac12,\frac12}.\nonumber
\end{align*}
This shows that the adjoint of $\chi '(A'_\gamma
-\bar\lambda   )^{-1}$ acts like
$K^\lambda _\gamma $ on functions $\varphi \in H^{s-\frac12}(\Sigma
)$, $s\in (2-\tau ,2]$.

To see that 
$(\chi '(A'_\gamma -\bar\lambda )^{-1})^*$ maps into the nullspace of 
$A-\lambda $,
  let $\varphi \in H^{-\frac12}(\Sigma )$ and let $v\in C_0^\infty
  (\Omega )$. Then, using the definition of $A$ in the weak sense,
\begin{align*}
\langle (A-\lambda )(\chi '(A'_\gamma -\bar\lambda )^{-1})^*\varphi ,
\bar v\rangle _\Omega &= \langle (\chi '(A'_\gamma -\bar\lambda
)^{-1})^*\varphi , \overline { (A'-\bar\lambda )v}\rangle _\Omega \\
&=
( (\chi '(A'_\gamma -\bar\lambda
)^{-1})^*\varphi ,  { (A'-\bar\lambda )v}) _{L_2(\Omega )}\\
&= ( \varphi , \chi '(A'_\gamma -\bar\lambda
)^{-1} { (A'-\bar\lambda )v}) _{-\frac12, \frac12}=0,
\end{align*}
since $v\in C_0^\infty  (\Omega )$ implies $(A'_\gamma -\bar\lambda
)^{-1} (A'-\bar\lambda )v=v$ (since $\gamma _0v=0$), and $\chi 'v=0$.
Thus $(A-\lambda )(\chi '(A'_\gamma -\bar\lambda )^{-1})^*\varphi=0$
in the weak sense, so since $(\chi '(A'_\gamma -\bar\lambda
)^{-1})^*\varphi  \in L_2(\Omega )$, it lies in $Z^0_\lambda (A)$.
\end{proof} 

Since $(\chi '(A'_\gamma -\bar\lambda )^{-1})^*$ extends $K^\lambda
_\gamma $ and maps into $Z^0_\lambda (A)$, we {\it define} this to be
the operator $K^\lambda _\gamma $ for $s=0$:
\begin{equation}\label{tag4.29a}
K^\lambda _\gamma =(\chi '(A'_\gamma -\bar\lambda )^{-1})^*\colon H^{-\frac12}(\Sigma )\to L_2(\Omega ) .
\end{equation}

\begin{theorem}\label{Theorem4.2a} Let $\Omega $
  and $A$ satisfy Assumption {\rm \ref{tau-assumption2}}. 
The operator $K^\lambda _\gamma $ defined in {\rm \eqref{tag4.29a}} 
maps $H^{s-\frac12}(\Sigma )$ to $H^s(\Omega
)$ continuously for $0\le s\le 2$. Moreover, $\gamma _0$ defined in
Theorem {\rm \ref{Theorem9.10}} is a homeomorphism {\rm
  \eqref{tag2.3b}} for all $s\in [0,2]$, with $K^\lambda _\gamma $
acting as its inverse.

There is a similar result for $K^{\prime\bar\lambda }_\gamma $. 
\end{theorem}

\begin{proof}
Since the composed operator is continuous:
\[
\chi '(A'_\gamma -\lambda )^{-1}\colon H^s(\Omega )\to
 H^{s+\frac12}(\Sigma )
\]
for $-\tau <s\le 0$,
it follows by duality that
\begin{equation}
K^\lambda _\gamma \colon   H^{s'-\frac12}(\Sigma )\to H^{s'}(\Omega
),
\label{tag4.29}
\end{equation}
when $0\le s'<\tau $ (recall that $\tau <\frac12$, cf.\ \eqref{asspt2}).
Taking this together with the larger values that were covered by \eqref{tag4.26},
we find that \eqref{tag4.29} holds for
\begin{equation}
0\le s'\le 2;\label{tag4.31}
\end{equation}
the intermediate values are included by interpolation. We can replace this $s'$
by $s$.

The identities 
\[ \gamma _0 K^\lambda _\gamma \varphi =\varphi \text{ for }\varphi
\in H^{s-\frac12}(\Sigma ), \quad  
  K^\lambda _\gamma \gamma_0 z =z \text{ for }z
\in Z^{s}_\lambda (A), 
\]  
were shown in Section \ref{resolv} to hold for $s\in (2-\tau , 2]$. The first
identity now extends by continuity to $H^{-\frac12}(\Sigma)$, since
$H^\frac32(\Sigma)$ is dense in this space. The second identity will
extend by continuity to $Z^{0}_\lambda (A)$, if we can prove that
$Z^{2}_\lambda (A)$ is dense in $Z^{0}_\lambda (A)$. Indeed, this
follows from Proposition  \ref{Theorem9.8}: 

Let $z\in Z^{0}_\lambda (A)$. By Proposition  \ref{Theorem9.8} applied
to $A-\lambda $, there
is a sequence $u_k\in C_{(0)}^\infty (\comega)$ such that $u_k\to z$
and $(A-\lambda )u_k\to 0$ in $L_2(\Omega )$. Then
$v_k=$ \linebreak $(A_\gamma-\lambda )^{-1}(A-\lambda )u_k\to 0$ in
$H^2(\Omega)$. Let $z_k=u_k-v_k$; then $z_k\in H^2(\Omega)$,
$(A-\lambda )z_k=0$, and $z_k\to z$ in $L_2(\Omega)$. Hence, $z_k$ is a
sequence of elements of $Z^{2}_\lambda (A)$ that converges to $z$ in
$Z^{0}_\lambda (A)$, showing the desired denseness.

Thus the identities are valid for $s=0$, and hence for all $s\in
[0,2]$. In particular, $K^\lambda _\gamma $ maps
$H^{s-\frac12}(\Sigma)$ bijectively onto $ Z^{s}_\lambda (A)$,
and $\gamma _0$ maps $ Z^{s}_\lambda (A)$ bijectively onto
$H^{s-\frac12}(\Sigma)$, for
$s\in [0,2]$, as inverses of one another.

The proof in the primed situation is analogous.
\end{proof}

The
adjoints also extend, e.g.
\begin{equation}
(K^{\prime\bar\lambda }_\gamma)^*\colon  H^s_0(\comega )\to   H^{s+\frac12}(\Sigma
),\text{ for }-2\le s\le 0;
\end{equation}
recall that $H^s_0(\comega)=H^s(\Omega )$ when $|s|< \frac12$. 

  From \eqref{eq:lambdaest1} we conclude moreover that when $0\le s<\tau $,
\begin{equation}
\|\chi '(A'_\gamma -\bar\lambda )^{-1}\|_{\mathcal{L}(H^{-s,\mu}(\Omega ),H^{-s+\frac12,\mu
  }(\Sigma  ))}\text{ and }\|K^\lambda _\gamma \|_{\mathcal{L}(H^{s-\frac12,\mu}(\Sigma),H^{s,\mu
  }(\Omega ))}
\text{ are }O(1),\label{eq:lambdaest5}
\end{equation}
for $\lambda $ going to infinity on rays $\lambda =e^{i\eta }\mu ^2$,
$\eta\in (\pi/2-\delta ,3\pi/2+\delta )$. In particular,
\begin{equation}
\|K^\lambda _\gamma \varphi \|_{0}\le C\min\{
\|\varphi \|_{-\frac12},  \ang \lambda ^{-1/4}\|\varphi \|_0\}.\label{eq:lambdaest6}
\end{equation}

We shall now analyze the
structure somewhat further. It should be noted that a Poisson
operator maps a Sobolev space over $\Sigma $ to a Sobolev space over
$\Omega $; the co-restriction of  $K^\lambda _\gamma $ mapping into
$Z^0_\lambda (A)$ will be denoted $\gamma _{Z_\lambda }^{-1}$ further
below, cf.\ \eqref{tag2.6c}.

\begin{theorem}\label{Theorem4.2} Let $0<\delta<1$, and let $\Omega $
  and $A$ satisfy Assumption {\rm \ref{tau-assumption2}}. 

$1^\circ$ 
$K^\lambda _\gamma $ is the sum of a Poisson operator of the form
\begin{equation}\label{eq:principalPoisson}
 {K}^{0,\lambda}_\gamma v=\sum_{j=1}^J \psi_jF^{-1,\ast}_j \Lambda_{-,+}^{-2}(\lambda )k_{j,\lambda} (x',D_x)F^{\ast}_{j,0} \varphi_j v,
\end{equation}
where $k_{j,\lambda }$ has symbol-kernel $\tilde k_{j,\lambda }\in C^\tau S^{1}_{1,0}(\R^{n-1}\times
\R^{n-1},\srplus)$, and a remainder $\mathcal {S}(\lambda )$ that for
$s\in (2-\tau,2]$ maps $H^{s-\frac12-\theta }(\R^{n-1})\to H^s(\Rn_+)$,
when  $0<\theta < s-2+\tau $.

$2^\circ$ 
$K^\lambda _\gamma $ is a 
generalized Poisson
operator in the sense that it
is the sum of a Poisson operator of the form
\begin{equation}\label{eq:PoissonParametrix}
  {K}^{\sharp\lambda}_\gamma v=\sum_{j=1}^J \psi_jF^{-1,\ast}_j \Lambda_{-,+}^{-2}(\lambda )k_{j,\lambda}^\sharp (x',D_x)F^{\ast}_{j,0} \varphi_j v,
\end{equation}
with $\tilde{k}_{j,\lambda}^\sharp\in S^{1}_{1,\delta}(\R^{n-1}\times
\R^{n-1},\SD(\overline{\R}_+))$, and a remainder ${\mathcal R}(\lambda )$ that
for $s\in (0,2]$  maps $
H^{s-\frac12-\eps}(\R^{n-1})\to H^s(\Rn_+)$, for some $\eps=\eps(s)>0$.

Here $\psi_j,\varphi_j, F_j$, and $F_{j,0}$ are as in
Remark {\rm \ref{rem:Charts}}. 

There are similar statements for the primed version $K^{\prime\bar\lambda }_\gamma $. 
\end{theorem}

\begin{proof}
We give the proof of the statements for $K^\lambda_{\gamma}$; the
proofs for ${K}^{\prime\bar\lambda}_{\gamma}$ are analogous.

The first statement follows from the construction in Section
\ref{subsec:gendom}, applied to $A-\lambda $ and with $\lambda
$-dependent order-reducing operators (where $\xi $ is replaced by
$(\xi ,\mu )$, $\mu =|\lambda |^{\frac12}\in \crp$). The composition
$k_{j,\lambda }$ of the $\lambda $-dependent variant of $k^0_j$ and $\Lambda
_0^\frac32(\lambda )$ is of order 2 and has symbol-kernel $\tilde
k_{j,\lambda }$ in  $ C^\tau
S^{1}_{1,0}(\R^{n-1}\times 
\R^{n-1},\srplus)$ for each $\lambda $. The mapping properties of $\mathcal{S}(\lambda)$ follow from \eqref{eq:lambdaest1}.

For the second statement, observe that we have from 1$^\circ$
that 
\begin{equation}\label{eq:FirstApproxPoisson}
K^\lambda_\gamma=K^{0,\lambda}_\gamma+ S(\lambda),\quad \text{where}\ S(\lambda)\in \mathcal{L}(H^{\frac32-\theta}(\Sigma),H^2(\Omega)), 
\end{equation}
for every $0<\theta<\tau$.
Now, applying
Lemma~\ref{lem:PoissonSmoothing} for $\delta\in (0,1)$, 
we obtain that $K^{0,\lambda}_\gamma= {K}^{\sharp\lambda}_\gamma+
S'(\lambda)$, where ${K}^{\sharp\lambda}_\gamma$ is as described in
(\ref{eq:PoissonParametrix}) and 
\begin{equation*}
  S'(\lambda)\colon H^{\frac32-\tau\delta}(\Sigma)\to H^2(\Omega).
\end{equation*}
Since also $K^\lambda_\gamma$, ${K}^{\sharp\lambda}_\gamma\in  \mathcal{L}(H^{s-\frac12}(\Sigma),H^s(\Omega))$ for all $s\in [0,2]$, interpolation yields that for every $s\in (0,2]$ there is some $\eps=\eps(s)>0$ such that
\begin{equation*}
  K^\lambda_\gamma-{K}^{\sharp\lambda}_\gamma\in  \mathcal{L}(H^{s-\frac12-\eps}(\Sigma),H^s(\Omega)).
\end{equation*}
This proves the theorem.
\end{proof}

\begin{remark} \label{rem:reginfty}As a technical observation we note
  that the above  
approximate Poisson solution operators
  ${K}^{0,\lambda}_\gamma$ and $K^{\sharp\lambda}_\gamma$ are
  constructed in such a way that their symbol-kernels are smooth 
in $(\xi ',\mu )\in \crnp$ (outside a neighborhood of zero); this
  is the case of symbols ``of regularity $+\infty $'' in the sense of
  \cite{FunctionalCalculus}.
\end{remark}

\subsection{Dirichlet-to-Neumann operators}

Finally, we shall study the composed operators $P^\lambda _{\gamma  ,\chi }
=\chi K^\lambda _\gamma $ and $
P^{\prime\bar\lambda }_{\gamma  ,\chi '}=\chi '
K^{\prime\bar\lambda }_\gamma $; often called Dirichlet-to-Neumann operators.

It follows immediately from Theorem \ref{Theorem4.2a}
that they are continuous for  $s\in [0,2]$,
\begin{equation}
P^\lambda _{\gamma  ,\chi },P^{\prime\bar\lambda }_{\gamma  ,\chi '}\colon 
 H^{s-\frac12}(\Sigma )\to  H^{s-\frac32}(\Sigma ).\label{tag4.33}
\end{equation}

Applying Green's formula \eqref{tag9.29} to functions $u,v$ with
$Au=0$, $A'v=0$, we see that
\begin{equation*}
(P^\lambda _{\gamma  ,\chi }\varphi, \psi)_{-\frac32,\frac32}=(\varphi,P^{\prime\bar\lambda }_{\gamma  ,\chi '}\psi)_{-\frac12,\frac12}
\end{equation*}
for all $\varphi \in H^{-\frac12}(\Sigma), \psi\in H^{\frac32}(\Sigma)$,
so $P^\lambda _{\gamma  ,\chi  }$ and $P^{\prime\bar\lambda }_{\gamma  ,\chi '
}$ are consistent with each other's adjoints.

\begin{theorem}\label{Theorem4.3} Assumptions as in
  Theorem {\rm \ref{Theorem4.2}}. 
$P^\lambda _{\gamma ,\chi  }$
maps
$H^{s-\frac12}(\Sigma )$ continuously to $H^{s-\frac32}(\Sigma)$
for $s\in [0,2]$, and satisfies:

$1^\circ$
$P^\lambda _{\gamma ,\chi }$  is 
the sum of a first-order $\psi $do of the form 
\begin{equation*}
    S^\lambda v= \sum_{j=1}^J \eta_j \widetilde{F}^{-1,\ast}_{j,0}\Lambda _0^{-\frac12} s_{j,\lambda }(x',D_{x'}) F^{\ast}_{j,0}\varphi_j v,
\end{equation*}
where $s_{j,\lambda }\in C^\tau S^{\frac32}_{1,0}(\R^{n-1}\times \R^{n-1})$,
and a remainder, such that for every $s\in (2-\tau,2]$, the remainder  
 maps $H^{s-\frac12-\eps}(\Sigma )$ continuously to $H^{s-\frac32}(\Sigma)$ for some $\eps=\eps(s)>0$.

$2^\circ$ There is a pseudodifferential operator
$P_{\gamma,\chi}^{\sharp\lambda}$ on $\Sigma$ with symbol in \linebreak 
$S^1_{1,\delta}(\R^{n-1}\times \R^{n-1})$ (cf.\ {\rm
  \eqref{eq:DirNeumannSharp}}), such that for every $s\in (0,2]$, $P^\lambda _{\gamma ,\chi }$ is a generalized  $\psi $do of
order $1$ in the sense that it is the sum of  $P_{\gamma,\chi}^{\sharp\lambda}$
and a remainder mapping
$H^{s-\frac12-\eps}(\Sigma )$ continuously to $H^{s-\frac32}(\Sigma)$
for some $\eps=\eps(s)>0$. Here $\eps>0$ can be chosen uniformly with respect to $s\in [s',2]$ for every $s'\in (0,2]$.
 
$3^\circ$ There is a pseudodifferential operator
${P}^{\sharp{\lambda}1}_{\gamma,\chi}$ on $\Sigma$ with 
symbol in \linebreak
$S^1_{1,\delta}(\R^{n-1}\times \R^{n-1})$ (cf.\ {\rm
  \eqref{eq:DirNeumannSharp1}}), such that  for $s=0$, 
$P^\lambda _{\gamma ,\chi }$ is the sum of  ${P}^{\sharp{\lambda}1}_{\gamma,\chi}$
and a remainder mapping $H^{-\frac12}(\Sigma )$ continuously to $H^{-\frac32+\eps}(\Sigma)$ for some $\eps>0$.

There are similar statements for $P^{\prime\bar\lambda }_{\gamma ,\chi '
}$; it acts like an adjoint of $P^\lambda _{\gamma ,\chi }$. 
\end{theorem}

\begin{proof} 
The first and last statements were shown above. 
We shall prove statements $1^\circ$--$3^\circ$ for
$P^\lambda_{\gamma,\chi}$; the 
treatment of ${P}^{\prime\bar\lambda}_{\gamma,\chi'}$ is analogous.

For $1^\circ$, we note that  since 
$K_\gamma^\lambda$ coincides in highest order with the approximation
\eqref{eq:principalPoisson}, $P^\lambda _{\gamma ,\chi }$ coincides in
the highest order with
  \begin{equation*}
   \chi  K_\gamma^{0,\lambda }v= \sum_{j=1}^J \chi \psi_j(x) F^{-1,\ast}_j\Lambda_{-,+}^{-2}k_{j,\lambda }(x',D_{x}) F^\ast_{j,0} \varphi_jv; 
  \end{equation*} 
here $v\in H^{s-\frac12}(\Sigma )$, $s\in (2-\tau,2]$, and 
$\tilde k_{j,\lambda }\in C^\tau S^{1}_{1,0}(\R^{n-1}\times
\R^{n-1},\mathcal{S}(\ol{\R}_+))$.  

Moreover, 
because of Corollary~\ref{cor:TraceChi}, for every $\eta_j\in C^\infty_{(0)}(\ol{\Omega})$ with $\eta_j\equiv 1$ on $\supp \psi_j$ and $\supp \eta_j\subset U_j$ 
we have the representation
 \begin{eqnarray*}
   \chi (\psi_jF^{-1,\ast}_j v_j) &=& \eta_j s_0 \gamma_1 (\psi_j
   F^{-1,\ast}_j v_j)+\eta_j \sum_{|\alpha|\leq 1}c_{\alpha,j}D_{x'}^\alpha\psi_j\gamma_0  F^{-1,\ast}_j v_j\\ 
&=&  \eta_j \widetilde{F}^{-1,\ast}_{j,0} t_j(x',D_x)v_j,
 \end{eqnarray*}
where the operators $t_j(x',D_x)$ are  differential trace operators of
order $1$  and class $2$ on $\Rn_+$ (in $x$-form) with coefficients in
$H^{\frac12}_p(\R^{n-1})$. --- Note that the factor $\kappa $ in the
definition of $\widetilde{F}^{-1,\ast}_{j,0}$, cf.\
\eqref{eq:TildeF2}, can be absorbed into the coefficients of $t_j(x',D_x)$. ---
Now if we set
${t'_{j,\lambda }}(x',\xi',D_n)=\weight{(\xi',\mu)}^{\frac12}t_j(x',\xi',D_n)$
(where as usual $\mu =|\lambda|^{\frac12}$ and $\weight{(\xi',\mu)}^r$ is the
symbol of $\Lambda _0^r(\lambda )$),
 we
have from Theorem~\ref{thm:PsDOCompo} that 
  \begin{equation*}
    \Lambda_0^{\frac12}t_j(x',D_x) - t'_{j,\lambda }(x',D_x)\colon H^{s-\eps}(\Rn_+)\to H^{s-2}(\R^{n-1}), 
  \end{equation*}
for all $s,\eps$ such that $\frac32+\eps<s\leq 2 $ and $0<\eps<\tau=  \frac12-\frac{n-1}p$.
In particular, we can choose $s=2$. Moreover, since $t_j(x',D_x)$ is a differential trace operator, we have that
\begin{eqnarray*}
  t'_{j,\lambda }(x',D_x) \Lambda_{-,+}^{-2} &=& b_{j,1,\lambda }(x',D_{x'})\gamma_1\Lambda_{-,+}^{-2} +b_{j,0,\lambda }(x',D_{x'})\gamma_0\Lambda_{-,+}^{-2} 
\end{eqnarray*}
for some  $b_{j,k,\lambda } \in H^{\frac12}_pS^{\frac32-k}_{1,0}(\R^{n-1}\times \R^{n-1})$, $k=0,1$,
which implies that $t'_{j,\lambda }(x',D_x) \Lambda_{-,+}^{-2}= t_{j,\lambda }''(x',D_x)$ with 
$\tilde{t}_{j,\lambda }''\in C^\tau
S^{-\frac12}_{1,0}(\R^{n-1}\times\R^{n-1},\SD(\R^{n-1}))$, since
$H^{\frac12}_{p}(\R^{n-1})\hookrightarrow C^\tau(\R^{n-1})$. (More
precisely, $b_{j,1,\lambda }(x',\xi')= s_0(x')\weight{(\xi',\mu)}^{\frac12}$ and
$b_{j,0,\lambda }(x',\xi')=$ \linebreak $\sum_{|\alpha|\leq 1} c_{\alpha,j}(x')(\xi')^\alpha
\weight{(\xi',\mu)}^{\frac12}$.) 
Set
\begin{equation*}
s_{j,\lambda }(x',\xi ')=t''_{j,\lambda }(x',\xi ',D_n)k_{j,\lambda }(x',\xi ',D_n);
\end{equation*}
it is in $ C^\tau S^{\frac32}_{1,0}(\R^{n-1}\times \R^{n-1})$. Then 
we can apply the composition rules for Green operators with
$C^\tau$-coefficients, cf.\ \cite[Theorem~4.13.3]{NonsmoothGreen}, to
conclude that  
\begin{equation*}
    t'_{j,\lambda }(x',D_x) \Lambda_{-,+}^{-2}k_{j,\lambda } (x',D_x) - s_{j,\lambda}(x',D_{x'})\colon H^{\frac32-\eps}(\R^{n-1})\to L^2(\R^{n-1}),
  \end{equation*}
  for some $\eps>0$.

Summing up,
we have that
\begin{eqnarray*}
  P_{\gamma,\chi}^\lambda v=\chi {K}_\gamma^\lambda v &=&
  \sum_{j=1}^J \eta_j \widetilde{F}^{-1,\ast}_{j,0}
  t_j(x',D_x)\Lambda_{-,+}^{-2}k_j^\lambda(x',D_x)F^\ast_{j,0}\varphi_j v+{\mathcal R}(\lambda ) v\\ 
&=& \sum_{j=1}^J \eta_j \widetilde{F}^{-1,\ast}_{j,0}
\Lambda_0^{-\frac12}s_{j,\lambda}(x',D_{x'})F^\ast_{j,0}\varphi_j v +
\mathcal{R}'(\lambda ) v, 
\end{eqnarray*}
where $\mathcal{R}(\lambda ) $ and $\mathcal{R}'(\lambda ) $ are
continuous from $ H^{\frac32-\eps}(\Sigma)$ to $ H^{\frac12}(\Sigma)$ for some $\eps>0$.

Hence, if we define
\begin{equation*}
  S^\lambda v= \sum_{j=1}^J
  \eta_j \widetilde{F}^{-1,\ast}_{j,0}\Lambda_0^{-\frac12} s_{j,\lambda}(x',D_{x'}) F^{\ast}_{j,0}\varphi_j v,
\end{equation*}
then
$
  P_{\gamma,\chi}^\lambda - S^\lambda\colon H^{\frac32-\eps}(\Sigma) \to H^{\frac12}(\Sigma)
$
is a bounded operator for some $\eps>0$. Moreover, because of
Proposition~\ref{prop:HsBddness}, $S^\lambda\in
\mathcal{L}(H^{s-\frac12}(\Sigma),H^{s-\frac32}(\Sigma))$ for every
$s\in (2-\tau,2]$. Finally, interpolation yields that for every $s\in
(2-\tau,2]$ there is some $\eps>0$ such that 
\begin{equation*}
    P_{\gamma,\chi}^\lambda - S^\lambda\colon
    H^{s-\frac12-\eps}(\Sigma) \to H^{s-\frac32}(\Sigma) 
\end{equation*}
is bounded. This proves the first statement.

To prove the second statement we apply symbol smoothing,
cf.\ (\ref{eq:SymbolSmoothing}), to $s_{j,\lambda}$. Then we obtain for
any $0<\delta <1$ that 
\begin{equation*} 
  s_{j,\lambda}(x',D_{x'})= s_{j,\lambda}^{\sharp}(x',D_{x'})+s_{j,\lambda}^{b}(x',D_{x'})
\end{equation*}
where $s_{j,\lambda}^{\sharp}\in S^{\frac32}_{1,\delta}(\R^{n-1}\times
\R^{n-1})$ and $s_{j,\lambda}^b \in C^\tau
S^{\frac32-\tau\delta}_{1,\delta}(\R^{n-1}\times \R^{n-1})$. 
Hence, 
\begin{equation*}
\Lambda_0^{-\frac12}
s_{j,\lambda}(x',D_{x'})-\Lambda_0^{-\frac12}
s_{j,\lambda}^\sharp(x',D_{x'})\colon H^{\frac32-\tau\delta}(\R^{n-1})\to
H^{\frac12}(\R^{n-1}), 
\end{equation*}
 since
\begin{equation*}
  s_{j,\lambda}^b(x',D_{x'})\colon H^{\frac32-\tau\delta}(\R^{n-1})\to
  L^2(\R^{n-1}) ,
\end{equation*}
by Proposition~\ref{prop:HsBddness}. 
Moreover, by the composition rules of the (smooth)
$S^m_{1,\delta}$-calculus, cf. e.g. \cite[Theorem~1.7,
Chapter~II]{KumanoGo}, 
\begin{equation*}
  \Lambda_0^{-\frac12} s_{j,\lambda}^\sharp(x',D_{x'}) = p_{j,\lambda}^{\sharp}(x',D_{x'})
\end{equation*}
for some $p_{j,\lambda}^\sharp\in S^{1}_{1,\delta}(\R^{n-1}\times
\R^{n-1})$. Altogether, if we define 
\begin{equation}\label{eq:DirNeumannSharp}
  P_{\gamma,\chi}^{\sharp\lambda} v= \sum_{j=1}^J \eta_j
  \widetilde{F}^{-1,\ast}_{j,0}p_{j,\lambda}^\sharp(x',D_{x'}) F^{\ast}_{j,0}\varphi_j v, 
\end{equation}
then
\begin{equation}\label{eq:ParametrixDiffMapping}
 P_{\gamma,\chi}^\lambda - {P}_{\gamma,\chi}^{\sharp\lambda}\colon
 H^{\frac32-\eps_0}(\Sigma) \to H^{\frac12}(\Sigma) 
\end{equation}
is a bounded operator for some $\eps_0>0$. Since the $p^\sharp_{j,\lambda }$ are smooth
symbols, we have that 
\begin{equation*} 
  P_{\gamma,\chi}^\lambda,
  P_{\gamma,\chi}^{\sharp\lambda}\in\mathcal{L}
  (H^{s-\frac12}(\Sigma), H^{s-\frac32}(\Sigma))   
\end{equation*}
for every $s\in [0,2]$, by Proposition~\ref{prop:HsBddness}. Hence,
interpolation implies that for every $s\in (0,2]$ there is some
$\eps=\eps(s)>0$ such that 
\begin{equation*}
  P_{\gamma,\chi}^\lambda - P_{\gamma,\chi}^{\sharp\lambda}\colon
  H^{s-\frac12-\eps}(\Sigma) \to H^{s-\frac32}(\Sigma) 
\end{equation*}
is a bounded operator. Here $\eps=\eps_0\frac{s}2$, which shows that $\eps$ can be chosen uniformly with respect to $s\in [s',2]$ for every $s'\in (0,2]$.

Finally, the case $s=0$ will be treated differently. To this end let
${P}^{\prime\sharp\bar{\lambda}}_{\gamma,\chi'}$ denote the approximation
to ${P}^{\prime\bar{\lambda}}_{\gamma,\chi'}$ defined analogously to
\eqref{eq:DirNeumannSharp}. Then 
\begin{equation*}
  P^\lambda_{\gamma,\chi}-
  ({P}^{\prime\sharp\bar{\lambda}}_{\gamma,\chi'})^*
  =({P}^{\prime\bar{\lambda}}_{\gamma,\chi'})^*
  -({P}^{\prime\sharp\bar{\lambda}}_{\gamma,\chi'})^*\colon
  H^{-\frac12}(\Sigma)\to H^{-\frac32+\eps}(\Sigma),    
\end{equation*}
by \eqref{eq:ParametrixDiffMapping} for the primed
operators. Moreover, by the standard calculus for pseudo\-differential
operators with $S^m_{1,\delta}$-symbols, 
\begin{align}
  ({P}^{\prime\sharp\bar{\lambda}}_{\gamma,\chi'})^* v &=
  \sum_{j=1}^J  \varphi_j\widetilde{F}^{-1,\ast}_{j,0}{p}_{j,\lambda}^{\prime\sharp\ast}(x',D_{x'})
  F^{\ast}_{j,0}\eta_j v +{\mathcal R}(\lambda ) v \nonumber\\
&= \sum_{j=1}^J  \eta_j\widetilde{F}^{-1,\ast}_{j,0}{p}_{j,\lambda}^{\prime\sharp\ast}(x',D_{x'})
F^{\ast}_{j,0}\varphi_j v +\mathcal{R}'(\lambda ) v =
{P}^{\sharp{\lambda}1}_{\gamma,\chi}v+\mathcal{R}'(\lambda ) v\label{eq:DirNeumannSharp1}
\end{align}
for all $v\in H^{-\frac12}(\Sigma)$,
where ${p}_{j,\lambda}^{\prime\sharp\ast} \in
S^{1}_{1,\delta}(\R^{n-1}\times \R^{n-1})$, and
$\mathcal{R} (\lambda ) $ and ${\mathcal R}'(\lambda ) $ are bounded from
$H^{-\frac12}(\Sigma)$ to $ H^{-\frac12}(\Sigma)$. 
\end{proof}

Operators $P^\lambda _{\gamma ,\beta }$ can be defined for any trace
operator $\beta $, as $P^\lambda _{\gamma ,\beta }=\beta K^\lambda
_\gamma $. For $\beta =\gamma _1$, one has that $P^\lambda _{\gamma
  ,\gamma _1 }=\gamma _1K^\lambda _\gamma $ is {\it elliptic} (the principal symbol is
invertible), cf.\
\cite{A62}, \cite{Grubb71} 
(it is also documented in Ch.\ 11 of \cite{G09}, Exercise 11.7ff.).
Note that
\[P^\lambda _{\gamma ,\chi  }=s_0 P^\lambda _{\gamma ,\gamma _1
}+{\cal A}_1,
\]
where ${\cal A}_1$ is the first-order differential operator on $\Sigma
$ explained in  Theorem \ref{thm:Green}. The ellipticity of $P^\lambda _{\gamma ,\chi }$ depends on ${\cal
  A}_1$, which is defined from the coefficients in the
divergence form \eqref{eq:DefnA}. It should be
recalled that the choice of coefficients in \eqref{eq:DefnA} is not
unique for a given operator $A$, see the discussion in \cite{Grubb71}; it
is shown there in the smooth case that the choice of coefficients in
\eqref{eq:DefnA} can be adapted to give {\it any desired first-order
differential operator on} $\Sigma $ in the place of ${\cal A}_1$.
Ellipticity of $P^\lambda _{\gamma ,\chi  }$ holds if and only if
the system $\{A,\chi \}
$ is elliptic.

In the papers \cite{BGW09}, \cite{G08},
$\chi $ is replaced by $\nu _1=s_0\gamma _1$ in order to have an
elliptic Dirichlet-to-Neumann operator (then Green's formula looks slightly different). However,
the modified trace operator  $\Gamma ^\lambda $ introduced further
below is actually
independent of the choice of ${\cal A}_1$ (cf.\ Remark \ref{rem:2.9b}).

\setcounter{section}{5}
\section{Boundary value problems}\label{bdrycond}

We shall now apply the abstract results from Section
\ref{extend} to boundary value problems. The realizations
$\Ama$, $\Ami$ and $A_\gamma $ of $A$ in $H=L_2(\Omega )$ introduced
in the beginning of Section \ref{resolv} have the properties 
described in Section
\ref{extend}, with the realizations $\Ami'$, $\Ama'$ and $A_\gamma '$
of $A'$ representing the adjoints. Then the general constructions of
Section \ref{extend} can be applied. The operators $\wA\in {\cal M}$
are the general realizations of $A$. Note that
\begin{equation}
Z^0_0(A)=Z,\quad Z^0_0(A')=Z',\quad Z^0_\lambda (A )=Z_\lambda ,\quad
Z^0_{\bar\lambda }(A' )
=Z'_{\bar\lambda }.\label{tag2.6a} 
\end{equation}

For an interpretation of the correspondence between $\wA$ and $T\colon V\to W$, 
we need a modified version of Green's formula (as introduced
originally in \cite{Grubb68}). Here we use the 
trace operators (in the sense of the $\psi $dbo calculus) defined by:
\begin{equation}
\Gamma ^\lambda =\chi-P^\lambda _{\gamma ,\chi }\gamma _0 ,\quad \Gamma^{\prime\bar\lambda }
=\chi ' -P^{\prime \bar\lambda }_{\gamma ,\chi ' }\gamma _0\label{tag2.9}
\end{equation}
for $\lambda \in \varrho (A_\gamma )$; we call them the {\it reduced
  Neumann trace operators}.

\begin{theorem}\label{modGreen}
The trace operators $\Gamma ^\lambda $ and $\Gamma ^{\prime\bar\lambda }$  
map $D(\Ama)$ resp.\ $D(\Ama')$ continuously onto
$ H^\frac12(\Sigma )$, and satisfy
\begin{align}
\Gamma ^0  &=\chi A_\gamma ^{-1}\Ama ,\quad \Gamma^{\prime 0 }
=\chi '(A^*_\gamma ) ^{-1}\Ama' ,\label{tag2.9a}
\\
\Gamma ^\lambda  &=\chi (A_\gamma -\lambda )^{-1}(\Ama -\lambda ),\quad \Gamma^{\prime\bar\lambda }
=\chi '(A^*_\gamma -\bar\lambda ) ^{-1}(\Ama' -\bar\lambda ).\label{tag2.9b}
\end{align}
In particular, $\Gamma ^0$ vanishes on $Z$, etc. 

With these trace operators there is a modified Green's formula valid {\bf for all} 
$u\in D(\Ama)$, $v\in D(\Ama')$:
\begin{equation}
(Au,v)_{\Omega }-(u,A'v)_{\Omega }=(\Gamma ^0 u,\gamma _0v)_{\frac12,
-\frac12}-(\gamma _0u, \Gamma ^{\prime0 }v)_{-\frac12,\frac12};\label{tag2.10}
\end{equation}
in particular,
\begin{equation}
(Au,w)_{\Omega }=(\Gamma ^0 u,\gamma _0w)_{\frac12,
-\frac12},\text{ when }w\in Z'.\label{tag2.10a}
\end{equation}
Similarly, for all $u\in D(\Ama)$, $v\in D(\Ama')$,
\begin{equation}
((A-\lambda )u,v)_{\Omega }-(u,(A'-\bar\lambda )v)_{\Omega }
=(\Gamma ^\lambda  u,\gamma _0v)_{\frac12,
-\frac12}-(\gamma _0u, \Gamma ^{\prime\bar\lambda }v)_{-\frac12,\frac12},\label{tag2.10b}
\end{equation}
which is also equal to $(Au,v)_{\Omega }-(u,A'v)_{\Omega }$.
\end{theorem} 

\begin{proof} With the preparations we have made, this goes exactly as
  in the smooth case \cite{Grubb68}; we give some details for the convenience of the reader. Take $\lambda
  =0$. Writing $u=u_\gamma +u_\zeta $, where $u_\gamma =A_\gamma
  ^{-1}\Ama u$ and $u_\zeta \in Z$, we have that $\gamma _0u=\gamma
_0  u_\zeta $, and $u_\zeta =K^0_\gamma  \gamma _0u$. Then 
\[
\Gamma ^0u=\chi u-P^0_{\gamma ,\chi }\gamma _0u=\chi u-P^0_{\gamma
  ,\chi }\gamma _0u_\zeta =\chi u-\chi u_\zeta =\chi  u_\gamma =\chi
A_\gamma ^{-1}\Ama u.
\]
This shows \eqref{tag2.9a}, and since $A_\gamma ^{-1}\Ama $ is
surjective from $D(\Ama)$ to $D(A_\gamma )$, and $\chi $ is surjective
from $D(A_\gamma )=H^2(\Omega )\cap H^1_0(\Omega )$ to
$H^\frac12(\Sigma )$, we get the surjectiveness from $D(\Ama)$ to
$H^\frac12(\Sigma )$. (The continuity is with respect to the
graph-norm on $D(\Ama)$.)

Let $u\in D(\Ama)$, $v\in D(\Ama')$. Then, using the fact that
$(Au_\gamma ,v_{\gamma '})-(u_\gamma ,A'v_{\gamma '})=0$, and using the
extended Green's formula \eqref{tag9.29}, we find
\begin{align*}
&(Au,v)-(u,A'v)=(Au_\gamma ,v_{\gamma '}+v_{\zeta '})-(u_\gamma
+u_\zeta ,A'v_{\gamma '})\\
&=(Au_\gamma ,v_{\zeta '})-(u_\zeta ,A'v_{\gamma '})\\
&=(Au_\gamma ,v_{\zeta '})-(u_\gamma , A'v_{\zeta '})+(Au_\zeta ,v_{\gamma '})-(u_\zeta ,A'v_{\gamma '})\\
&=(\chi u_\gamma ,\gamma _0v_{\zeta '})_{\frac12,-\frac12}-(\gamma _0u_\gamma ,\chi
'v_{\zeta '})_{\frac32,-\frac32}+(\chi u_\zeta  ,\gamma _0v_{\gamma '})_{-\frac32,\frac32}-(\gamma _0u_\zeta  ,\chi
'v_{\gamma  '})_{-\frac12,\frac12}\\
&=(\chi u_\gamma ,\gamma _0v)_{\frac12,-\frac12}-(\gamma _0u  ,\chi
'v_{\gamma  '})_{-\frac12,\frac12}
=(\Gamma ^0  u,\gamma _0v)_{\frac12,
-\frac12}-(\gamma _0u, \Gamma ^{\prime 0}v)_{-\frac12,\frac12}.
\end{align*}
This shows \eqref{tag2.10}; \eqref{tag2.10a} is a special case. The
proof for general $\lambda $ is similar.
\end{proof}

\begin{remark}\label{rem:2.9b}
Note that when $\chi =s_0\gamma _1+{\cal A}_1\gamma _0 $, then 
\[
\Gamma ^\lambda =\chi -P^\lambda _{\gamma ,\chi }\gamma _0= s_0\gamma _1+{\cal
  A}_1\gamma _0-( P^\lambda _{\gamma , s_0 \gamma _1 } -{\cal
  A}_1)\gamma _0=s_0\gamma _1-P^\lambda _{\gamma ,s_0\gamma _1};
\]
so in fact $\Gamma ^\lambda $ is independent of 
the choice of coefficient of $\gamma _0$ in $\chi $ (cf. \eqref{eq:conormal}).
\end{remark}

By composition with suitable isometries (order-reducing operators),
 \eqref{tag2.10} can be
turned into a formula with $L_2$-scalar products over the boundary,
but since this leads to more overloaded formulas, we
shall not pursue that line of thought here.

We denote by $\gamma _{Z_\lambda }$ the restriction of $\gamma _0$ to a mapping from
$Z_\lambda$ (closed subspace of $L_2(\Omega )$) to 
$H^{-\frac12}(\Sigma ) $; its adjoint $\gamma
_{Z_\lambda }^*$ goes from $H^\frac12(\Sigma )$ to $Z_\lambda $:
\begin{equation}
\gamma _{Z_\lambda }\colon  Z_\lambda \simto H^{-\frac12}(\Sigma ) ,\text{
with adjoint }\gamma _{Z_\lambda }^*\colon    H^\frac12 (\Sigma )\simto
Z_\lambda .\label{tag2.6c}
\end{equation}
The inverse $\gamma _{Z_\lambda }^{-1}$ gives by composition with
$\inj_{Z_\lambda }$ the Poisson operator $K^\lambda _\gamma $: 
\[ K^\lambda _\gamma =\inj_{Z_\lambda }\gamma _{Z_\lambda }^{-1}\colon 
H^{-\frac12}(\Sigma )\to L_2(\Omega ).
\]
There is a similar notation for the primed operators. When $\lambda
=0$, this index can be left out.

For the study of general realizations $\wA$ of $A$,
the homeomorphisms \eqref{tag2.6c} make it possible to translate the characterization in
terms of operators $T\colon V\to W$ in Section \ref{extend} into a characterization in
terms of operators $L$ over the boundary. 

First let $\lambda =0$. For $V\subset Z$, $W\subset Z'$, let $X=\gamma _0V$, $Y=\gamma _0W$, with the notation
for the restrictions of $\gamma _0$ to homeomorphisms:
\begin{equation}
\gamma _V\colon  V\simto X,\quad \gamma _W\colon  W\simto Y.\label{tag3.22} 
\end{equation}
The map $\gamma _V\colon V\simto X$ has the adjoint $\gamma _V^*\colon X^*\simto
V$.
Here $X^*$ denotes the antidual space of $X$, with a duality
coinciding with the scalar product in $L_2(\Sigma )$ when
applied to elements that come from 
$ L_2(\Sigma )$. The duality is written $(\psi ,\varphi
)_{X^*,X}$. We also write $\overline{(\psi ,\varphi )_{X^*,X}}=(\varphi ,\psi )_{X,X^*}$.
Similar conventions are applied to $Y$.

When $A$ is replaced by $A-\lambda  $ for $\lambda \in \varrho
(A_\gamma )$, we use a
similar notation where $Z$, $Z'$, $V$ and $W$ are replaced by $Z_\lambda $,
$Z'_{\bar\lambda }$, $V_\lambda $, $W_{\bar\lambda }$. Since $\gamma _0E^\lambda z=\gamma
_0z$ (cf.\ Section \ref{extend}), the mapping defined by $\gamma _0$ on $V_\lambda $ has {\it the same range space $X$  as
when $\lambda =0$}. Similarly,  the mapping defined by $\gamma _0$ on
$W_{\bar\lambda }$ has  the  range space $Y$ for all $\lambda $. So $\gamma _0$
defines homeomorphisms
\begin{equation}
\gamma _{V_\lambda }\colon V_\lambda
\simto X,\quad \gamma _{W_{\bar\lambda }}\colon W_{\bar\lambda }\simto Y, \label{tag3.23}
\end{equation}
 
For $\lambda \in\varrho (A_\gamma )$, we denote 
\begin{equation}K^\lambda _{\gamma
  ,X}= \inj_{V _\lambda }\gamma _{V_\lambda }^{-1}\colon X \to V_\lambda
\hookrightarrow  H,\quad K^{\prime\bar\lambda }_{\gamma
  ,Y}= \inj_{W _{\bar\lambda }}\gamma _{W_{\bar\lambda }}^{-1}\colon Y \to W_{\bar\lambda }\hookrightarrow H.\label{tag3.23a}
\end{equation}

Now a given $T\colon V\to W$ is carried over to a closed, densely 
defined operator $L\colon  X\to
Y^*$ by the definition
\begin{equation}
L=(\gamma _W^{-1})^*T\gamma _V^{-1},\quad D(L)=\gamma _V D(T);\label{tag3.24}
\end{equation}
it is expressed in the diagram
\begin{equation}
\CD
V     @>\sim>  \gamma _V   >    X\\
@VTVV           @VV  L  V\\
   W  @>\sim> (\gamma _W^{-1})^*>   Y^*\endCD 
 \label{tag3.25}
\end{equation}

When $\wA$ corresponds to $T\colon V\to W$ and $L\colon X\to Y^*$, we can write
\begin{equation}
(Tu_\zeta ,w)=(T\gamma _V^{-1}\gamma _0u,\gamma _W^{-1}\gamma
_0w)=(L\gamma _0u,\gamma _0w)_{Y^*,Y},\text{ all }u\in D(\wA), w\in W.\label{tag3.36}
\end{equation}
The formula $(Au)_W=Tu_\zeta $ in \eqref{tag1.2} is then in view of
\eqref{tag2.10a} turned into
\[
(\Gamma ^0u,\gamma _0w)_{\frac12, -\frac12}=(L\gamma _0u,\gamma _0w)_{Y^*,Y},\text{ all }w\in W,
\]
or, since $\gamma _0$ maps $W$ bijectively onto $Y$,
\begin{equation}
(\Gamma ^0u,\varphi )_{\frac12,
-\frac12}=(L\gamma _0u,\varphi  )_{Y^*,Y}\text{ for all }\varphi  \in Y.\label{tag3.37}
\end{equation}

To simplify the notation, note that the injection $\inj_Y\colon  Y\to H^{-\frac12}(\Sigma )$ has as adjoint the mapping
$\inj _{Y}^*\colon  H^{\frac12}(\Sigma )\to Y^*$ that sends a functional $\psi $
on $H^{-\frac12}(\Sigma )$ over into a functional $\inj_Y^*\psi $ on $Y$ by:
\begin{equation}
(\inj_Y^*\psi ,\varphi )_{Y^*,Y}=(\psi ,\varphi
)_{\frac12,-\frac12}\text{ for all }\varphi \in Y.\label{tex3.37a}
\end{equation}
With this notation (also indicated in \cite{Grubb74} after (5.23)), \eqref{tag3.37} may be rewritten as
$$
\inj_Y^*\Gamma ^0u = L\gamma _0u,
$$
or, when we use that $\Gamma ^0=\chi -P^0_{\gamma ,\chi }\gamma _0$,
\begin{equation}
\inj_Y^*\chi  u= (L+\inj_Y^*P^0_{\gamma ,\chi })\gamma _0u. \label{tag3.38}
\end{equation}

We have then obtained:
 
\begin{theorem}\label{Theorem3.1'} For a closed operator $\wA\in{\cal M}$, the
following statements {\rm (i)} and {\rm (ii)} are equivalent:

{\rm (i)} $\wA$ corresponds to $T\colon V\to W$ as in Section {\rm \ref{extend}}. 

{\rm (ii)} $D(\wA)$ consists of the functions  $u\in D(\Ama)$
that satisfy the boundary condition
\begin{equation}
\gamma _0u\in D(L),\quad 
\inj_Y^*\chi  u= (L+\inj_Y^*P^0_{\gamma ,\chi })\gamma _0u.
\label{tag3.39}
\end{equation}
 Here $T\colon V\to W$ and $L\colon X\to Y^*$ are defined from one another as
described in {\rm \eqref{tag3.22}--\eqref{tag3.25}}.

\end{theorem}

Note that when
$Y$ is the full space $H^{-\frac12}(\Sigma )$, $\inj_{Y^*}$ is
superfluous, and \eqref{tag3.39} takes the form
\begin{equation}
\gamma _0u\in D(L), \quad \chi u=(L+P^0_{\gamma ,\chi })\gamma _0u.\label{tag3.40}
\end{equation}

The whole construction can be carried out with $A$ replaced by $A-\lambda
$, when $\lambda \in \varrho (A_\gamma )$. 
We define $L^\lambda $ from $T^\lambda $ as
in \eqref{tag3.24}--\eqref{tag3.25} with $T\colon V\to W$ replaced by
$T^\lambda \colon V_\lambda \to W_{\bar\lambda }
$ and use of \eqref{tag3.23}; here $L^\lambda $ maps from $X$ to $Y^*$;
\begin{equation}
L^\lambda =(\gamma _{W_{\bar\lambda }}^{-1})^*T^\lambda \gamma _{V_\lambda }^{-1},\quad D(L^\lambda )=\gamma _{V_\lambda } D(T)=D(L).\label{tag3.41}
\end{equation}

This can be expressed in  the following diagram, where we also take
\eqref{1.2a} into account:
\[
\CD
V     @>\sim>{E^\lambda _{V}} >V _\lambda     @>\sim>  \gamma _{V_\lambda } 
 >   X\\
@V  T+G^\lambda _{V,W} VV @VT^\lambda VV           @VV  L^\lambda   V\\
 W @>\sim>(F^{\prime\bar\lambda }_{W})^* >  W_{\bar\lambda }
@>\sim>(\gamma _{W_{\bar\lambda }}^*)^{-1} >  Y^*\endCD
 \] 
Here the horizontal homeomorphisms compose as 
\begin{equation}
\gamma _{V_\lambda }E^\lambda _V=\gamma _V,\quad (\gamma
_{W_{\bar\lambda }}^*)^{-1}(F^{\prime\bar\lambda }_{W})^*=(\gamma
_{W}^*)^{-1},\label{tag3.41b}
\end{equation}
so
\begin{equation}   L^\lambda =(\gamma _{W}^*)^{-1}(T+G^\lambda _{V,W})\gamma _V^{-1}.
\label{tag3.41a}
\end{equation}

In this $\lambda $-dependent situation,  Theorem \ref{Theorem3.1'} takes the form:

\begin{theorem}\label{Theorem3.2} Let $\lambda \in\varrho (A_\gamma )$. For a closed
operator $\wA\in{\cal M}$, the 
following statements {\rm (i)} and {\rm (ii)} are equivalent:

{\rm (i)} $\wA-\lambda $ corresponds to $T^\lambda \colon V_\lambda \to W_{\bar\lambda }$ as in Section {\rm \ref{extend}}. 

{\rm (ii)} $D(\wA)$ consists of the functions  $u\in D(\Ama)$ such that
\begin{equation}
\gamma _0u\in D(L),\quad\inj_Y^*\chi  u= (L^\lambda +\inj_Y^*P^\lambda _{\gamma ,\chi
})\gamma _0u. \label{tag3.42}
\end{equation}

\end{theorem}

Observe that since the boundary conditions \eqref{tag3.39} and \eqref{tag3.42} define the same
realization, we obtain moreover the information that
\[
(L^\lambda +\inj_Y^*P^\lambda _{\gamma ,\chi
})\gamma _0u=(L +\inj_Y^*P^0 _{\gamma ,\chi
})\gamma _0u, \text{ for }\gamma _0u\in D(L)=D(L^\lambda ),
\]
i.e.,
\begin{equation}
L^\lambda =L+\inj_Y^*(P^0 _{\gamma ,\chi }-P^\lambda  _{\gamma ,\chi })\text{
on }D(L)=D(L^\lambda ).\label{tag3.43}
\end{equation}
Note also that in view of \eqref{tag3.24} and \eqref{tag3.41a},
\[ L^\lambda =L+
(\gamma _{W}^*)^{-1}G^\lambda _{V,W}\gamma _V^{-1}
\text{ on }D(L).
\]

This has the particular consequence:
\begin{equation}
\inj_Y^*(P^0 _{\gamma ,\chi }-P^\lambda  _{\gamma ,\chi })=
(\gamma _{W}^*)^{-1}G^\lambda _{V,W}\gamma _V^{-1}\text{
on }D(L).
\end{equation}
Since the last statement will hold for fixed choices of $V,W,X,Y$,
regardless of how the operator $L$ is chosen (it can e.g.\ be taken as
the zero operator), we conclude that 
\begin{equation}
\inj_Y^*(P^0 _{\gamma ,\chi }-P^\lambda  _{\gamma ,\chi })\inj_X=
(\gamma _{W}^*)^{-1}G^\lambda _{V,W}\gamma _V^{-1},\label{tag3.43a}
\end{equation}
 as bounded operators from $X$ to $Y^*$. In particular, in the case $X=Y=H^{-\frac12}(\Sigma )$:
\begin{equation}
P^0 _{\gamma ,\chi }-P^\lambda  _{\gamma ,\chi }=
(\gamma _{Z'}^*)^{-1}G^\lambda _{Z,Z'}\gamma _Z^{-1};\label{tag3.43b}
\end{equation}
bounded operators from $H^{-\frac12}(\Sigma )$ to $H^{\frac12}(\Sigma )$.

We can now connect the description with $M$-functions and establish
Kre\u\i{}n-type resolvent formulas. (The following formulation differs
slightly
from that in \cite{BGW09} using
order-reducing operators carrying $H^{\pm \frac12}(\Sigma )$ over to
$L_2(\Sigma )$ and orthogonal projections.)

\begin{theorem}\label{Theorem2.1} 
 Let $\wA$ correspond to $T\colon V\to W$,
carried over to $L\colon X\to Y^*$, whereby $\wA$ represents
the boundary condition {\rm \eqref{tag3.39}}, as well as {\rm
  \eqref{tag3.42}} when $\lambda \in \varrho (A_\gamma )$.

{\rm (i)} For $\lambda \in \varrho (A_\gamma )$, $L$ and $L^\lambda $
satisfy {\rm \eqref{tag3.43}}, where $P^0 _{\gamma
,\chi }-P^\lambda _{\gamma ,\chi }\in \mathcal L(H^{-\frac12}(\Sigma
), H^{\frac12}(\Sigma
))$. The relations to $G^\lambda _{V,W}$ are as described in {\rm
  \eqref{tag3.43a}, \eqref{tag3.43b}}. 

{\rm (ii)} For $\lambda \in \varrho (\wA)$, there is a related
$M$-function $\in\mathcal L(Y^*,X)$:
 \[
M_L(\lambda ) =\gamma _0\bigl(I-(\wA-\lambda )^{-1}(\Ama-\lambda
)\bigr)A_\gamma  ^{-1}\inj_W\gamma _W^*.
\]

{\rm (iii)} For $\lambda \in \varrho (\wA)\cap \varrho (A_\gamma )$,
\[
M_L(\lambda )=-(L+\inj_Y^*(P^0 _{\gamma ,\chi }-P^\lambda _{\gamma
,\chi })\inj_X)^{-1}=-(L^\lambda )^{-1}.
\]

{\rm (iv)} For $\lambda \in \varrho (A_\gamma )$ (recall {\rm \eqref{tag3.23a}}),
\begin{align}
\operatorname{ker}(\wA-\lambda )&=K^\lambda _{\gamma , X}
\operatorname{ker}L^\lambda,\nonumber\\
\operatorname{ran}(\wA-\lambda )&=\gamma _{W_{\bar\lambda }}^*
\operatorname{ran}L^\lambda+H\ominus W_{\bar\lambda }.
 \label{tag2.17}
\end{align}

{\rm (v)} For $\lambda \in \varrho (\wA)\cap \varrho (A_\gamma )$
there is a
Kre\u\i{}n-type  resolvent formula:
\begin{align} (\wA-\lambda )^{-1}&=(A_\gamma -\lambda )^{-1}-\inj_{V_\lambda }\gamma _{V_\lambda }^{-1} 
M_L(\lambda )(\gamma _{W_{\bar\lambda }}^*)^{-1}\pr_{W_{\bar\lambda }}\nonumber\\
&=(A_\gamma -\lambda )^{-1}-K^\lambda _{\gamma ,X} 
M_L(\lambda )(K^{\prime\bar\lambda }_{\gamma ,Y})^*
\label{tag2.19}
\\
&=(A_\gamma -\lambda )^{-1}+K^\lambda _{\gamma ,X} 
(L^\lambda )^{-1}(K^{\prime\bar\lambda }_{\gamma ,Y})^*
.\nonumber\end{align}

\end{theorem}

\begin{proof}
Statement (i) was accounted for before the theorem. 

In (ii), the $M$-function $M_L(\lambda )$ is obtained from
$M_{\wA}(\lambda )$ in Theorem \ref{Theorem1.5} (i) by composition to the
right with $\gamma _W^*$ and to the left with $\gamma _V$:
\begin{align}
M_L(\lambda )&=
\gamma _VM_{\wA}(\lambda )\gamma _W^*=
\gamma _V\pr_\zeta (I-(\wA-\lambda )^{-1}(\Ama-\lambda ))A_\gamma
^{-1}\inj_{W }\gamma _W^*\label{tag2.19a}\\
&=
\gamma _0(I-(\wA-\lambda )^{-1}(\Ama-\lambda ))A_\gamma
^{-1}\inj_{W }\gamma _W^*.\nonumber
\end{align}

Statement (iii) follows  from  Theorem \ref{Theorem1.5} (ii) in a similar way.

For (iv), we use the homeomorphism properties of $\gamma _{V_\lambda
}$ and $\gamma _{W_{\bar\lambda }}$ and their adjoints.

For (v), we calculate the last term in  \eqref{tag1.5a}, using \eqref{tag2.19a}:
\begin{align*}
-\inj_{V_\lambda  }E^\lambda _VM_{\wA}(\lambda
)(E^{\prime\bar\lambda }_W)^*\pr_{W_{\bar\lambda }}&=
-\inj_{V_\lambda  }E^\lambda _V\gamma _V^{-1}M_{L}(\lambda
)(\gamma _W^*)^{-1}(E^{\prime\bar\lambda }_W)^*\pr_{W_{\bar\lambda
  }}\\
&=-\inj_{V_\lambda  }\gamma _{V_\lambda }^{-1}M_{L}(\lambda
)(\gamma _{W_{\bar\lambda }}^{-1})^*\pr_{W_{\bar\lambda
  }}\\
&=-K^\lambda _{\gamma ,X} 
M_L(\lambda )(K^{\prime\bar\lambda }_{\gamma ,Y})^*;
\end{align*}
cf.\ \eqref{tag3.41b} and \eqref{tag3.23a}. 
\end{proof}

Hereby, Kre\u\i{}n-type resolvent formulas are established for all
closed realizations of $A$ in the present nonsmooth case.

In the cases where $X$ and $Y$ differ from $H^{-\frac12}(\Sigma )$,
the formulas are quite different from those established in \cite{GM10}
for selfadjoint realizations of the Laplacian, where an $M$-function
acting between full boundary Sobolev spaces is used (if the domain is $C^{\frac32+\varepsilon }$).

\section{Neumann-type conditions}\label{neumann}

The case where
$X=Y=H^{-\frac12}(\Sigma )$, i.e., $V=Z$, $W=Z'$, is particularly
interesting for applications of the theory. Here the boundary
condition has the form in \eqref{tag3.40} and we say that
$\wA$ represents a
{\it Neumann-type condition} (this includes the information that $D(L)$ is a dense
subset of $H^{-\frac12}(\Sigma )$). It may be written 
\begin{equation}
\chi  u=C\gamma _0u,\text{ with }C=L+P^0_{\gamma ,\chi },\quad
D(C)=D(L).
\label{tag2.11b} 
\end{equation}
An interesting case is where $C$ acts like a differential operator (or
pseudodifferential operator) of order 1. 
In the differential operator case we can assume, to match the smoothness
properties of $s_0$ and ${\cal A}_1$ in 
Green's formula, that
\begin{equation}
C=c\cdot D_\tau +c_0,\text{ where }c=(c_1,\dots,c_n), \;c_j\in
H^{\frac12}_p(\Sigma )\text{ for } j=0,1,\dots,n.\label{tag2.11c}
\end{equation}
As a $\psi $do $C$ we can take an operator constructed from local
first-order pieces as in \eqref{eq:DefnPt} with $H^{\frac12}_qS^1_{1,0}$-
symbols. 

Then $L$ acts as a
pseudodifferential operator of order 1,
\[ L=C-P^0_{\gamma ,\chi };
\]
cf.\ Theorem \ref{Theorem4.3} for the properties of $P^0_{\gamma ,\chi }$.

It should be noted that $L$ is determined in a precise way from $\wA$ as 
an operator from $D(L)\subset
H^{-\frac12}(\Sigma )$ to $H^\frac12(\Sigma )$; it is generally {\it unbounded}
from $H^{-\frac12}(\Sigma )$ to $H^\frac12(\Sigma )$ since it is of
order 1.

In the study of boundary conditions, the situation is sometimes set up
in a slightly different way:

It is $C$ that is given as a first-order operator, and we define
$\wA$ as the restriction of $\Ama$ with domain
\begin{equation}
D(\wA)=\{u\in D(\Ama)\mid \chi u=C\gamma _0u\}.\label{tag2.13}
\end{equation}
Note that for the $H^2(\Omega )$-functions satisfying  $\chi u=C\gamma
_0u$, the Dirichlet data $\gamma _0u$ fill out the space
$H^{\frac32}(\Sigma )$, since $\{\gamma _0,\chi \}$ is surjective from
$H^2(\Omega )$ to $H^{\frac32}(\Sigma )\times H^{\frac12}(\Sigma )$.
When $u\in D(\wA)$,
\[ 
L\gamma_0u=\Gamma ^0u=\chi u-P^0_{\gamma ,\chi }\gamma _0u=(C-P^0_{\gamma ,\chi })\gamma _0u,
\]
so necessarily $\gamma _0u$ belongs to the subset of
$H^{-\frac12}(\Sigma )$ that is mapped by $C-P^0_{\gamma ,\chi }$ into
$H^{\frac12}(\Sigma )$. When $u$ lies there, it moreover has to satisfy
$\Gamma ^0u =(C-P^0_{\gamma ,\chi })\gamma _0u$, in order to belong to
$D(\wA)$. This shows:

\begin{lemma}\label{CtoL} Let $C$ be a first-order operator on
  $\Sigma $ as
  described above and define the realisation
  $\wA$ 
by {\rm    \eqref{tag2.13}}. Then $L$ is the operator acting like $C-P^0_{\gamma ,\chi }$ with
  domain
\[
D(L)=\{\varphi \in H^{-\frac12}(\Sigma )\mid (C-P^0_{\gamma ,\chi
})\varphi \in H^{\frac12}(\Sigma )\}.
\]
$D(L)$ contains $H^{\frac32}(\Sigma )$, hence is dense in
$H^{-\frac12}(\Sigma )$.  
\end{lemma}

In the $\lambda $-dependent setting, 
\[ L^\lambda \text{ acts like }C-P^\lambda _{\gamma ,\chi }\text{ with }D(L^\lambda )=D(L).
\]

Further information can be obtained in the {\it elliptic} case.
This is the case where the model operator, defined from the
principal symbols at each $(x',\xi ')$ in local coordinates, is invertible:
\[ \begin{pmatrix} \underline a^0(x',0,\xi ',D_n)\\ \quad\\ \underline
  \chi ^0(x',\xi ',D_n)-\underline c^0(x',\xi
')\gamma _0\end{pmatrix} \colon  H^2({\R}_+)\simto \begin{matrix} L_2({\R}_+)\\
\times \\ {\C}\end{matrix},
\]
for all $x'$, all $|\xi '|\ge 1$. Using the various reductions
introduced above on this one-dimensional level, we find that $L$ has
the principal symbol
\[ \underline l^0(x',\xi ')=
\underline
c^0(x',\xi ')-\underline p^0(x',\xi ').
\]
where $\underline c^0(x',\xi ')$ and $\underline p^0(x',\xi ')$ are
the principal symbols of $C$ and 
$P^0_{\gamma
,\chi }$; moreover, ellipticity holds 
if and only if $\underline l^0(x',\xi ')\ne 0$ for $|\xi
'|\ge 1$.
These considerations take place pointwise in $x'$ regardless of
smoothness with respect to $x'$.

\begin{theorem}\label{regularity1} For a given first-order 
pseudodifferential operator  $C$ as described above, 
let $\wA$ be defined by 
  {\rm \eqref{tag2.13}}. Assume that $C-P^0 _{\gamma ,\chi }$ is
  elliptic. Then $D(L)=H^{\frac32}(\Sigma)$, and $D(\wA)\subset H^2(\Omega ) $.
\end{theorem}

\begin{proof} Let $\varphi \in D(L)$; then we know to begin with that
  $\varphi \in H^{-\frac12 }(\Sigma)$ and \linebreak$(C-P^0 _{\gamma ,
    \chi })\varphi \in H ^{\frac12}(\Sigma)$. It follows from Theorem
  \ref{Theorem4.3} $3^\circ$ that $(C-P^0 _{\gamma ,
    \chi })\varphi =(C-P^{\sharp01}_{\gamma , \chi })\varphi +\psi $,
  where $\psi \in
  H^{-\frac32+\varepsilon}(\Sigma)$. This together with $(C-P^0 _{\gamma ,
    \chi })\varphi \in H ^{\frac12}(\Sigma)$ implies $(C-P^{\sharp01}_{\gamma , \chi  })
\varphi \in H^{-\frac32+\varepsilon}(\Sigma)$. Here
$C-P^{\sharp01}_{\gamma , \chi  }$ is defined as in
\eqref{eq:DirNeumannSharp}, constructed from localized
pieces with elliptic smooth $\psi $do symbols, and it follows by use
of a parametrix in each localization that $\varphi \in H^{-\frac12+\varepsilon
}(\Sigma )$. (Details on cutoffs and partitions of unity in 
parametrix constructions can
e.g.\ be found in \cite{G09}, Sect.\ 8.2.)

Next, let $s'=-\frac32+\varepsilon $. By Theorem
  \ref{Theorem4.3} $2^\circ$ there is an $\varepsilon '>0$ such that 
$P^0_{\gamma ,\chi }-P^{\sharp 0}_{\gamma ,\chi }$ is continuous from 
$H^{s-\frac12-\varepsilon '}(\Sigma)$ to $H^{s-\frac32}(\Sigma)$ for
all $s\in [s',2]$. In a similar way as in the preceding construction,
one finds by use of a parametrix of $C-P^{\sharp 0}_{\gamma ,\chi }$ that 
$(C-P^0 _{\gamma ,
    \chi })\varphi \in H ^{\frac12}(\Sigma)$ together with 
$\varphi \in H^{s-\frac12-\varepsilon '}(\Sigma)$ imply $\varphi \in
H^{s-\frac12}(\Sigma)$, when $s\in [s',2]$. Starting from
$s=s'-\varepsilon '$ and applying the argument successively with
$s=s'+k\varepsilon '$, $k=-1,0,1,2,\dots$, we reach the conclusion
$\varphi \in H^{\frac32}(\Sigma )$ in finitely many steps.

Since we know from Lemma \ref{CtoL} that $D(L)\supset  H^{\frac32}(\Sigma
)$, this shows that  $D(L)=  H^{\frac32}(\Sigma )$. Now any
$u\in D(\wA )$ satisfies
$\gamma _0u\in H^{\frac32}(\Sigma )$, so it follows
from our knowledge of the Dirichlet problem that $u\in H^2(\Omega )$.
\end{proof}

We also have:

\begin{theorem}\label{regularity} 
For a given first-order differential or 
pseudodifferential operator  $C$ as described around {\rm \eqref{tag2.11c}}, 
let $\wA$ be defined by 
  {\rm \eqref{tag2.13}}.
If there is a 
  $\lambda \in \varrho (\wA)\cap \varrho (A_\gamma )$ such that $(\wA-\lambda )^{-1}$ is continuous from
  $L_2(\Omega )$ to $H^2(\Omega )$, then $D(L)=H^{\frac32}(\Sigma )$.

In this case, 
there is a Kre\u\i{}n resolvent formula for $\lambda \in \varrho (\wA)\cap \varrho (A_\gamma )$:
\begin{equation} (\wA-\lambda )^{-1}=(A_\gamma -\lambda )^{-1}+K^\lambda _{\gamma } 
(L^\lambda )^{-1}(K^{\prime\bar\lambda }_{\gamma })^*
,\label{tag7.4}\end{equation}
where the operators $L^\lambda $ and $K^\lambda _{\gamma }$ are a $\psi $do and a Poisson operator, respectively, belonging to the
nonsmooth calculus, acting on $H^{\frac32}(\Sigma )$
(the case $s=2$ in Theorems {\rm \ref{Theorem4.2}} and {\rm \ref{Theorem4.3}}). \end{theorem}

\begin{proof} According to Theorem \ref{Theorem2.1} (ii), $M_L(\lambda )$ has the
  form 
 \[
M_L(\lambda ) =\gamma _0\bigl(I-(\wA-\lambda )^{-1}(\Ama-\lambda
)\bigr)A_\gamma  ^{-1}\inj_{Z'}\gamma _{Z'}^*.
\]
Here  the mapping property of $(\wA-\lambda )^{-1}$ assures that
$I-(\wA-\lambda )^{-1}(\Ama-\lambda )$ preserves $H^2(\Omega )$, which
implies that $M_L(\lambda )$ maps $H^{\frac12}(\Sigma )$ continuously into
  $H^{\frac32}(\Sigma )$. Moreover, by (iii), $-M_L(\lambda )$ is the inverse of
  $L^\lambda $, whose domain contains $H^{\frac32}(\Sigma )$ by Lemma 
\ref{CtoL}. Then the
  domain must equal $H^{\frac32}(\Sigma )$. Since $D(L)=D(L^\lambda
  )$, it follows that also $D(L)=H^{\frac32}(\Sigma )$.

The next statement follows from the definition of $L$, and the last
statement is a consequence of the fact that $D(L)=H^{\frac32}(\Sigma )$.
\end{proof}

This theorem includes general Neumann-type boundary conditions with
$C$ of order 1 in the
discussion, where earlier treatments such as \cite{PR09} and
\cite{GM08} had conditions of compactness relative to order 1 or
lower order than 1 in the picture (Robin conditions). \cite{GM09} has
a somewhat more general class of nonlocal operators $C$, also used in
\cite{GM10}, for selfadjoint realizations of $A=-\Delta $.

As a sufficient condition for the validity of the assumptions in Theorem
\ref{regularity} we can mention parameter-ellipticity of the system
$\{A-\lambda , \chi -C\gamma _0\}$ on a ray in $\C$ in the sense of
\cite{FunctionalCalculus}, when $C$ is a differential operator  {\rm \eqref{tag2.11c}}. Here one can
construct the resolvent in an exact way for large $\lambda $ on the
ray, using Agmon's trick in this situation in the same way as in the resolvent
construction for $A_\gamma $ we described above; this is also
accounted for at the end of
\cite{G08}. (It is used that $C-P^\lambda _{\gamma , \chi }$ has
``regularity $\nu=+\infty$'', cf.\ Remark \ref{rem:reginfty}.) 
We can denote $\wA=A_{\chi -C\gamma_0}$ in these cases. The case $C=0$
(the oblique Neumann problem) satisfies the
hypothesis for rays $re^{i\eta }$ with $\eta \in (\pi /2-\delta ', 3\pi
/2+\delta ')$, some $\delta '>0$, when the sesquilinear form $a(u,v)$ satisfies 
\begin{equation}
\operatorname{Re}a(u,u)\ge c_1\|u\|^2_1-k\|u\|_0^2, \quad u\in
H^1(\Omega ),
\label{eq:coercive}
\end{equation}
 with $c_1>0$. It defines the realization $A_\chi $.

If $A$ is symmetric, $A_\gamma $ is selfadjoint (and then $\wA$ will be
selfadjoint if and only if $X=Y$ and $L\colon X\to X^*$ is selfadjoint, cf.\
\cite{Grubb68}). If, moreover, $a(u,u)$ is real for $u\in H^1(\Omega )$
and satisfies \eqref{eq:coercive}, $A_\chi $ will be selfadjoint with
domain in $H^2(\Omega )$.

The preceding choices of $L$ are the most natural ones in connection with applications
to physical problems. One can of course more generally choose $L$
abstractly to be
of a convenient form and derive $C$ from it as in \eqref{tag2.11b}.

Besides the Kre\u\i{}n-type resolvent formulas, there are many
other applications of the characterization of realizations in terms of
the operators $L$. Let us mention numerical range estimates and lower
bounds, and coerciveness estimates, as e.g.\ in \cite{Grubb71},
and spectral
asymptotics estimates as e.g.\ in \cite{Grubb74}, for the smooth
case. Both papers are followed
up in the recent literature with further developments.
Spectral asymptotic formulas (Weyl-type spectral estimates) have been
worked out  in Grubb \cite{G12} for the boundary term in the Kre\u\i{}n formula 
\eqref{tag7.4}, and show the expected appearance of the boundary
dimension $n-1$.

\appendix
\section[Pseudodifferential boundary value problems]{Pseudodifferential boundary value problems with nonsmooth coefficients}

\subsection{Definitions, symbol smoothing}\label{sec:symbolsmoothing}

\begin{defn}\label{defn:nonsmoothPsDO}
  Let $X$ be a Banach space and let $X^\tau=C^\tau$ or $X^\tau=\mathcal{C}^\tau$. The symbol space $X^\tau S^m_{1,\delta}(\Rn\times\Rn;X) $, $\tau>0$,
  $\delta\in [0,1]$, $m\in\R$, is the set of all functions $p\colon\Rn\times \Rn\to X$ that are smooth with respect to $\xi$ and are in $X^\tau$ with respect to $x$ satisfying the estimates
  \begin{eqnarray*}
\|D_{\xi}^{\alpha}D_x^\beta p(.,\xi) \|_{L^\infty(\Rn;X)} &\leq&
    C_{\alpha,\beta} \weight{\xi}^{m-|\alpha|+\delta|\beta|},\\ 
    \|D_{\xi}^{\alpha}p(.,\xi) \|_{X^\tau(\Rn;X)} &\leq& C_{\alpha} \weight{\xi}^{m-|\alpha|+\delta\tau},   
  \end{eqnarray*}
and if $X=C^\tau$ additionally
  \begin{eqnarray*}
    \|D_{\xi}^{\alpha}p(.,\xi) \|_{C^j(\Rn;X)} &\leq& C_{\alpha} \weight{\xi}^{m-|\alpha|+\delta j} \quad \text{for all}\ j\in\N_0,\,j\leq  [\tau],
  \end{eqnarray*}
for all $\alpha\in\N_0^n$ and $|\beta|\leq [\tau]$. 
  \end{defn}

Obviously, $\bigcap_{\tau >0} C^\tau S^m_{1,\delta}(\Rn\times\Rn;X)$ coincides with the usual
H\"ormander class $S^m_{1,\delta}(\Rn\times\Rn;X)$ in the vector-valued variant.

In particular, we are interested in the case $\delta=0$, where we
simply say that the symbols (and operators) have $X^\tau $-smoothness in $x$. But we need
the classes $C^\tau S^m_{1,\delta}$ with $\delta >0$ when working
with the technique called \emph{symbol
  smoothing}: If $p\in C^\tau S^m_{1,\delta}(\Rn\times\Rn;X)$,
$\delta\in [0,1)$,
then for every $\gamma\in (\delta,1)$ there  is a decomposition $p=p^\#+p^b$
  with
  \begin{eqnarray}\label{eq:SymbolSmoothing}
    p^\# \in S^m_{1,\gamma}(\Rn\times\Rn;X),\qquad p^b\in C^\tau S^{m-(\gamma-\delta)\tau}_{1,\gamma}(\Rn\times\Rn;X), 
  \end{eqnarray}
  cf. \cite[Equation (1.3.21)]{TaylorNonlinearPDE}.
  We  note that the proofs in
  \cite{TaylorNonlinearPDE} are formulated for scalar symbols only; 
but they still
  hold in the $X$-valued setting since they are based on direct
  estimates. 

In the case where $X=\mathcal{L}(X_0,X_1)$ is the space of all bounded linear
operators $A\colon X_0\to X_1$ for
some Banach spaces $X_0$ and $X_1$, we define the pseudodifferential
operator with symbol $p\in C^\tau S^m_{1,0}(\Rn\times\Rn; 
\mathcal{L}(X_0,X_1))$ by 
\begin{equation}
p(x,D_x) u= \OP(p)u = \int_{\Rn} e^{ix\cdot\xi} p(x,\xi) \hat{u}(\xi)\dd \xi \quad \text{for\ } u\in\SD(\Rn;X_0),\label{eq:x-form} 
\end{equation}
where $\dd\xi:= (2\pi)^{-n}d\xi$. --- An operator defined from a
symbol $p(x,\xi )$ by formula \eqref{eq:x-form} is said to be ``in $x$-form'' in
contrast to more general formulas, e.g.\ 
 where $\hat u(\xi )=\int e^{-iy\cdot\xi
}u(y)\,dy$ is inserted, and $p$ is allowed to
depend also on $y$. Compositions often lead to more general formulas.

\begin{prop}\label{prop:HsBddness2}
  Let $1<q<\infty$ and let $p\in C^rS^m_{1,\delta}(\Rn\times \Rn
  ;\mathcal{L}(H_0,H_1))$, $m\in\R$, $\delta \in [0,1]$, $r>0$, where
  $H_0$ and $H_1$ are Hilbert spaces. Then $p(x,D_x)$ is continuous 
\[
p(x,D_x)\colon H^{s+m}_q(\Rn;H_0) \to H^{s}_q(\Rn;H_1)
\]
for all $s\in\R$ with $-r(1-\delta) <s < r$.
\end{prop}
\begin{proof}
The proposition is an operator-valued variant of
\cite[Proposition 2.1.D]{TaylorNonlinearPDE}. As indicated in
\cite[Appendix]{HInftyInLayer} the proof given in  \cite{TaylorNonlinearPDE}
directly carries over to the present setting by using the
Mihlin multiplier theorem in the $\mathcal{L}(H_0,H_1)$-valued
version, where it is essential
that $H_0$ and $H_1$ are Hilbert spaces. 
\end{proof} 

 We denote by ${\cal S}(\crp)$ the
space of restrictions to $\crp$ of functions in ${\cal S}(\R)$.

\begin{defn}\label{defn:Symbolkernels}
  The space $C^\tau S^d_{1,\delta}(\R^N\times\R^{n-1},\SD(\ol{\R}_+))$,
  $d\in\R$, $n,N\in\N$,
  consists of all functions $\tilde{f}(x,\xi',y_n)$, which are smooth in
  $(\xi',y_n)\in \R^{n-1}\times \ol{\R}_+$, are in $C^\tau(\R^N)$ with
  respect to $x$, and 
  satisfy
  \begin{eqnarray}
    \label{eq:PoissonSymbolEstim}
\sup_{x'\in\R^{N}}    \|y_n^l\partial_{y_n}^{l'}D_{\xi'}^{\alpha} \tilde{f}(x',\xi',.)\|_{L^2_{y_n}(\R_+)}
    &\leq& C_{\alpha,l,l'}\weight{\xi'}^{d+\frac12-l+l'-|\alpha|} \\
    \|y_n^l\partial_{y_n}^{l'}D_{\xi'}^{\alpha} \tilde{f}(.,\xi',.)\|_{C^\tau (\R^N;L^2_{y_n}(\R_+))}
    &\leq& C_{\alpha,l,l'}\weight{\xi'}^{d+\frac12-l+l'-|\alpha|+|\delta|\tau} 
  \end{eqnarray} 
  for all $\alpha\in\N_0^{n-1}$, $ l,l'\in\N_0$.
  Moreover, we define $S^m_{1,\delta}(\R^{n-1}\times \R^{n-1};\mathcal{S}(\ol{\R}_+)):= \cap_{k\in\N} C^kS^m_{1,\delta}(\R^{n-1}\times \R^{n-1};\mathcal{S}(\ol{\R}_+))$.
\end{defn}

Now we define an $S^m_{1,\delta}$-variant of the Poisson operators with nonsmooth coefficients as studied in \cite{NonsmoothGreen}. 
\begin{defn}\label{defn:PoissonOperator}
Let $\tilde{k}=\tilde{k}(x',\xi',y_n)\in C^\tau
 S^{d-1}_{1,\delta}(\R^{n-1}\times\R^{n-1},\SD(\ol{\R}_+))$, $d\in\R$,
 $0\leq \delta <1$, $\tau>0$.
Then we define the \emph{Poisson operator} of order $d$ by
\begin{equation*}
  k(x',D_x)v= \F^{-1}_{\xi'\mapsto x'}\left[\tilde{k}(x',\xi',x_n) \hat{v}(\xi')\right], \qquad v\in \SD(\R^{n-1}),
\end{equation*}
where the Fourier transform is applied in the $x'$-variables. 
$\tilde k$ is called a Poisson symbol-kernel of order $d$.
The associated boundary symbol operator $k(x',\xi',D_n)\in \mathcal{L}(\C,\mathcal{S}(\overline{\R}_+))$ is defined by
$$
(k(x',\xi',D_n)v)(x_n)= \tilde{k}(x',\xi',x_n)v\qquad \text{for all}\ x_n\geq 0, v\in\C.
$$
\end{defn}
\begin{thm}\label{thm:PoissonBddness}
  Let $\tilde{k}=\tilde{k}(x',\xi',y_n)\in C^\tau
 S^{d-1}_{1,\delta}(\R^{n-1}\times\R^{n-1},\SD(\ol{\R}_+))$, $d\in\R$, $0\leq \delta< 1$. Then $k(x',D_x)$ extends to a bounded operator
 \begin{equation}\label{eq:PoissonBddness}
   k(x',D_x)\colon H^{s+d-\frac12}(\R^{n-1})\to H^s(\Rn_+)
 \end{equation}
 for every $-(1-\delta)\tau<s<\tau$. In particular, {\rm
   (\ref{eq:PoissonBddness})} holds for every $s\in\R$ if $\tilde k\in S^{d-1}_{1,\delta}(\R^{n-1}\times\R^{n-1},\SD(\ol{\R}_+))$. 
\end{thm}
\begin{proof}
  From the symbol estimates one easily derives that
  \begin{equation*}
    k(x',\xi',D_n)\in C^\tau S^{d-\frac12+s}_{1,\delta}(\R^{n-1}\times \R^{n-1};\mathcal{L}(\C,H^s(\Rn_+)))
  \end{equation*}
for every $s\geq 0$, cf.\ \cite[Proof of Lemma 4.5]{NonsmoothGreen} for details. Hence Proposition~\ref{prop:HsBddness2} implies that
\begin{equation*}
  k(x',D_x)\colon H^{s+d-\frac12}(\R^{n-1})\to H^s(\R^{n-1};L^2(\R_+))
\end{equation*}
is a bounded operator for every $-(1-\delta)\tau<s<\tau$. If $s\leq 0$, then $H^s(\R^{n-1};L^2(\R_+))\hookrightarrow H^{s}(\Rn_+)$ and the theorem is proved. In the case $s\geq 0$,  Proposition~\ref{prop:HsBddness2} additionally implies
that
\begin{equation*}
  k(x',D_x)\colon H^{s+d-\frac12}(\R^{n-1})\to L^2(\R^{n-1};H^s(\R_+))
\end{equation*}
is a bounded operator for every $s\geq 0$. Since 
\begin{equation*}
  H^s(\Rn_+)= H^s(\R^{n-1};L^2(\R_+))\cap L^2(\R^{n-1};H^s(\R_+))\quad \text{if}\ s\geq 0,
\end{equation*}
the theorem is also obtained in this case. 
\end{proof}

\begin{rem}
  One can easily modify the arguments in the proof of \cite[Th.\~4.8]{NonsmoothGreen} to even get that
 \begin{equation*}
   k(x',D_x)\colon B^{s+d-\frac12}_{p,p}(\R^{n-1})\to H^s_p(\Rn_+)
 \end{equation*}
 is bounded for every $-(1-\delta)\tau<s<\tau$ and $1<p<\infty$.
\end{rem}

  We note that $\tilde{f}\in C^\tau S^d_{1,\delta}(\R^{n-1}\times \R^{n-1};\mathcal{S}(\ol{\R}_+))$ if and only if 
  \begin{equation*}
    x_n^l \partial_{x_n}^{l'}f(x',\xi',D_n)\in C^\tau S^{d+\frac12-l+l'}_{1,\delta}(\R^{n-1}\times \R^{n-1};\mathcal{L}(L^2(\R_+)))\quad \text{for all}\ l,l'\in\N_0,
  \end{equation*} 
where
$(x_n^l \partial_{x_n}^{l'}f(x',\xi',D_n)v)(x_n)=x_n^l \partial_{x_n}^{l'}\tilde{f}(x',\xi',x_n)v$
for all $v\in\C$, $x_n\geq 0$.
  Hence, applying symbol smoothing with respect to $(x',\xi')$, we obtain that  $\tilde{f}= \tilde{f}^\sharp + \tilde{f}^b$, where
  \begin{equation*}
    \tilde{f}^\sharp\in S^m_{1,\delta}(\R^{n-1}\times
    \R^{n-1};\mathcal{S}(\ol{\R}_+)),\quad
 \tilde{f}^b\in  C^{\tau } S^{m-\tau \delta}_{1,\delta}(\R^{n-1}\times \R^{n-1};\mathcal{S}(\ol{\R}_+)),
  \end{equation*}
which can be proved the same way as e.g.\ \cite[Prop.\ 1.3.E]{TaylorNonlinearPDE}.
  As a consequence we derive directly from Theorem~\ref{thm:PoissonBddness}:
  \begin{lem}\label{lem:PoissonSmoothing}
  Let $\tilde{k}=\tilde{k}(x',\xi',y_n)\in C^\tau
 S^{d-1}_{1,0}(\R^{n-1}\times\R^{n-1},\SD(\ol{\R}_+))$, $d\in\R$, $0<
 \delta< 1$.
Then $k(x',D_x)=k^\sharp(x',D_x)+k^b(x',D_x)$, where 
  \begin{equation*}
    \tilde{k}^\sharp\in S^m_{1,\delta}(\R^{n-1}\times
    \R^{n-1};\mathcal{S}(\ol{\R}_+)),\quad
 \tilde{k}^b\in  C^\tau S^{m-\tau\delta}_{1,\delta}(\R^{n-1}\times \R^{n-1};\mathcal{S}(\ol{\R}_+)).
  \end{equation*}
  In particular, we have that
  \begin{equation*}
    k(x',D_x)-k^\sharp(x',D_x)\colon H^{s+d-\delta\tau-\frac12}(\R^{n-1})\to H^s(\Rn_+)
  \end{equation*}
  for every $-(1-\delta)\tau<s<\tau$, and $k^\sharp(x',D_x) \in \mathcal{L}(H^{s+d-\frac12}(\R^{n-1}), H^s(\Rn_+))$ for every $s\in\R$.
  \end{lem}

Let us recall the definition of trace operators from \cite{NonsmoothGreen}: A trace operator of order $m\in \R$ and  class $r\in\N_0$ with $C^\tau$-coefficients (in $x$-form) is defined as
  \begin{eqnarray*}
       t(x',D_x) f &=& \sum_{j=0}^{r-1} s_j(x',D_{x'}) \gamma_j f + t_0(x',D_x) f\\ 
       t_0(x',D_x) f &=& \F^{-1}_{\xi'\mapsto x'}\left[\int_0^\infty
         \tilde{t}_0(x',\xi',y_n) \acute{f}(\xi',y_n) dy_n \right],
     \end{eqnarray*} 
     where $\tilde{t}_0\in C^\tau
     S^{m}_{1,0}(\R^{n-1}\times\R^{n-1},\SD(\ol{\R}_+))$, $s_j\in C^\tau
     S_{1,0}^{m-j}(\R^{n-1}\times \R^{n-1})$, $j=0,\ldots, r-1$,
     $\acute{f}(\xi',x_n)= \F_{x'\mapsto \xi'}[f(.,x_n)]$ (a partial
     Fourier transform) and 
     $\gamma_j f = \partial_n^j f|_{x_n=0}$. The associated
     \emph{boundary symbol operator $t(x',\xi',D_n)$} is defined by
     applying the above definition to $f\in\SD(\overline{\R}_+)$ for
     every fixed $x',\xi'\in\R^{n-1}$. Since $\gamma _j$ is
     well-defined on $H^k(\rp)$ ($k\in \N_0$) if and only if $k>j$,
     the boundary symbol operators of class $r$ are those that are
     well-defined on $H^r(\rp)$ for all $(x',\xi ')$.

In particular, $t(x',D_x)$ is called a \emph{differential trace operator}
of order $m$ and class $r$ with $C^\tau$-coefficients if $t_0\equiv 0$
and the $s_j(x',\xi')$ are polynomials in $\xi'$.

For the precise definitions of singular Green and Green operators of
order $m$ and class $r$ with $C^\tau$-coefficients (in $x$-form) as
well as the definition of the (global) transmission condition for
$p\in C^\tau S^m_{1,0}(\Rn\times \Rn)$, $m\in\Z$, we refer to
\cite{NonsmoothGreen}. The precise definitions are not important for
our considerations in the present paper. 

\subsection{A parametrix result}

\begin{theorem}\label{Theorem3.1}

$1^\circ$ Let 
\[  \mathcal A=\begin{pmatrix} P_++G \\T\end{pmatrix},
\]
where $P_+$ is the truncation to $\rnp$ of a zero-order $\psi $do satisfying the transmission condition at $x_n=0$, $G$ is a
zero-order singular Green operator, such that $P_++G$ is of class $r\in{\mathbb Z}$, and $T$ is a trace
operator of order $-\frac12$ and class $r$, all in $x$-form with
$C^\tau $-smoothness in $x$. Then $\mathcal A$ maps continuously
\begin{equation}
\mathcal A=\begin{pmatrix}P_++G \\T\end{pmatrix} \colon  
H^{s}(\rnp) \to
\begin{matrix}
H^{s}(\rnp)\\ \times \\ H^{s}({\mathbb R}^{n-1})\end{matrix} ,
\label{tag3.2}\end{equation}
when 

{\rm (i)} $|s|<\tau $, 

{\rm (ii)} $s>r-\tfrac12$ (class restriction).
\smallskip

$2^\circ$ Let $\mathcal A$ be as in $1^\circ$, and polyhomogeneous and
uniformly elliptic with principal symbol $a^0$. Then the operator $\mathcal B^0$ with symbol $(a^0)^{-1}$,
\[  {\mathcal B}^0=\begin{pmatrix} R^0& K^0\end{pmatrix},
\]
with $R^0$ of order $0$ and class $r$ (being the sum of a truncated
$\psi $do and a singular Green operator), $K^0$ a Poisson operator of
order $\frac12$, all in $x$-form with
$C^\tau $-smoothness in $x$, satisfies that ${\mathcal B}^0$ is 
continuous in the opposite direction of $\mathcal A$, and  
${\mathcal R}={\mathcal
A}{\mathcal B}^0-I$ is continuous:
\begin{equation}
{\mathcal R}\colon  \begin{matrix}
H^{s-\theta }(\rnp)\\ \times \\ H^{s-\theta }({\mathbb
R}^{n-1})\end{matrix} \to
\begin{matrix}
H^{s}(\rnp)\\ \times \\ H^{s}({\mathbb
  R}^{n-1})\end{matrix} ,
\label{tag3.4}
\end{equation}
when

{\rm (i)} $-\tau +\theta <s<\tau $;

{\rm (ii)}
 $ s-\theta > r-\tfrac12$ (class restriction).

\end{theorem}

\smallskip

\begin{proof}
The first part of the theorem follows from \cite[Th.\ 1.1 and 1.2]{NonsmoothGreen}. The second part of the theorem essentially
follows from \cite[Th.\ 6.4]{NonsmoothGreen} with the only
difference that there is an additional restriction
$|s-\frac12|<\tau$. This comes from the fact that for the parametrix
construction there, the trace operator is reduced to order $m=0$ (the same
order as the order of $P_+$ and $G$). But when we take the trace operator
to be of order $-\frac12$, the proof of  \cite[Th.\
6.4]{NonsmoothGreen} applies to the present situation and the
restriction $|s-\frac12|<\tau$ is not needed.   
\end{proof}

 ${\mathcal B}^0$ is called a parametrix of $\mathcal A$.

 A more
general version than the above is
quoted in \cite[Th.\ 6.3]{G08}.

\begin{remark}\label{Remark3.2} To make the above theorem useful for systems where the elements have
other orders we need the so-called 
``order-reducing
operators''. There are two types, one acting over the domain and one acting
over the boundary: 
\begin{align}
&\Lambda _{-,+}^r=\operatorname{OP}(\lambda ^r_-(\xi ))_+\colon  H^t(\rnp)\simto H^{t-r}(\rnp),
\label{tag3.5}\\
&\Lambda _0^s=\operatorname{OP}'(\ang{\xi '}^s)\colon H^t({\mathbb R}^{n-1})\simto H^{t-s}({\mathbb
R}^{n-1}),\text{ all $t\in {\mathbb
R}$,}\nonumber
\end{align}
$r\in\mathbb Z$ and $s\in\R$, with inverses 
$\Lambda
_{-,+}^{-r}$ resp.\ $\Lambda _0^{-s}$. Here  $\lambda _-^r$ is the
``minus-symbol'' defined in
\cite[Prop.\ 4.2]{G90}  as a refinement of $(\ang{\xi '}-i\xi _n)^{r}$. 
Composition of an operator in $x$-form with an order-reducing operator
to the right gives an operator in $x$-form (since the order-reducing
operator acts on the symbol level essentially as a multiplication by
an $x$-independent symbol). Composition with the order-reducing
operator to the left 
gives a more complicated
expression when applied to an $x$-form operator.

It should be noted that when e.g.\ $S=s(x',D_{x'})$ is a $\psi$do of order $m$ on
$\R^{n-1}$ with $C^\tau $-smoothness, then 
it maps  $H^{s+m}(\R^{n-1})\to H^{s}(\R^{n-1})$ for $|s|< \tau
$, so by \eqref{tag3.5},
\begin{equation}\Lambda _0^r S\colon H^{s+m}(\R^{n-1})\to H^{s-r}(\R^{n-1})
\text{ for }-\tau <s<\tau .\label{tag3.6}\end{equation}
The
composition rule  Theorem \ref{thm:PsDOCompo} (in a version for
$C^\tau $-smooth symbols)  shows that
$\Lambda _0^rS$ can be written as the sum of an operator in the
calculus $\operatorname{OP}'(\ang{\xi '}^r s(x',\xi '))$ in $x$-form and a
remainder, such that the sum maps  $H^{s'+m+r}(\rnp)\to H^{s'}(\rnp)$ for $-\tau
<s'<\tau $; this gives a mapping property like in \eqref{tag3.6}
but with a shifted interval $-\tau +r<s<\tau +r$. This extends the
applicability, but one has to keep in mind that the new decomposition
produces different operators;
$\Lambda
_0^rS$ is not in $x$-form but is an operator in $x$-form composed to the left
with $\Lambda _0^r$, not equal to
$\operatorname{OP}'(\ang{\xi '}^r s(x',\xi '))$.

The operators $\Lambda _{-,+}^r$ allow an extension of the class
concept for trace operators to negative values: When $T=T_0\Lambda
_{-,+}^{-k}$, where $T_0$ is of class 0 and $k\in \N_0$,
$T$ is said to be of class $-k$. Then the boundary symbol operator is
well-defined on $H^{-k}(\rp)$. There is a similar concept for
operators $P_++G$.
More details on operators of negative class can be found in
\cite[Section~5.4]{NonsmoothGreen}, \cite{G90}, or \cite{FunctionalCalculus}.
With this extension, Theorem \ref{Theorem3.1} is valid for $r\in \Z$.
\end{remark}

\footnotesize

\def\cprime{$'$}

\end{document}